%% file: main_jmlr.tex
\documentclass[twoside,11pt,a4paper]{article}

\input{preamble_jmlr}

\input{commands_jmlr}

\input{theorems_jmlr}

\ShortHeadings
  {Boundary-layer asymptotics for Gaussian-smoothed singular measures}
  {Brosse and Dalalyan}

\firstpageno{1}

\title{Boundary-layer asymptotics for Gaussian-smoothed singular measures}

\author{%
\name Nicolas Brosse \email nicolas.brosse@ensae.fr \\
\addr CREST, ENSAE Paris, Institut Polytechnique de Paris \\
\AND
\name Arnak Dalalyan \email arnak.dalalyan@ensae.fr \\
\addr CREST, ENSAE Paris, Institut Polytechnique de Paris and
Mohamed bin Zayed University of Artificial Intelligence
}

\editor{Action Editor}

\begin{document}

\maketitle

\begin{abstract}
We study the small-noise asymptotics of Euclidean heat regularizations of probability measures
supported on manifolds with corners. Near a boundary or corner stratum, the relevant regime is a
conical boundary layer in which the observation point approaches the stratum at the same scale as
the Gaussian smoothing parameter. After rescaling this layer, the support is replaced to leading
order by its inward tangent cone. We prove a two-term expansion for the heat-regularized density in
this regime. The leading coefficient is the Gaussian mass of the linearized cone, weighted by the
density on the support and by the adapted corner Jacobian; the first correction records the
variation of the density, the Jacobian, and the quadratic geometry of the embedding. A localization
argument then yields the corresponding expansion for the full heat regularization, with the nonlocal
contribution exponentially small. From this density expansion we derive logarithmic asymptotics and
uniform expansions for the score, the log-Hessian, and the scale derivative of the score. These
formulas show how lower-dimensional support, boundary faces, corners, and curvature are encoded in
the singular differential structure of small-noise Gaussian regularizations.
\end{abstract}

\begin{keywords}
Gaussian smoothing, singular measures, manifolds with corners, tangent cones, score functions, heat
regularization
\end{keywords}


\renewcommand\thepart{}
\renewcommand\partname{}

\doparttoc
\faketableofcontents

\part{}

\input{sections/introduction}
\input{sections/preliminaries}
\input{sections/formal_main}
\input{sections/part_cases}
\input{sections/proof_sketch}
\input{sections/conclusion}


\acks{This project has received funding from the European Research Council (ERC) under the European
Union's Horizon Europe research and innovation programme (grant agreement No.~101201229).}

\bibliography{bibliography}


\newpage
\appendix

\renewcommand\contentsname{}

\part{Appendix}
\parttoc

\newpage

\input{sections/appendices}

\end{document}

%% file: preamble_jmlr.tex
\usepackage[preprint]{jmlr2e}

\renewenvironment{proof}[1][Proof]{\par\noindent{\bfseries #1\ }}{\hfill\BlackBox\\[2mm]}

\usepackage[T1]{fontenc}
\usepackage[english]{babel}
\usepackage{lmodern}
\usepackage{enumitem}

\usepackage{amsmath}
\usepackage{mathrsfs}
\usepackage{dsfont}
\usepackage{nicefrac}

\usepackage{cleveref}

\usepackage{xcolor}
\usepackage{graphicx}
\usepackage{afterpage}

\usepackage{tikz}
\usetikzlibrary{
  arrows.meta,
  calc,
  decorations.pathreplacing,
  shadings
}

\definecolor{BLMcolor}{RGB}{230,238,247}
\definecolor{BLLayercolor}{RGB}{210,229,255}
\definecolor{BLScolor}{RGB}{38,88,160}
\definecolor{BLTcolor}{RGB}{35,120,80}
\definecolor{BLCcolor}{RGB}{190,100,35}
\definecolor{BLNcolor}{RGB}{130,70,160}
\definecolor{BLPointcolor}{RGB}{25,25,25}
\definecolor{BLObscolor}{RGB}{190,45,45}

\tikzset{
  bl-lab/.style={font=\footnotesize},
  bl-smalllab/.style={font=\scriptsize},
  bl-manifold/.style={draw=black!55, fill=BLMcolor, line width=0.55pt},
  bl-layer/.style={draw=none, fill=BLLayercolor, opacity=0.72},
  bl-stratum/.style={draw=BLScolor, line width=1.25pt},
  bl-vec/.style={-{Latex[length=2.1mm,width=1.4mm]}, line width=0.75pt},
  bl-dash/.style={densely dashed, line width=0.45pt},
  bl-dot/.style={circle, fill=BLPointcolor, inner sep=1.55pt},
  bl-ydot/.style={circle, fill=BLObscolor, inner sep=1.75pt},
  bl-panel-title/.style={font=\footnotesize\bfseries},
  bl-whitebox/.style={fill=white, fill opacity=0.86, text opacity=1, inner sep=1.5pt}
}

\usepackage{booktabs}
\usepackage{tabularx}
\usepackage{longtable}

\usepackage{tocloft}
\usepackage[toc,page]{appendix}
\usepackage{minitoc}


%% file: commands_jmlr.tex
\newcommand{\dd}{\mathrm{d}}

\newcommand{\Diff}{\mathrm{D}}
\newcommand{\tr}{\operatorname{tr}}

\newcommand{\dist}{\operatorname{dist}}

\newcommand{\vol}{\mathrm{vol}}

\newcommand{\bfA}{\mathbf{A}}
\newcommand{\bfB}{\mathbf{B}}
\newcommand{\bfC}{\mathbf{C}}
\newcommand{\bfH}{\mathbf{H}}
\newcommand{\bfI}{\mathbf{I}}
\newcommand{\bfL}{\mathbf{L}}
\newcommand{\bfN}{\mathbf{N}}
\newcommand{\bfP}{\mathbf{P}}
\newcommand{\bfQ}{\mathbf{Q}}

\newcommand{\bfS}{\mathbf{S}}

\newcommand{\score}{\boldsymbol{s}_\sigma}
\newcommand{\scorezero}{\boldsymbol{s}_0}
\newcommand{\Hess}{\mathbf{H}_\sigma}
\newcommand{\Hesszero}{\mathbf{H}_0}
\newcommand{\Vel}{\dot{\boldsymbol{s}}_\sigma}

\newcommand{\bsh}{\boldsymbol{h}}

\newcommand{\Man}{\mathcal{M}}
\newcommand{\Strat}{\mathcal{S}}
\newcommand{\Compct}{\mathcal{K}}

\newcommand*{\standingtag}{%
  \ifcase\value{enumi}\or M\or Ch\or D\or P\fi}

\newcommand{\AssMeas}{\Cref{ass:standing}~\ref{ass:standing:measure}}
\newcommand{\AssChart}{\Cref{ass:standing}~\ref{ass:standing:chart}}
\newcommand{\AssDens}{\Cref{ass:standing}~\ref{ass:standing:density}}

%% file: theorems_jmlr.tex
\usepackage{aliascnt}

\newaliascnt{assumption}{theorem}
\newtheorem{assumption}[assumption]{Assumption}
\aliascntresetthe{assumption}

\crefname{assumption}{assumption}{assumptions}
\Crefname{assumption}{Assumption}{Assumptions}

\makeatletter

\@ifundefined{theproposition}{}{}
\@ifundefined{c@proposition}{}{\let\c@proposition\relax}
\newaliascnt{proposition}{theorem}
\newtheorem{proposition}[proposition]{Proposition}
\aliascntresetthe{proposition}

\@ifundefined{theremark}{}{}
\@ifundefined{c@remark}{}{\let\c@remark\relax}
\newaliascnt{remark}{theorem}
\newtheorem{remark}[remark]{Remark}
\aliascntresetthe{remark}

\@ifundefined{thelemma}{}{}
\@ifundefined{c@lemma}{}{\let\c@lemma\relax}
\newaliascnt{lemma}{theorem}
\newtheorem{lemma}[lemma]{Lemma}
\aliascntresetthe{lemma}

\@ifundefined{thedefinition}{}{}
\@ifundefined{c@definition}{}{\let\c@definition\relax}
\newaliascnt{definition}{theorem}
\newtheorem{definition}[definition]{Definition}
\aliascntresetthe{definition}

\@ifundefined{thecorollary}{}{}
\@ifundefined{c@corollary}{}{\let\c@corollary\relax}
\newaliascnt{corollary}{theorem}
\newtheorem{corollary}[corollary]{Corollary}
\aliascntresetthe{corollary}

\makeatother

\makeatletter
\newenvironment{jmlrunnumbered}[2][]{%
  \par\vskip\baselineskip\noindent
  {\bfseries #2\if\relax\detokenize{#1}\relax\else\ (#1)\fi.}%
  \ \itshape
}{%
  \par\vskip\baselineskip
}
\makeatother

\newenvironment{theorem*}[1][]{%
  \begin{jmlrunnumbered}[#1]{Theorem}%
}{%
  \end{jmlrunnumbered}%
}

\newenvironment{proposition*}[1][]{%
  \begin{jmlrunnumbered}[#1]{Proposition}%
}{%
  \end{jmlrunnumbered}%
}

\newenvironment{lemma*}[1][]{%
  \begin{jmlrunnumbered}[#1]{Lemma}%
}{%
  \end{jmlrunnumbered}%
}

\newenvironment{corollary*}[1][]{%
  \begin{jmlrunnumbered}[#1]{Corollary}%
}{%
  \end{jmlrunnumbered}%
}

\newenvironment{definition*}[1][]{%
  \begin{jmlrunnumbered}[#1]{Definition}%
}{%
  \end{jmlrunnumbered}%
}

\newenvironment{assumption*}[1][]{%
  \begin{jmlrunnumbered}[#1]{Assumption}%
}{%
  \end{jmlrunnumbered}%
}

%% file: sections/introduction.tex

\section{Introduction}
\label{sec:introduction}

Let \(q\) be a Borel probability measure on \(\mathbb R^d\). For \(t>0\), its Euclidean heat
regularization is the density \( q_t(y) = \int_{ \mathbb R^d} {(4\pi t)^{-d/2}} \exp\!\big(-
{\|y-x\|^2}/{4t}\big) q(\dd x)\), \(y \in\mathbb R^d\). Equivalently, setting \(\sigma =\sqrt{2t}\),
we write the same regularization in the Gaussian scale parameter as
\begin{equation*}
  p_\sigma(y)=q_{\sigma^2/2}(y)= (\phi_\sigma\star q)(y)
  =\int_{\mathbb R^d} \phi_\sigma(y-x)\,q(\dd x),
\end{equation*}
where \(\phi_\sigma\) is the centered Gaussian density with covariance \(\sigma^2 \bfI_d\). For
every \(\sigma>0\), the density \(p_\sigma\) is smooth and strictly positive, even when \(q\) is
singular with respect to the Lebesgue measure. This paper aims to study how the local geometry of
\(q\) is encoded in the small-noise asymptotics of \(p_\sigma\), \(\log p_\sigma\), \(\nabla \log
p_\sigma\), \(\nabla^2 \log p_\sigma\), as well as in the scale derivative of the score,
\(\partial_\sigma\nabla \log p_\sigma\). The paper focuses on measures \(q\) supported on a manifold
\(\Man\subset \mathbb R^d\) that may be lower-dimensional and may have boundary or corners, as
illustrated in \Cref{fig:manifold-examples}.

To be more precise, we study the small-\(\sigma\) asymptotics of \(p_\sigma(y_\sigma)\), \(\log
p_\sigma(y_\sigma)\), \(\nabla_y\log p_\sigma(y_\sigma)\), \(\nabla_y^2\log p_\sigma(y_\sigma)\),
and \((\partial_\sigma\nabla_y\log p_\sigma)(y_\sigma)\), where the observation point \(y_\sigma\)
may itself depend on \(\sigma\). A useful way to keep the main scaling in mind is to write
\(y_\sigma=x+\sigma b\), with \(b\) bounded: the observation point then approaches \(x\) at the same
spatial scale as the Gaussian smoothing.
The expansions obtained below are locally uniform in the base point \(x\) and in the rescaled
displacement \(b\).

The geometry of the support determines the form of these expansions. If \(\Man\) is full-dimensional
and \(x\) lies in its smooth interior, as in the ball example in \Cref{fig:manifold-examples}, then
the usual regular small-noise expansion is recovered. By contrast, near a lower-dimensional support,
a boundary, or a corner, the score \(\nabla_y\log p_\sigma(y)\) and its derivatives with respect to
\(y\) and \(\sigma\) may contain negative powers of \(\sigma\). The aim of the paper is to identify
these singular terms and express their leading coefficients in terms of the local geometry of the
support.

\subsection{Motivation and guiding context}
\label{subsec:motivation-and-guiding-context}

The central theme of this paper is that the small-noise behavior of Gaussian regularizations is
governed by local geometry. When the support of \(q\) is smooth and full-dimensional, classical
heat-kernel arguments lead to regular asymptotic expansions. In contrast, near lower-dimensional,
boundary, or corner structure, singular scales appear. As we shall show, these singular effects are
controlled by a tangent-cone model obtained by zooming in on the support at scale \(\sigma\) near
the relevant stratum. This geometric principle governs the leading asymptotics of the density, its
logarithm, the score and its derivatives.

\begin{figure}[t]
  \centering
  \input{figures/manifold_examples}
  \vspace*{-20pt}
  \caption{\small Examples of manifolds embedded in \(\mathbb R^3\). The closed ball is
  full-dimensional, with \(m=d=3\), and has a smooth boundary. The warped disk and warped square are
  two-dimensional supports, with \(m=2\) and \(d=3\). The disk has boundary but no corners, while
  the square has both boundary and corners.}
  \label{fig:manifold-examples}
\end{figure}
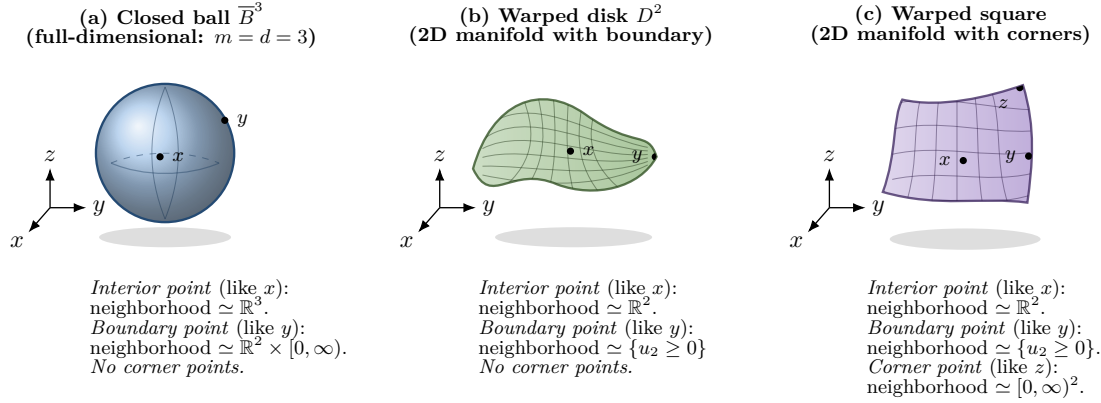

\paragraph{Heat equation perspective.}

For every \(\sigma>0\), the Gaussian-scale regularization \(p_\sigma=q_{\sigma^2/2}\) is smooth and
strictly positive. Hence its logarithm is smooth, and we can define the score, the log-Hessian and
the scale-derivative of the score pointwise by
\begin{equation*}
  \score(y)=\nabla_y\log p_\sigma(y), \qquad \Hess(y)=\nabla_y^2\log p_\sigma(y), \qquad \Vel(y)
  = \partial_\sigma \score(y).
\end{equation*}
Exact heat identities tie together these quantities. Recall that \(\partial_t q_t=\Delta q_t\),
which implies that \(\partial_t\log q_t=\Delta\log q_t+\|\nabla\log q_t\|^2\). In the noise scale
\(\sigma\), since \(p_\sigma=q_{\sigma^2/2}\), this leads to \(\partial_\sigma
p_\sigma(y)=\sigma\Delta p_\sigma(y)\) and
\begin{equation*}
  \partial_\sigma\log p_\sigma(y) = \sigma\Delta\log p_\sigma(y)
  + \sigma\|\nabla\log p_\sigma(y)\|^2 = \sigma\bigl(\tr \Hess (y)+\|\score(y)\|^2\bigr),
  \qquad y\in\mathbb R^d .
\end{equation*}
Differentiating the last identity in \(y\), we obtain
\begin{equation*}
  \Vel(y) = \sigma\bigl( \nabla(\tr \Hess (y)) + 2\Hess (y)\score(y) \bigr),
  \qquad y\in\mathbb R^d .
\end{equation*}
Thus the scale derivative of the score is not an auxiliary object introduced for technical reasons.
It is the scale-variation of the same logarithmic field, and it is coupled exactly to the score, the
log-Hessian, and the next spatial derivative of the log-density.

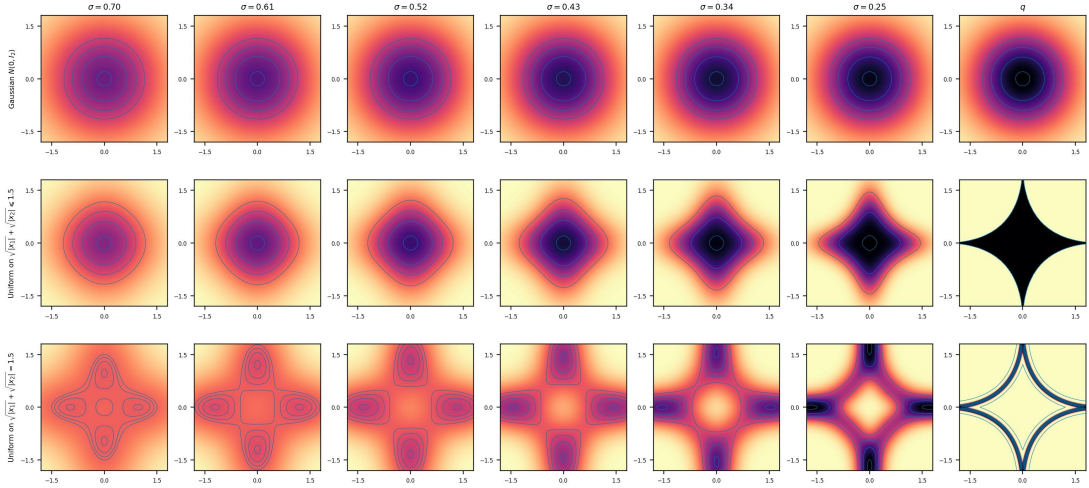
\begin{figure*}[t]
  \centering
  \input{figures/heat_density}
  \vspace*{-10pt}

  \caption{\small Heat regularizations of three probability distributions in \(\mathbb R^2\). The
  columns show \(p_{\sigma}=\phi_{\sigma}\star q\) for \(\sigma\in\{0.70-0.09k;k=0,1,\ldots,5\}\),
  followed by a visualization of the measure \(q\). First row: standard Gaussian. Second row:
  uniform distribution on \(\{\sqrt{|x_1|}+\sqrt{|x_2|}\leqslant 1.5\}\). Third row: uniform
  arclength measure on \(\{\sqrt{|x_1|}+\sqrt{|x_2|}=1.5\}\).}

  \label{fig:heat-density}
\end{figure*}

The singular powers of \(\sigma\) in this paper should be read against these identities. If \(q\)
has a smooth positive density in \(\mathbb R^d\), then \(\score\), \(\Hess \), and \(\Vel \) remain
regular as \(\sigma\downarrow 0\). By contrast, when the mass of \(q\) is concentrated on a
lower-dimensional set, or when the observation point is close to a boundary or corner stratum, the
regularized density changes across spatial layers of thickness \(O(\sigma)\). Derivatives transverse
to such layers may therefore produce inverse powers of \(\sigma\). In particular, the score may be
of order \(\sigma^{-1}\), the log-Hessian of order \(\sigma^{-2}\), and the fixed-\(y\) scale
derivative of the score of order \(\sigma^{-2}\). The purpose of the tangent-cone expansion is to
identify these singular terms geometrically. \Cref{fig:heat-density} illustrates this boundary-layer
phenomenon: as \(\sigma\downarrow0\), spatial variations of \(p_\sigma\) become increasingly
pronounced near the boundary of the support of \(q\).

\paragraph{Denoising interpretation.}

Let \(X\) and \(Z\) be independent random vectors such that
\begin{equation*}
  X\sim q, \qquad Z\sim\mathcal N(0,\bfI_d), \qquad Y_\sigma=X+\sigma Z.
\end{equation*}
Then \(Y_\sigma\), the noisy signal, has density \(p_\sigma\), and the exact Tweedie identities
\citep{robbins1956empirical,efron2011tweedie,vincent2011connection} lead
to
\begin{equation*}
  \score(y) = \frac{\mathbb E[X\,|\,Y_\sigma=y]-y}{\sigma^2}, \qquad \Hess (y)
  = \frac{\operatorname{Cov}(X\,|\,Y_\sigma=y)}{\sigma^4} - \frac1{\sigma^2}\bfI_d .
\end{equation*}
These identities imply that
\begin{equation*}
  \mathbb E[X\,|\,Y_\sigma=y] = y+\sigma^2\score(y), \qquad \Diff_y\mathbb E[X\,|\,Y_\sigma=y] = \bfI_d+\sigma^2\Hess (y).
\end{equation*}
Thus the score determines the denoising displacement, while the log-Hessian determines the local
linearization of the denoiser.

\paragraph{Transport interpretation.}

The same objects also arise from the canonical transport representation of the heat path. In heat
time, write \(\widetilde{\boldsymbol{s}}_t(y)=\nabla_y\log q_t(y)\) and
\(\widetilde\bfH_t(y)=\nabla_y^2\log q_t(y)\). Since \(q_t\widetilde{\boldsymbol{s}}_t=\nabla q_t\),
the heat equation may be written as the continuity equation
\begin{equation}\label{eq:intro-heat-velocity}
  \partial_t q_t+\nabla\cdot(\widetilde{\boldsymbol{v}} _t q_t)=0,
  \qquad \widetilde{\boldsymbol{v}} _t=-\widetilde{\boldsymbol{s}}_t .
\end{equation}
Hence \(\nabla_y \widetilde{\boldsymbol{v}}_t=-\widetilde\bfH_t\) and \(\partial_t
\widetilde{\boldsymbol{v}}_t =-\partial_t\widetilde{\boldsymbol{s}}_t\). If the same path is
parameterized by the noise level \(\sigma\), then
\begin{equation}\label{eq:intro-sigma-velocity}
  \partial_\sigma p_\sigma + \nabla\cdot(\boldsymbol{v}_\sigma p_\sigma) = 0,
  \qquad \boldsymbol{v}_\sigma=-\sigma \score ,
\end{equation}
and therefore \(\nabla_y \boldsymbol{v}_\sigma= -\sigma \Hess \), and \(\partial_\sigma
\boldsymbol{v}_\sigma= - \score-\sigma\Vel\). Thus the score controls the canonical heat-path
velocity, the log-Hessian controls its spatial linearization, and the scale derivative of the score
controls its variation along the path.

\paragraph{Connection with generative modeling.}

Gaussian regularization also underlies several constructions in modern generative modeling. In
denoising score matching, diffusion models, and score-based generative modeling, one learns the
score of a noisy data distribution \citep{hyvarinen2005scorematching,vincent2011connection,
sohl2015nonequilibrium,ho2020ddpm,song2021scorebased, karras2022edm}. In consistency-type and
flow-matching methods, one learns maps or vector fields along paths of probability distributions
\citep{song2023consistency,chen2018neuralode,grathwohl2019ffjord, lipman2023flowmatching,tong2023conditional,
liu2023flowstraight,lipman2024flowmatchingguide, albergo2025stochasticinterpolants,tong2024sf2m}.
For the heat-regularized law considered here, the population denoising score-matching target at
noise level \(\sigma\) is exactly
\begin{equation*}
  \mathbb E\bigg[ \frac{X-Y_\sigma}{\sigma^2} \,\Big|\, Y_\sigma=y \bigg] = \score(y).
\end{equation*}
Likewise, for the heat path, the canonical velocity fields are precisely those in
\cref{eq:intro-heat-velocity,eq:intro-sigma-velocity}. Thus the score, the log-Hessian, and the
scale derivative of the score are not only natural analytic objects. They are also the ideal
population fields that denoising and transport-based methods learn, differentiate, approximate, or
use as velocities.

A related line of work studies score-based generative models under the manifold hypothesis, where
the data law is concentrated on a lower-dimensional set. There, convergence analyses and structural
results describe how the score behaves near the support, with the leading normal attraction
\(\score(y)\approx-(y-\pi(y))/\sigma^2\) as \(\sigma\downarrow0\)
\citep{debortoli2022manifold,pidstrigach2022manifolds,chen2023scoremanifold}. Relative to these
antecedents, the present contribution is a \emph{uniform two-term} expansion, valid down to the
support, that resolves the next-order corrections and, crucially, the modifications produced by
boundaries, corners, and curvature.

These connections motivate studying the local small-noise form of the population-level fields that
denoising and transport-based methods aim to learn, differentiate, or approximate. Near
lower-dimensional support, boundaries, or corners, these fields are not governed at leading order by
a smooth Euclidean density expansion. They are governed instead by the Gaussian mass of the inward
tangent cone and by the derivatives of its logarithm.

\subsection{Relation to classical heat-kernel asymptotics}
\label{subsec:introduction-heat-kernel-asymptotics}

The structure of our expansion is close in spirit to classical short-time heat-kernel asymptotics.
On smooth manifolds, short-time behavior is local and governed by geometry seen at scale \(\sqrt t\)
\citep{MinakshisundaramPleijel1949,McKeanSinger1967, Varadhan1967}. Near smooth boundaries, one sees
boundary-layer variables of the form \(r/\sqrt t\), where \(r\) is the distance to the boundary,
with the half-space as the local model \citep{Seeley1969Trace,Seeley1969Resolvent,Greiner1971,
Grubb1996,Gilkey1995,Grieser2004}. In singular geometries, corners, edges, cusps, and other local
models modify the expansion \citep{vanDenBergSrisatkunarajah1988,
vanDenBergSrisatkunarajah1990,vanDenBerg1998Cusps}. The setting here is different. We do not study
the intrinsic heat kernel of a manifold, nor the heat kernel of a domain with boundary conditions.
We study the Euclidean Gaussian regularization of a measure \(q\), possibly singular with respect to
Lebesgue measure. Nevertheless, the same locality principle holds: small-noise asymptotics are
governed by the geometry seen at the \(\sigma\)-scale.

There is also a connection with kernel methods on manifolds. In diffusion maps and related
constructions, ambient Gaussian kernels recover intrinsic diffusion operators after suitable
normalization \citep{BerardBessonGallot1994,CoifmanLafon2006, BelkinNiyogi2008}. The present work is
complementary: it focuses on pointwise small-noise asymptotics of the density and its logarithmic
derivatives, rather than on operator convergence.

The leading boundary behavior is also classical in nonparametric statistics. In kernel density
estimation, the bias near the edge of the support is governed by a half-space model on a boundary
layer of width comparable to the bandwidth, and boundary kernels are designed to correct it
\citep{wandjones1995,jones1993boundary}. This is exactly our leading half-space factor \(\Phi_{\rm
N}\) in the codimension-one, full-dimensional case \(c=1\), \(k=0\); see \Cref{sec:part_cases}.
The generalization developed here --- to ambient codimension \(k\), boundary faces, and corners of
arbitrary codimension \(c\), together with the uniform control of the logarithmic derivatives ---
appears to be new.

\subsection{Organization of the paper}
\label{subsec:introduction-proof-strategy}

The rest of the paper is organized as follows. \Cref{sec:preliminaries-on-manifolds} recalls the
local geometric notions used throughout the paper: manifolds with corners, strata, adapted tangent
and normal frames, tubular coordinates, volume measure, inward tangent cones, and the Jacobian
factors associated with local corner charts. \Cref{sec:main_results} states the main boundary-layer
asymptotic results. After introducing the standing global and local assumptions, it defines the
boundary-layer coordinates, the linearized cone coefficient, and the first correction coefficient;
it then gives the corresponding expansions for the heat-regularized density, its logarithm, the
score, the log-Hessian, and the scale derivative of the score. \Cref{sec:part_cases} specializes the
general formulas to three representative cases: smooth full-dimensional densities, smooth embedded
manifolds without boundary, and a full-dimensional support with boundary. \Cref{sec:sketch_proofs}
explains the main ideas of the proof, including localization of the kernel, rescaling to the conical
layer, Taylor expansion of the chart and exponent, control of the far-field contribution, and
differentiation of the logarithmic expansion. The appendices contain the complete technical
arguments.

\paragraph{Notation.}
We write \(\|\cdot\|\) and \(\langle\cdot,\cdot\rangle\) for the Euclidean norm and inner product.
If \(0\leqslant c\leqslant m\), we denote the standard corner quadrant by \(\mathbb H_c^m = \mathbb
R^{m-c} \times[0,\infty)^c\). Throughout, \(\mathbb N=\{0,1,2, \ldots\}\). For \(A\subset\mathbb
R^d\), $\overline{A}$ is the closure of \(A\). The notation \(A\Subset B\) means that \(A\) is
compactly contained in \(B\), that is, \(\overline A\) is compact and \(\overline A\subset B\). For
\(R>0\) and \(a\in\mathbb R^p\), \(\mathbb B_R^p(a)\) denotes the open Euclidean ball of radius
\(R\) and center \(a\) in \(\mathbb R^p\); when the dimension is clear and \(a=0_p\), we simply
write \(\mathbb B_R\). Throughout the paper, \(d\) is the ambient dimension, \(m\) is the dimension
of the support manifold, and \(k=d-m\) is its ambient codimension. The letter \(c\) denotes the
codimension of the active stratum inside the support. The measure \(\dd\vol_{\Man}\) denotes the
\(m\)-dimensional volume measure on \(\Man\) induced by the Euclidean metric. We write \(\mathsf
O(p)\) for the orthogonal group and \(\mathsf{GL}(p)\) for the group of invertible \(p\times p\)
matrices. For two matrices or vectors \(\bfA\) and \(\bfB\), we denote by \([\,\bfA\ \bfB\,]\) and
\([\,\bfA;\, \bfB\,]\) their horizontal and vertical concatenation, respectively, provided that
their dimensions are compatible. For a differentiable map \(F:\mathbb R^n\to\mathbb R^p\),
\(v\mapsto F(v)\), we write \(\Diff F(v)\) for its Jacobian, the \(p\times n\) matrix with entries
\((\Diff F(v))_{ij}=\partial F_i/\partial v_j(v)\). Thus, for \(w\in\mathbb R^n\), \((\Diff
F(v))[w]=\Diff F(v)w\in\mathbb R^p\), and its \(i\)-th component is \(\langle\nabla
F_i(v),w\rangle\). For scalar-valued \(f\), we write \(\nabla f=(\Diff f)^\top\). More generally,
for vector-valued \(F\), \((\Diff F)^\top\) is the transpose Jacobian; equivalently, \(\big((\Diff
F(v))[w]\big)^\top=w^\top(\Diff F(v))^\top\). The same convention applies to distinguished blocks of
variables: \(\Diff_aF\) denotes the Jacobian with respect to \(a\), and, for scalar-valued \(f\),
\(\nabla_af=(\Diff_af)^\top\). Higher differentials \(\Diff^jF(v)\) are viewed as multilinear maps;
in particular, for scalar \(f\), \(\nabla^2 f=\Diff(\nabla f)\). Finally, a function defined on a
relatively open subset of a quadrant is said to be \(C^\ell\) if it is the restriction of a
Euclidean \(C^\ell\) function in a neighborhood of each point.

%% file: figures/manifold_examples.tex

\definecolor{sphereblue}{RGB}{65,135,210}
\definecolor{diskgreen}{RGB}{88,145,62}
\definecolor{squarepurple}{RGB}{108,66,170}
\definecolor{triangorange}{RGB}{220,125,30}

\tikzset{
  axis/.style={-Latex,line width=.6pt,black},
  labeltext/.style={font=\fontsize{9}{10}\selectfont},
  titletext/.style={font=\bfseries\fontsize{9.0}{10.0}\selectfont,align=center},
  point/.style={circle,fill=black,inner sep=1.15pt},
  surfline/.style={line width=.45pt,opacity=.45},
  borderline/.style={line width=1.05pt},
}

\newcommand{\figaxes}[1]{%
  \begin{scope}[shift={#1},scale=.55]
    \draw[axis] (0,0)--(0,1.15) node[above=-1pt] {$z$};
    \draw[axis] (0,0)--(1.05,0) node[right=-1pt] {$y$};
    \draw[axis] (0,0)--(-.63,-.72) node[below left=-2pt] {$x$};
  \end{scope}}

\newcommand{\paneltextA}[1]{%
  \node[labeltext,anchor=north west,align=left,text width=4.45cm] at #1 {\emph{Interior point} (like $x$):\\[-1pt]
  neighborhood $\simeq \mathbb R^3$.\\[-1pt]
  \emph{Boundary point} (like $y$):\\[-1pt]
  neighborhood $\simeq \mathbb R^2\times[0,\infty)$.\\[-1pt]
  \emph{No corner points.}};}

\newcommand{\paneltextB}[1]{%
  \node[labeltext,anchor=north west,align=left,text width=4.45cm] at #1 {\emph{Interior point} (like $x$):\\[-1pt]
  neighborhood $\simeq \mathbb R^2$.\\[-1pt]
  \emph{Boundary point} (like $y$):\\[-1pt]
  neighborhood $\simeq \{u_2\ge 0\}$\\[-1pt]
  \emph{No corner points.}};}

\newcommand{\paneltextC}[1]{%
  \node[labeltext,anchor=north west,align=left,text width=4.45cm] at #1 {\emph{Interior point} (like $x$):\\[-1pt]
  neighborhood $\simeq \mathbb R^2$.\\[-1pt]
  \emph{Boundary point} (like $y$):\\[-1pt]
  neighborhood $\simeq \{u_2\ge 0\}$.\\[-1pt]
  \emph{Corner point} (like $z$):\\[-1pt]
  neighborhood $\simeq [0,\infty)^2$.};}

\resizebox{\textwidth}{!}{%
\begin{tikzpicture}[x=1cm,y=1cm]

\coordinate (P1) at (0,0);
\coordinate (P2) at (6.2,0);
\coordinate (P3) at (12.4,0);

\node[titletext] at ($(P1)+(2.15,5.45)$) {(a) Closed ball $\overline B^3$\\[-1pt](full-dimensional: $m=d=3$)};
\node[titletext] at ($(P2)+(2.15,5.45)$) {(b) Warped disk $D^2$\\[-1pt](2D manifold with boundary)};
\node[titletext] at ($(P3)+(2.15,5.45)$) {(c) Warped square\\[-1pt](2D manifold with corners)};

\figaxes{($(P1)+(.22,2.55)$)}
\figaxes{($(P2)+(.22,2.55)$)}
\figaxes{($(P3)+(.22,2.55)$)}

\begin{scope}[opacity=.13]
\fill[black] ($(P1)+(2.05,2.08)$) ellipse (1.05 and .16);
\fill[black] ($(P2)+(2.35,2.08)$) ellipse (1.25 and .16);
\fill[black] ($(P3)+(2.38,2.05)$) ellipse (1.20 and .17);
\end{scope}

\begin{scope}[shift={($(P1)+(2.05,3.42)$)}]
  \shade[ball color=sphereblue!55!white,opacity=.84] (0,0) circle (1.10);
  \draw[sphereblue!45!black,line width=.55pt,opacity=.42]
    (-.86,-.16) .. controls (-.35,-.38) and (.35,-.38) .. (.86,-.16);
  \draw[sphereblue!45!black,line width=.45pt,dashed,opacity=.50]
    (-.86,-.16) .. controls (-.35,.05) and (.35,.05) .. (.86,-.16);
  \draw[surfline,sphereblue!40!black]
    (0,-1.05) .. controls (-.28,-.35) and (-.28,.35) .. (0,1.05);
  \draw[surfline,sphereblue!40!black]
    (0,-1.05) .. controls (.28,-.35) and (.28,.35) .. (0,1.05);
  \draw[borderline,sphereblue!55!black] (0,0) circle (1.10);
  \node[point,label={[labeltext]right:$y$}] at (.95,.52) {};
  \node[point,label={[labeltext]right:$x$}] at (-.08,-.06) {};
\end{scope}
\paneltextA{($(P1)+(.72,1.55)$)}

\begin{scope}[shift={($(P2)+(2.15,3.42)$)}]
  \path[shade,left color=diskgreen!38!white,right color=diskgreen!65!white,opacity=.78,draw=diskgreen!55!black,borderline]
    (-1.40,-.25) .. controls (-1.10,.82) and (-.32,1.12) .. (.38,.62)
    .. controls (1.18,.05) and (1.38,.38) .. (1.52,-.06)
    .. controls (1.25,-.56) and (.54,-.72) .. (-.06,-.50)
    .. controls (-.80,-.18) and (-1.10,-.88) .. (-1.40,-.25) -- cycle;
  \clip (-1.40,-.25) .. controls (-1.10,.82) and (-.32,1.12) .. (.38,.62)
    .. controls (1.18,.05) and (1.38,.38) .. (1.52,-.06)
    .. controls (1.25,-.56) and (.54,-.72) .. (-.06,-.50)
    .. controls (-.80,-.18) and (-1.10,-.88) .. (-1.40,-.25) -- cycle;
  \foreach \t in {-1.0,-.6,-.2,.2,.6,1.0}{
    \draw[surfline,diskgreen!35!black] (-1.25,\t*.35) .. controls (-.55,\t*.78+.28) and (.55,\t*.42-.18) .. (1.38,\t*.14-.05);
    \draw[surfline,diskgreen!35!black] (\t,-.62) .. controls (\t*.32,-.10) and (\t*.45,.45) .. (\t*.60,.87);
  }
  \node[point,label={[labeltext]left:$y$}] at (1.50,-.06) {};
  \node[point,label={[labeltext]right:$x$}] at (.15,.03) {};
\end{scope}
\paneltextB{($(P2)+(.70,1.55)$)}

\begin{scope}[shift={($(P3)+(2.25,3.35)$)}]
  \path[shade,left color=squarepurple!28!white,right color=squarepurple!56!white,opacity=.76,draw=squarepurple!55!black,borderline]
    (-1.18,-.55) .. controls (-.95,.08) and (-1.03,.62) .. (-.87,.92)
    .. controls (-.26,.78) and (.50,.96) .. (1.05,1.15)
    .. controls (1.26,.47) and (1.20,-.12) .. (1.12,-.72)
    .. controls (.40,-.62) and (-.27,-.81) .. (-1.18,-.55) -- cycle;
  \clip (-1.18,-.55) .. controls (-.95,.08) and (-1.03,.62) .. (-.87,.92)
    .. controls (-.26,.78) and (.50,.96) .. (1.05,1.15)
    .. controls (1.26,.47) and (1.20,-.12) .. (1.12,-.72)
    .. controls (.40,-.62) and (-.27,-.81) .. (-1.18,-.55) -- cycle;
  \foreach \t in {-0.8,-0.4,0,0.4,0.8}{
    \draw[surfline,squarepurple!40!black] (-1.03,\t+.05) .. controls (-.25,\t+.22) and (.45,\t-.02) .. (1.17,\t*.85+.08);
    \draw[surfline,squarepurple!40!black] (\t,-.70) .. controls (\t*.75,-.04) and (\t*.90,.55) .. (\t*.68,1.07);
  }
\node[point,label={[labeltext]below left:$z$}] at (1.0,1.11) {};
\node[point,label={[labeltext]left:$y$}] at (1.14,.02) {};
\node[point,label={[labeltext]left:$x$}] at (.10,-.05) {};
\end{scope}
\paneltextC{($(P3)+(.68,1.55)$)}

\end{tikzpicture}%
}

%% file: figures/heat_density.tex
%

\begin{tikzpicture}
  \node[inner sep=0pt] at (0,0) {%
    \includegraphics[width=0.95\textwidth]{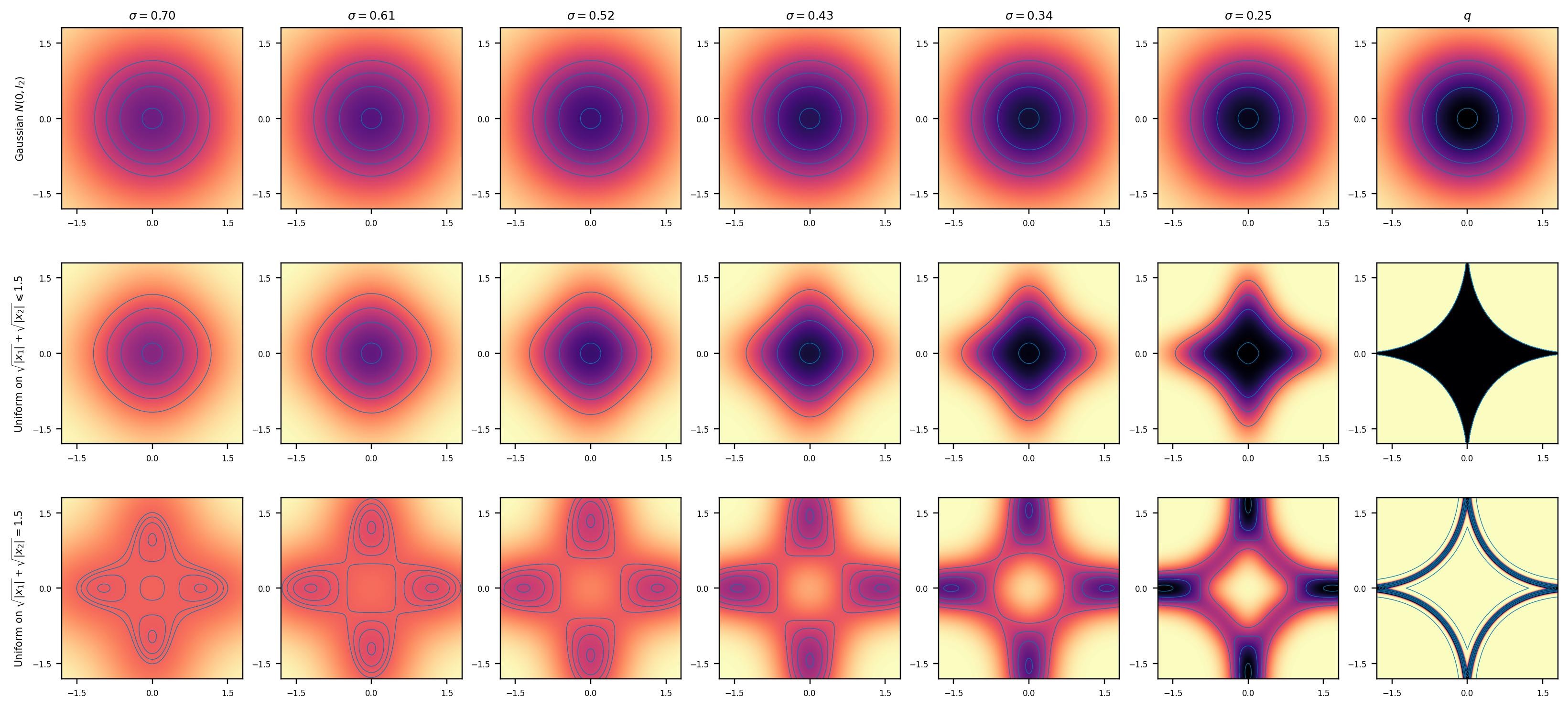}%
  };
\end{tikzpicture}

%% file: sections/preliminaries.tex
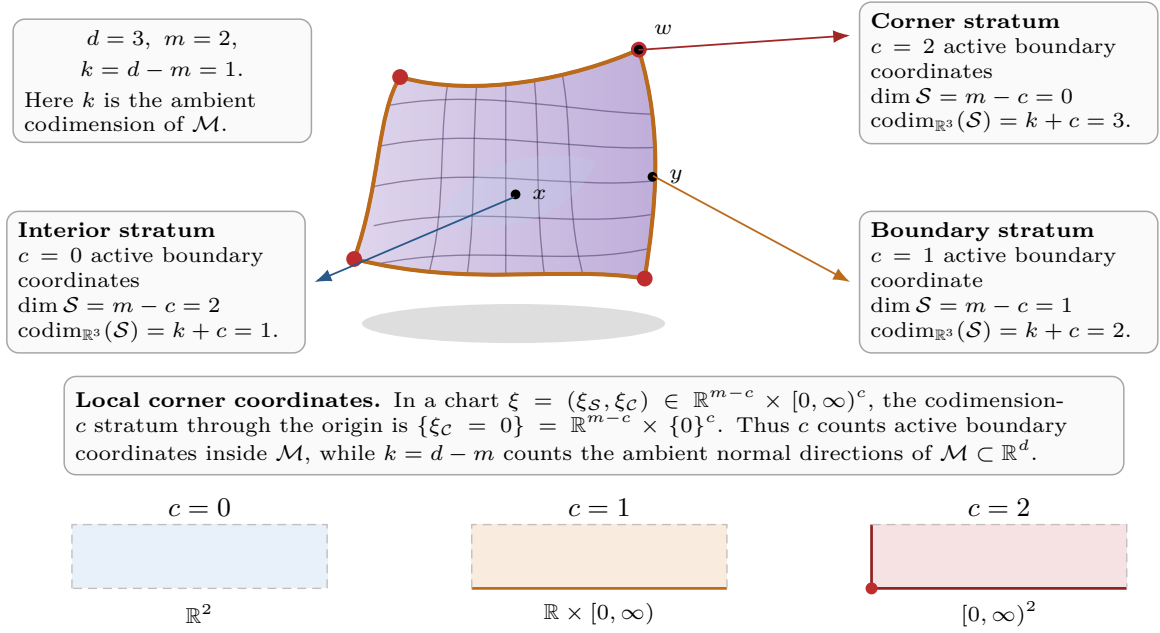
\begin{figure*}[t]
	\centering
	\input{figures/strata_warped_square_input_layout}
	\vspace{-20pt}
	\caption{\small Strata of a two-dimensional manifold with corners \(\Man\subset\mathbb R^3\). Here
		\(m=2\), \(d=3\), and \(k=d-m=1\). The interior, boundary curves, and corner vertices have
		codimensions \(c=0,1,2\) inside \(\Man\), respectively.}
	\label{fig:strata-warped-square}
\end{figure*}

\section{Preliminaries on manifolds}
\label{sec:preliminaries-on-manifolds}

We use only local geometric notions, all inherited from the ambient Euclidean space \(\mathbb R^d\).
This section recalls the few notions from the geometry of manifolds with boundary and corners that
are needed in the paper. Typical examples of manifolds we have in mind include smooth data manifolds
such as curves and surfaces, supports constrained to a half-space or a bounded domain, and
piecewise-smooth sets with edges or corners such as rectangles, cubes, and simplices. Our
assumptions allow the support of the measure to exhibit all of these local geometries. The notions
introduced below are used only to describe the local structure seen by the Gaussian smoothing at
small scales. Readers already familiar with manifolds with boundary or corners may skip this section
on first reading. Readers seeking a more detailed background may consult \citep{lee2013smooth} for
smooth manifolds and embedded submanifolds and \citep{francisStaiteJoyce2024corners} for manifolds
with corners.

\paragraph{Manifolds with corners.}

An \(m\)-dimensional embedded manifold with corners \(\Man\subset\mathbb R^d\) is a set which, near
each of its points, can be described by a sufficiently differentiable parametrization of a Euclidean
quadrant. More precisely, for every point \(x\in\Man\), there exist an integer
\(c_x\in\{0,\ldots,m\}\), a relatively open set \(\,\mathcal U_x\subset\mathbb H_{c_x}^m\)
containing \(0\), and a \(C^\ell\) map \(\Phi_x:\mathcal U_x\to\mathbb R^d\) (called corner chart),
such that \(\Phi_x(0)=x\), \(\Diff\Phi_x\) has rank \(m\) at every point, and \(\Phi_x\) is a
homeomorphism from \(\,\mathcal U_x\) onto a neighborhood of \(x\) in \(\Man\). The integer \(c_x\)
indicates the number of active boundary directions at \(x\). Thus \(c_x=0\) at an ordinary interior
point, \(c_x=1\) at a boundary point, and \(c_x\geqslant2\) at a corner point where several boundary
faces meet. Typical examples include smooth surfaces such as spheres, manifolds with boundary such
as disks, and manifolds with corners such as polytopes embedded in \(\mathbb R^3\), see
\Cref{fig:manifold-examples}.

\paragraph{Strata and codimension.}

A stratum is a smooth piece of \(\Man\) on which the same number of boundary coordinates are active.
Thus, in the notation above, if \(\Strat\) is a codimension-\(c\) stratum and \(x\in \Strat\), then
the number of active boundary coordinates at \(x\) is \(c_x=c\). We choose local corner coordinates
adapted to \(\Strat\), represented by a corner chart \(\Phi_x:\mathcal U_x\subset \mathbb H_c^m \to
\Man \subset \mathbb R^d\), \( \Phi_x(0)=x \). Here ``adapted to \(\Strat\)'' means that the stratum
\(\Strat\) is represented in the coordinate domain by setting the last \(c\) coordinates equal to
zero. More precisely, locally near \(x\),
\begin{equation*}
  \Strat\cap \Phi_x(\,\mathcal U_x)
  = \Phi_x\!\left(\,\mathcal U_x\cap \bigl(\mathbb R^{m-c}\times\{0\}^c\bigr)\right).
\end{equation*}
Equivalently, writing \(\xi=(\xi_{\mathcal S},\xi_{\mathcal C})\) with \(\xi_{\mathcal S}\in\mathbb
R^{m-c}\) and \( \xi_{\mathcal C}\in [0,\infty)^c\), the stratum is described in the coordinate
domain by \(\bigl\{\xi\in \mathcal U_x : \xi_{\mathcal C}=0\bigr\}\). Thus a codimension-\(c\)
stratum \(\Strat\subset \Man\) has dimension \(m-c\). In this paper, \(k=d-m\) denotes the ambient
codimension of the support \(\Man\subset\mathbb R^d\), whereas \(c\) denotes the codimension of the
active stratum inside \(\Man\). Hence the ambient codimension of \(\Strat\) is \(d-(m-c)=k+c\), see
\Cref{fig:strata-warped-square}.

\paragraph{Tangent spaces and normal spaces.}

For \(x\in \Man\), choose a corner chart \(\Phi_x:\mathcal U_x \to \Man\) centered at \(x\), so that
\(\Phi_x(0)=x\). The tangent space is defined by
\begin{equation*}
  \mathcal T_x\Man = \operatorname{Im}\Diff\Phi_x(0) \subset \mathbb R^d .
\end{equation*}
Equivalently, \(\mathcal T_x\Man\) is the linear subspace spanned by the columns of
\(\Diff\Phi_x(0)\). This subspace does not depend on the chosen local parametrization: changing
coordinates only changes the basis used to describe the same image space. At boundary or corner
points, \(\mathcal T_x\Man\) is still a linear space, and therefore does not retain the one-sided
constraints defining the boundary or corner. These constraints are encoded instead by the inward
tangent cone defined below.

A stratum \(\Strat\) containing \(x\) has its own tangent space \(\mathcal T_x\Strat\subset \mathcal
T_x\Man\), of dimension \(m-c\). The directions in \(\mathcal T_x\Man\) that are orthogonal to
\(\mathcal T_x\Strat\) form the \(c\)-dimensional linear space \(\mathcal C_x = \mathcal T_x\Man\cap
(\mathcal T_x\Strat)^\perp\). It consists of directions that remain tangent to \(\Man\), but are
transverse to the stratum \(\Strat\) within \(\Man\). These directions encode boundary or corner
effects. The ambient normal space of \(\Man\) at \(x\) is \(\mathcal N_x\Man = (\mathcal
T_x\Man)^\perp \subset \mathbb R^d\). It has dimension \(k=d-m\). These directions encode the fact
that the support may be lower-dimensional in the ambient space. Thus, at a point \(x\in \Strat\),
the relevant orthogonal decomposition is
\begin{equation*}
  \mathbb R^d =\mathcal T_x\Strat \oplus \mathcal C_x \oplus \mathcal N_x\Man .
\end{equation*}
In the sequel, \(\bfS(x)\in\mathbb R^{d\times(m-c)}\), \(\bfC(x) \in\mathbb R^{d\times c}\) and
\(\bfN(x)\in\mathbb R^{d \times k}\) refer to orthonormal matrices whose columns span \(\mathcal
T_x\Strat\), \(\mathcal C_x\) and \(\mathcal N_x\Man\), respectively.

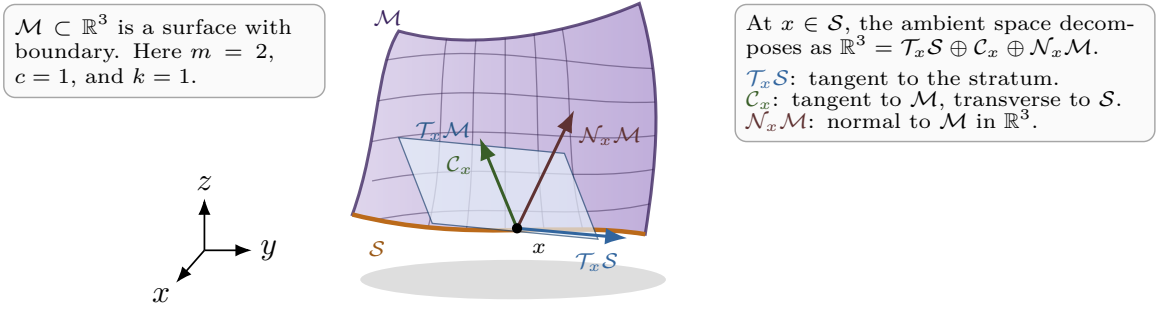
\begin{figure*}[t]
  \centering
  \input{figures/tangent_normal_decomposition_input}
  \vspace*{-20pt}
  \caption{\small Tangent and normal directions near a boundary stratum. Here \(\Man\subset\mathbb
  R^3\) is a two-dimensional surface with boundary, \(\Strat\) is a one-dimensional boundary
  stratum, and \(x\in\Strat\). The tangent space \(\mathcal T_x\Man\) splits into the direction
  \(\mathcal T_x\Strat\) tangent to the stratum and the transverse direction \(\mathcal C_x\) inside
  the support. The remaining direction \(\mathcal N_x\Man\) is normal to the support in the ambient
  space.}
  \label{fig:tangent-normal-decomposition}
\end{figure*}

\paragraph{Tubular coordinates near a stratum.}

Let \(\Compct_{\Strat}\) be a compact subset of a smooth stratum \(\Strat\). After restricting to a
sufficiently small neighborhood of \(\Compct_{\Strat}\), the nearest-point projection onto
\(\Strat\), denoted by \(\pi\), is well-defined and smooth. Thus every point \(y\) in this
neighborhood can be written uniquely in the form \(y=x+\bfC(x)u(y)+\bfN(x)\eta(y)\), where
\(x=\pi(y)\in\Strat\), \(u(y)\in\mathbb R^c\), and \(\eta(y)\in\mathbb R^k\). The vector
\(\bfC(x)u(y)\) is tangent to \(\Man\) but transverse to \(\Strat\), while \(\bfN(x)\eta(y)\) is
normal to \(\Man\) in the ambient space. We collect these transverse coordinates as
\begin{equation*}
  \nu(y)=(u(y),\eta(y)).
\end{equation*}
The boundary-layer scaling studies points whose transverse displacement from the stratum is of order
\(\sigma\). We therefore write such points as
\begin{equation*}
  y = y_\sigma(a,x) = x+\sigma \bfC(x)a_{\mathcal C}+\sigma \bfN(x)a_{\mathcal N},
  \qquad a=(a_{\mathcal C},a_{\mathcal N})\in\mathbb R^c\times \mathbb R^k.
\end{equation*}
Here \(a_{\mathcal C}\) and \(a_{\mathcal N}\) are dimensionless coordinates. The coordinate
\(a_{\mathcal C}\) measures displacement from the stratum inside the support, in units of
\(\sigma\), whereas \(a_{\mathcal N}\) measures ambient-normal displacement from \(\Man\), again in
units of \(\sigma\). The variable \(x\) indicates the location along the stratum, whereas \(a\)
refers to the rescaled transverse position.

\begin{table}[p]
  \centering
  \small
  \renewcommand{\arraystretch}{1.12}
  \begin{tabularx}{\textwidth}
{@{}p{0.11\textwidth}p{0.41\textwidth}X@{}}
\toprule
Notation & Object & Meaning \\
\midrule

\(\Man\)
& Carrying manifold
& \(m\)-dimensional manifold with corners in \(\mathbb R^d\) on which
\(q\) is concentrated; near \(\Compct_{\Strat}\) it agrees locally
with \(\operatorname{supp}(q)\). \\

\(\Strat\)
& Active stratum
& Smooth piece of \(\Man\) with \(c\) active boundary
coordinates. \\

\(k\)
& Ambient codimension \(d-m\)
& Codimension of \(\Man\) in \(\mathbb R^d\). \\

\(c\)
& Stratum codimension
& Codim.\ of \(\Strat\) inside \(\Man\);
\(\dim(\Strat)=m-c\). \\

\(\mathbb H_c^m\)
& Local quadrant
& \(\mathbb R^{m-c}\times[0,\infty)^c\). \\

\(\Phi_x\)
& Local parametrization or corner chart
& Map from a local quadrant $\mathcal U_x\subset \mathbb H_c^m$
into \(\Man\), centered at \(x\). \\

\(\xi\)
& Corner coordinates \((\xi_{\mathcal S},\xi_{\mathcal C})\)
& Coordinates split into stratum \(\xi_{\mathcal{S}}\in
\mathbb R^{m-c}\) and corner \(\xi_{\mathcal{C}}\in
[0,\infty)^{c}\) variables. \\

\(\mathcal T_x\Man\)
& Tangent space
& Linear tangent space to \(\Man\) at \(x\). \\

\(\mathcal T_x\Strat\)
& Stratum tangent space
& Tangent space to the stratum at \(x\). \\

\(\mathcal C_x\)
& Tangent-transverse space \(\mathcal T_x\Man\cap(\mathcal T_x\Strat)^\perp\)
& Directions tangent to \(\Man\),  transverse to
\(\Strat\). \\

\(\mathcal N_x\Man\)
& Ambient normal space \((\mathcal T_x\Man)^\perp\)
& Directions normal to \(\Man\) in \(\mathbb R^d\). \\

\(\bfS(x)\)
& \(\mathcal T_x\Strat\)-frame
& Matrix with orthonormal columns spanning \(\mathcal T_x\Strat\). \\

\(\bfC(x)\)
& \(\mathcal C_x\)-frame
& Matrix with orthonormal columns spanning \(\mathcal C_x\). \\

\(\bfN(x)\)
& \(\mathcal N_x\Man\)-frame
& Matrix with orthonormal columns spanning \(\mathcal N_x\Man\). \\

\(\bfP_{\mathcal{N}}(x)\)
& Normal projector \(\bfN(x)\bfN(x)^\top\)
& Orthogonal projection onto \(\mathcal N_x\Man\). \\

\(\pi(y)\)
& Base-point projection
& Projection of \(y\) onto the stratum \(\Strat\). \\

\(u(y)\)
& Tangent-transverse coordinate
& Coordinate in the \(\mathcal C_x\) directions. \\

\(\eta(y)\)
& Ambient-normal coordinate
& Coordinate in the \(\mathcal N_x\Man\) directions. \\

\(\nu(y)\)
& Transverse coordinate \((u(y),\eta(y))\)
& Transverse displacement from the stratum. \\
\(a\)
& Rescaled transverse coordinate \((a_{\mathcal C},a_{\mathcal N})\)
& Transverse displacement from the stratum measured
in units of
\(\sigma\). \\

\(y_\sigma(a,x)\)
& Boundary-layer point
& Observation point
\(x+\sigma \bfC(x)a_{\mathcal C}+\sigma \bfN(x)a_{\mathcal N}\). \\

\(\dd\vol_{\Man}\)
& Volume measure
& \(m\)-dim.\ Euclidean volume measure on \(\Man\). \\

\(J_x\)
& Volume Jacobian
& Jacobian in local integration over \(\Man\). \\

\(\mathcal T_x^+\Man\)
& Inward tangent cone
& Cone \(\Diff\Phi_x(0)\mathbb H_c^m\), retaining boundary and
corner constraints. \\

\(\bfL(x)\)
& Adapted differential
& Matrix from \(\mathsf{GL}(m)\) defined by \(\Diff\Phi_x(0) =
\left[\bfS\ \bfC\right](x)\bfL(x)\). \\

\(\mathrm{II}_x\)
& Second fundamental form
& Quadratic normal bending of \(\Man\) in \(\mathbb R^d\). \\

\(\bsh_{\Man}(x)\)
& Mean-curvature vector
& Unnormalized trace of \(\mathrm{II}_x\). \\

\(\nabla_{\Man} f(x)\)
& Manifold gradient
& Gradient of \(f\) along \(\Man\), viewed in
\(\mathcal T_x\Man\). \\

\bottomrule
  \end{tabularx}
  \caption{ Summary of the local geometric framework: manifolds, strata, tangent and normal spaces,
  boundary-layer coordinates, tangent cones, and curvature quantities. }
  \label{tab:notation-boundary-layer}
\end{table}

\paragraph{Volume measure on the support.}

The measure \(\dd\vol_{\Man}\) is the \(m\)-dimensional volume measure on \(\Man\) induced by the
Euclidean metric of \(\mathbb R^d\). In a corner chart \(\Phi_x:\mathcal U_x\to\Man\), integration
with respect to this measure is given by the change-of-variables formula
\begin{equation*}
  \int_{\Phi_x(\,\mathcal U_x)} f(x')\,\dd\vol_{\Man}(x')
  = \int_{\mathcal U_x} f(\Phi_x(\xi))\,J_x(\xi)\,\dd\xi .
\end{equation*}
Here \(\dd\xi\) denotes Lebesgue measure on the coordinate domain \(\mathcal U_x\), and
\begin{equation*}
  J_x(\xi) = \sqrt{\det\!\bigl(\Diff\Phi_x(\xi)^\top \Diff\Phi_x(\xi)\bigr)}
\end{equation*}
is the volume Jacobian of the chart \(\Phi_x\). Thus writing \(q(\dd x)=\rho(x)\,\dd\vol_{\Man}(x)\)
means that \(q\) has density \(\rho\) with respect to the natural \(m\)-dimensional volume measure
on the support \(\Man\), not necessarily with respect to the Lebesgue measure in \(\mathbb R^d\).

\paragraph{Inward tangent cone.}

At a boundary or corner point, the linear tangent space \(\mathcal T_x\Man\) does not retain the
one-sided constraints that define the local quadrant. The inward tangent cone retains these
constraints. Let \(x\in\Strat\), where \(\Strat\) is a codimension-\(c\) stratum, and let
\(\Phi_x:\mathcal U_x\to\Man\) be a corner chart adapted to \(\Strat\). The inward tangent cone at
\(x\) is defined by
\begin{equation*}
  \mathcal T_x^+\Man = \Diff\Phi_x(0)\,\mathbb H_c^m \subset \mathcal T_x\Man .
\end{equation*}
This cone is independent of the chosen adapted local parametrization. When \(c=0\), the quadrant is
\(\mathbb R^m\), and \(\mathcal T_x^+\Man=\mathcal T_x\Man\). When \(c=1\), the cone is a half-space
in \(\mathcal T_x\Man\). When \(c\geqslant2\), it is a corner cone. This cone is the first-order
object seen when one zooms in on the support near \(x\). Indeed, for bounded \(\zeta\in\mathbb
H_c^m\) and small \(\sigma\), Taylor expansion gives
\vspace{-3pt}
\begin{equation*}
  \Phi_x(\sigma\zeta) = x + \overbrace{\sigma \Diff\Phi_x (0)\,\zeta}^{\in \mathcal T_x^+\Man}
  + \,O(\sigma^2).
\end{equation*}
Thus, after subtracting \(x\) and rescaling by \(\sigma^{-1}\), the local support converges to the
linearized inward cone \(\Diff\Phi_x(0)\,\mathbb H_c^m\).

\paragraph{The matrix \(\bfL(x)\) and its Jacobian.}

In the main expansion, we express the differential \(\Diff\Phi_x(0)\) in an adapted orthonormal
frame of \(\mathcal T_x\Man\). More precisely, since \( [\,\bfS(x)\ \bfC(x)\,] \in\mathbb R^{d
\times m}\) is an orthonormal matrix whose first \(m-c\) columns span \(\mathcal T_x\Strat\) and
whose last \(c\) columns span \(\mathcal C_x\), there is an invertible matrix \(\bfL(x)\in
\mathsf{GL}(m)\) such that \(\Diff \Phi_x(0)=[\,\bfS(x)\ \bfC(x)\,] \bfL(x)\).

In adapted tangent coordinates, the linearized inward cone is \(\bfL(x)\mathbb H_c^m\).
Equivalently, in ambient coordinates, the same cone is \(\mathcal T_x^+\Man=[\,\bfS(x) \
\bfC(x)\,]\,\bfL(x)\mathbb H_c^m\). We can write \(\bfL(x) = [\,\bfL_{\mathcal
S}(x);\,\bfL_{\mathcal C}(x)\,]\) where \(\bfL_{\mathcal S}(x)\) contains the first \((m-c)\) rows
of \(\bfL(x)\) and \(\bfL_{\mathcal C}(x)\) contains the last \(c\) rows of the same matrix, so that
\([\,\bfS(x) \ \bfC(x)\,]\,\bfL(x) = \bfS(x)\bfL_{\mathcal S} (x) + \bfC(x) \bfL_{\mathcal C}(x)\).

The determinant factor \(|\det \bfL(x)|\) is the leading-order Jacobian converting corner
coordinates into orthonormal tangent coordinates on \(\Man\). Since \( \Diff\Phi_x(0) = [\,\bfS(x)\
\bfC(x)\,]\bfL(x)\) and \([\,\bfS(x)\ \bfC(x)\,]\) has orthonormal columns, the volume Jacobian
satisfies
\begin{equation*}
  J_x(0) = \sqrt{\det\!\bigl(\Diff\Phi_x(0)^\top \Diff\Phi_x(0)\bigr)} = |\det \bfL(x)|.
\end{equation*}
Hence, for bounded \(\zeta\),
\begin{equation*}
  J_x(\sigma\zeta)\,\sigma^m\,\dd\zeta = \sigma^m \bigl(|\det \bfL(x)|+O(\sigma)\bigr)\,\dd\zeta .
\end{equation*}

\paragraph{Curvature terms used in the smooth no-boundary case.}

In the special case \(c=0\), the inward tangent cone coincides with the tangent space, \(\mathcal
T_x^+\Man=\mathcal T_x\Man\). Thus the leading local model at \(x\) is the tangent plane \(\mathcal
T_x\Man\), and the first correction to this linear approximation is encoded by the second
fundamental form, a symmetric bilinear map \(\mathrm{II}_x: \mathcal T_x\Man\times\mathcal T_x \Man
\to \mathcal N_x\Man\). It measures the curvature of \(\Man\) in the ambient normal directions. For
a general adapted chart, the normal component of the quadratic Taylor term is described by
\(\mathrm{II}_x\). In suitable tangent coordinates centered at \(x\), for example geodesic
coordinates on \(\Man\), this gives the simpler expansion
\begin{equation}\label{eq:II}
  \Phi_x(\xi) = x+\xi + \frac12\mathrm{II}_x(\xi,\xi) + O(\|\xi\|^3),
  \qquad \xi\in\mathcal T_x^+\Man\subset\mathbb R^d .
\end{equation}
Thus the tangent-plane approximation is \(x+\xi\), while \((1/2)\mathrm{II}_x(\xi,\xi)\) is the
first normal curvature correction. Equivalently, after choosing orthonormal bases of \(\mathcal
T_x\Man\) and \(\mathcal N_x\Man\), the second fundamental form is represented by coefficients
\(\mathrm{II}^{\alpha}_{ij}(x)\), with tangent indices \(i,j\in\{1,\ldots,m\}\) and normal index
\(\alpha\in\{1,\ldots,k\}\). Thus it may be viewed as a collection of \(k\) symmetric Hessian
matrices \(\mathrm{II}^{1}(x),\ldots,\mathrm{II}^{k}(x)\). With this convention for \(\mathrm{II}\),
and denoting by \(n_1(x),\dots,n_k(x)\) an orthonormal basis of \(\mathcal N_x\Man\), we define
\begin{equation}\label{eq:HM}
  \bsh_{\Man}(x) = \sum_{j=1}^k \operatorname{tr} \bigl(\mathrm{II}^{j}(x)\bigr) n_j (x).
\end{equation}
With the normal frame matrix \(\bfN(x)=[\,n_1(x)\ \cdots\ n_k(x)\,]\), the projection onto the
ambient normal space is \(\bfP_{\mathcal N}(x) = \bfN(x)\bfN(x)^\top\).

Finally, for a smooth function \(f\) on \(\Man\), \(\nabla_{\Man}f(x)\) denotes the gradient along
the manifold, with respect to the Riemannian metric inherited from \(\mathbb R^d\). It is the unique
vector in \(\mathcal T_x\Man\) whose inner product with any tangent direction \(\tau\in\mathcal
T_x\Man\) gives the directional derivative of \(f\) at \(x\) along \(\tau\). We view it as an
ambient vector through the inclusion \(\mathcal T_x\Man\subset\mathbb R^d\).

%% file: figures/strata_warped_square_input_layout.tex

\definecolor{sphereblue}{RGB}{65,135,210}
\definecolor{diskgreen}{RGB}{88,145,62}
\definecolor{squarepurple}{RGB}{108,66,170}
\definecolor{triangorange}{RGB}{220,125,30}
\definecolor{stratumred}{RGB}{190,45,45}

\tikzset{
  axis/.style={-Latex,line width=.6pt,black},
  labeltext/.style={font=\fontsize{7.2}{8.3}\selectfont},
  titletext/.style={font=\bfseries\fontsize{9.0}{10.0}\selectfont,align=center},
  point/.style={circle,fill=black,inner sep=1.15pt},
  surfline/.style={line width=.45pt,opacity=.45},
  borderline/.style={line width=1.05pt},
  callout/.style={font=\fontsize{7.2}{8.3}\selectfont,align=left},
  calloutbox/.style={rounded corners=4pt,draw=black!35,line width=.5pt,fill=black!2},
}

\resizebox{\textwidth}{!}{%
\begin{tikzpicture}[x=1cm,y=1cm]

\node[calloutbox,callout,text width=3.25cm,anchor=north west] at (-.8,4.65)
  {\centering
   \(d=3,\  m=2,\)\\[2pt] \(k=d-m=1.\)\\[2pt]
   \raggedright
   Here \(k\) is the ambient codimension of \(\Man\).};

\begin{scope}[shift={(5.0,2.65)},scale=1.15]

  \fill[black,opacity=.13] (.08,-1.38) ellipse (1.55 and .20);

  \path[shade,left color=squarepurple!28!white,
        right color=squarepurple!56!white,
        opacity=.77,draw=squarepurple!55!black,borderline]
    (-1.55,-.72) .. controls (-1.28,.03) and (-1.32,.76) .. (-1.08,1.14)
    .. controls (-.30,.92) and (.70,1.13) .. (1.36,1.42)
    .. controls (1.62,.55) and (1.53,-.20) .. (1.42,-.92)
    .. controls (.55,-.78) and (-.35,-1.02) .. (-1.55,-.72) -- cycle;

  \begin{scope}
  \clip
    (-1.55,-.72) .. controls (-1.28,.03) and (-1.32,.76) .. (-1.08,1.14)
    .. controls (-.30,.92) and (.70,1.13) .. (1.36,1.42)
    .. controls (1.62,.55) and (1.53,-.20) .. (1.42,-.92)
    .. controls (.55,-.78) and (-.35,-1.02) .. (-1.55,-.72) -- cycle;

  \foreach \t in {-0.8,-0.4,0,0.4,0.8}{
    \draw[surfline,squarepurple!40!black]
      (-1.35,\t+.02) .. controls (-.30,\t+.25) and (.62,\t-.02) .. (1.52,\t*.85+.08);
    \draw[surfline,squarepurple!40!black]
      (\t,-.92) .. controls (\t*.72,-.13) and (\t*.90,.62) .. (\t*.68,1.32);
  }

  \fill[sphereblue!35,opacity=.22]
    (-.70,-.35) .. controls (-.45,.25) and (.35,.45) .. (.92,.28)
    .. controls (.74,-.20) and (.18,-.50) .. (-.70,-.35) -- cycle;
  \end{scope}

  \draw[triangorange!85!black,line width=1.35pt]
    (-1.55,-.72) .. controls (-1.28,.03) and (-1.32,.76) .. (-1.08,1.14);
  \draw[triangorange!85!black,line width=1.35pt]
    (-1.08,1.14) .. controls (-.30,.92) and (.70,1.13) .. (1.36,1.42);
  \draw[triangorange!85!black,line width=1.35pt]
    (1.36,1.42) .. controls (1.62,.55) and (1.53,-.20) .. (1.42,-.92);
  \draw[triangorange!85!black,line width=1.35pt]
    (1.42,-.92) .. controls (.55,-.78) and (-.35,-1.02) .. (-1.55,-.72);

  \foreach \p in {(-1.55,-.72),(-1.08,1.14),(1.36,1.42),(1.42,-.92)}{
    \fill[stratumred] \p circle (2.3pt);
  }

  \node[point,label={[labeltext]right:$x$}] (intpt) at (.10,-.06) {};
  \node[point,label={[labeltext]right:$y$}] (bdpt) at (1.50,.12) {};
  \node[point,label={[labeltext]above right:$w$}] (corpt) at (1.36,1.42) {};

\end{scope}

\draw[-Latex,line width=.6pt,sphereblue!65!black] (5.1,2.55) -- (2.75,1.55);
\node[calloutbox,callout,text width=3.25cm,anchor=east] at (2.65,1.55)
  {\textbf{Interior stratum}\\
   \(c=0\) active boundary coordinates\\
   \(\dim \Strat=m-c=2\)\\
   \(\mathrm{codim}_{\mathbb R^3}(\Strat)=k+c=1\).};

\draw[-Latex,line width=.6pt,triangorange!85!black] (6.73,2.80) -- (9.05,1.55);
\node[calloutbox,callout,text width=3.25cm,anchor=west] at (9.15,1.55)
  {\textbf{Boundary stratum}\\
   \(c=1\) active boundary coordinate\\
   \(\dim \Strat=m-c=1\)\\
   \(\mathrm{codim}_{\mathbb R^3}(\Strat)=k+c=2\).};

\draw[-Latex,line width=.6pt,stratumred!85!black] (6.56,4.28) -- (9.05,4.45);
\node[calloutbox,callout,text width=3.25cm,anchor=north west] at (9.15,4.85)
  {\textbf{Corner stratum}\\
   \(c=2\) active boundary coordinates\\
   \(\dim \Strat=m-c=0\)\\
   \(\mathrm{codim}_{\mathbb R^3}(\Strat)=k+c=3\).};

\node[calloutbox,callout,text width=12.2cm,anchor=north west] at (-.2,.45)
  {\textbf{Local corner coordinates.} In a chart
   \(\xi=(\xi_{\mathcal S},\xi_{\mathcal C})
   \in\mathbb R^{m-c}\times[0,\infty)^c\), the
   codimension-\(c\) stratum through the origin is
   \(\{\xi_{\mathcal C}=0\}
   =\mathbb R^{m-c}\times\{0\}^c\). Thus \(c\) counts
   active boundary coordinates inside \(\Man\), while
   \(k=d-m\) counts the ambient normal directions of
   \(\Man\subset\mathbb R^d\).};


\begin{scope}[shift={(0.0,-2.05)}]
  \node[titletext] at (1.4,.95) {\(c=0\)};
  \fill[sphereblue!12] (-.1,0) rectangle (2.9,.75);
  \draw[black!25,densely dashed,line width=.45pt]
    (-.1,0) rectangle (2.9,.75);
  \node[labeltext] at (1.4,-.3) {\(\mathbb R^2\)};
\end{scope}

\begin{scope}[shift={(4.7,-2.05)}]
  \node[titletext] at (1.4,.95) {\(c=1\)};
  \fill[triangorange!20,opacity=.75] (-.1,0) rectangle (2.9,.75);
  \draw[triangorange!85!black,line width=.9pt]
    (-.1,0) -- (2.9,0);
  \draw[black!25,densely dashed,line width=.45pt]
    (-.1,0) -- (-.1,.75) -- (2.9,.75) -- (2.9,0);
  \node[labeltext] at (1.4,-.3) {\(\mathbb R\times[0,\infty)\)};
\end{scope}

\begin{scope}[shift={(9.4,-2.05)}]
  \node[titletext] at (1.4,.95) {\(c=2\)};
  \fill[stratumred!18,opacity=.75] (-.1,0) rectangle (2.9,.75);
  \draw[stratumred!75!black,line width=.9pt]
    (-.1,0) -- (2.9,0);
  \draw[stratumred!75!black,line width=.9pt]
    (-.1,0) -- (-.1,.75);
  \fill[stratumred] (-.1,0) circle (2pt);
  \draw[black!25,densely dashed,line width=.45pt]
    (-.1,.75) -- (2.9,.75) -- (2.9,0);
  \node[labeltext] at (1.4,-.3) {\([0,\infty)^2\)};
\end{scope}

\end{tikzpicture}%
}

%% file: figures/tangent_normal_decomposition_input.tex
%

\providecommand{\Man}{\mathcal{M}}
\providecommand{\Strat}{\mathcal{S}}

\definecolor{sphereblue}{RGB}{65,135,210}
\definecolor{diskgreen}{RGB}{88,145,62}
\definecolor{squarepurple}{RGB}{108,66,170}
\definecolor{triangorange}{RGB}{220,125,30}
\definecolor{stratumred}{RGB}{190,45,45}
\definecolor{normaldark}{RGB}{105,55,55}

\tikzset{
  tnaxis/.style={-Latex,line width=.6pt,black},
  tnlabel/.style={font=\fontsize{7.4}{8.5}\selectfont},
  tntitle/.style={font=\bfseries\fontsize{9.2}{10.2}\selectfont,align=center},
  tnpoint/.style={circle,fill=black,inner sep=1.25pt},
  tnsurfline/.style={line width=.45pt,opacity=.45},
  tnborder/.style={line width=1.05pt},
  tnvec/.style={-Latex,line width=1.05pt},
  tnbox/.style={rounded corners=4pt,draw=black!35,line width=.5pt,fill=black!2},
}

\resizebox{\textwidth}{!}{%
\begin{tikzpicture}[x=1cm,y=1cm]


\begin{scope}[shift={(2.55,1.25)},scale=.55]
  \draw[tnaxis] (0,0)--(0,1.15) node[above=-1pt] {$z$};
  \draw[tnaxis] (0,0)--(1.05,0) node[right=-1pt] {$y$};
  \draw[tnaxis] (0,0)--(-.63,-.72) node[below left=-2pt] {$x$};
\end{scope}

\node[tnbox,tnlabel,anchor=north west,align=left,text width=3.65cm]
  at (.10,4.25)
  {\(\Man\subset\mathbb R^3\) is a surface with boundary. 
   Here \(m=2\), \(c=1\), and \(k=1\).};

\begin{scope}[shift={(5.25,2.55)},scale=1.20]

  \fill[black,opacity=.13] (.88,-1.38) ellipse (1.55 and .20);

  \path[shade,left color=squarepurple!28!white,right color=squarepurple!56!white,
        opacity=.77,draw=squarepurple!55!black,tnborder]
    (-0.75,-.72) .. controls (-0.48,.03) and (-0.52,.76) .. (-0.28,1.14)
    .. controls (.50,.92) and (1.50,1.13) .. (2.16,1.42)
    .. controls (2.42,.55) and (2.33,-.20) .. (2.22,-.92)
    .. controls (1.35,-.78) and (0.45,-1.02) .. (-0.75,-.72) -- cycle;

  \begin{scope}
  \clip
    (-0.75,-.72) .. controls (-0.48,.03) and (-0.52,.76) .. (-0.28,1.14)
    .. controls (.50,.92) and (1.50,1.13) .. (2.16,1.42)
    .. controls (2.42,.55) and (2.33,-.20) .. (2.22,-.92)
    .. controls (1.35,-.78) and (0.45,-1.02) .. (-0.75,-.72) -- cycle;

  \foreach \t in {-0.8,-0.4,0,0.4,0.8}{
    \draw[tnsurfline,squarepurple!40!black]
      (-0.55,\t+.02) .. controls (.60,\t+.25) and (1.42,\t-.02) .. (2.32,\t*.85+.08);
    \draw[tnsurfline,squarepurple!40!black]
      (\t+.8,-.92) .. controls (\t*.72+0.8,-.13) and (\t*.90+0.8,.62) .. (\t*.68+.8,1.32);
  }
  \end{scope}

  \draw[triangorange!85!black,line width=1.55pt]
    (2.22,-.92) .. controls (1.35,-.78) and (.45,-1.02) .. (-0.75,-.72);
  \node[tnlabel,triangorange!70!black,anchor=north]
    at (-.5,-.88) {$\Strat$};

  \coordinate (xpt) at (.92,-.86);

  \fill[sphereblue!14,opacity=.70,draw=sphereblue!55!black,line width=.45pt]
    ($(xpt)+(-.86,.05)$)
    -- ($(xpt)+(.82,-.11)$)
    -- ($(xpt)+(.48,.76)$)
    -- ($(xpt)+(-1.20,.92)$)
    -- cycle;
  \node[tnlabel,sphereblue!65!black,anchor=south]
    at ($(xpt)+(-.75,.78)$) {$\mathcal T_x\Man$};

  \draw[tnvec,sphereblue!75!black]
    (xpt) -- ($(xpt)+(1.12,-.10)$)
    node[pos=.72,below=2pt,tnlabel] {$\mathcal T_x\Strat$};

  \draw[tnvec,diskgreen!65!black]
    (xpt) -- ($(xpt)+(-.38,.92)$)
    node[pos=.70,left=2pt,tnlabel] {$\mathcal C_x$};

  \draw[tnvec,normaldark!90!black]
    (xpt) -- ($(xpt)+(.58,1.20)$)
    node[pos=.78,right=2pt,tnlabel] {$\mathcal N_x\Man$};

  \node[tnpoint,label={[tnlabel]below right:$x$}] at (xpt) {};

  \node[tnlabel,anchor=west,text=squarepurple!55!black]
    at (-0.65,1.28) {$\Man$};

\end{scope}

\node[tnbox,tnlabel,anchor=north west,align=left,text width=4.85cm]
  at (9.,4.25)
  {At \(x\in\Strat\), the ambient space decomposes as 
   \(\mathbb R^3=\mathcal T_x\Strat\oplus\mathcal C_x\oplus\mathcal N_x\Man\).\\[2pt]
   \textcolor{sphereblue!75!black}{\(\mathcal T_x\Strat\)}: tangent to the stratum.\\[-1pt]
   \textcolor{diskgreen!65!black}{\(\mathcal C_x\)}: tangent to \(\Man\), transverse to \(\Strat\).\\[-1pt]
   \textcolor{normaldark!90!black}{\(\mathcal N_x\Man\)}: normal to \(\Man\) in \(\mathbb R^3\).};


\end{tikzpicture}%
}

%% file: sections/formal_main.tex

\section{Main results on boundary-layer tangent-cone asymptotics}
\label{sec:main_results}

The goal of this section is to state the main results of the paper. These results provide
expansions, near a fixed stratum of the \(m\)-dimensional manifold \(\Man\) carrying \(q\), for the
heat-regularized density \(p_\sigma\) and for the related logarithmic quantities. We begin by
stating the conditions on the probability distribution \(q\) and on the local geometric structure of
its support near a compact subset of a stratum \(\Strat\).

\subsection{Standing assumptions}
\label{subsec:standing-assumptions}

\begin{assumption}[Standing measure and local corner hypothesis]
  \label{ass:standing}
  Let \(r\in\mathbb N\). Let \(q\) be a Borel probability measure on
  \(\mathbb R^d\). Let \(\Man\subset\mathbb R^d\) be an embedded
  \(m\)-dimensional manifold with corners. Let
  \(\Strat\subset\Man\) be a codimension-\(c\) stratum and let
  \(\Compct_{\Strat}\subset\Strat\) be compact. We assume the
  following.
  \begin{enumerate}[before={\renewcommand{\theenumi}{\textup{(\standingtag)}}%
    \renewcommand{\labelenumi}{\theenumi}}, itemsep=0pt, topsep=2pt]
    \item \label{ass:standing:measure} The measure \(q\) is concentrated on \(\Man\) and is
    represented by an \(m\)-dimensional density \(\rho\) with respect to the induced volume measure
    on \(\Man\):
    \begin{equation*}
      q(\dd x)=\rho(x)\,\dd\vol_{\Man}(x),
      \qquad x\in\Man .
    \end{equation*}
    In particular, \(\int_{\Man}\rho(x)\,\dd\vol_{\Man}(x)=1 \).

    \item \label{ass:standing:chart} There is an open set \(\mathcal V\subset\mathbb R^{m-c}\), a
    positive number \(\varepsilon>0\), and an open contractible set
    \(\Theta_{\mathrm{out}}\Subset\mathcal V\) such that, for a \(C^{r+1}\) corner chart
    \(\Phi:\mathcal V\times[0,\varepsilon)^c\to\Man\), 
    which  is a homeomorphism onto a relatively open subset of \(\Man\) with
    \(\operatorname{rank}(\Diff\Phi)=m\) at every point, one has
    \begin{equation*}
      \Compct_{\Strat}
      \subset
      \Phi(\Theta_{\mathrm{out}}\times\{0\})
      =
      \Strat_{\mathrm{out}}
      \subset
      \Strat .
    \end{equation*}
    The image of \(\Theta_{\mathrm{out}}\times\{0\}\) parametrizes the part of the stratum under
    consideration. We set
    \[
      \varphi(\theta)=\Phi(\theta,0),
      \quad\text{and}\quad 
      \Theta_{\Compct_{\Strat}}
      =
      \{\theta\in\Theta_{\mathrm{out}}:
      \Phi(\theta,0)\in\Compct_{\Strat}\} = \varphi^{-1}(\Compct_\Strat).
    \]

    \item \label{ass:standing:density} The density \(\rho\) in item~\ref{ass:standing:measure} is such that
    \(\rho\circ\Phi:
      \mathcal V\times[0,\varepsilon)^c\to[0,\infty)\)
    is of class \(C^r\).

    \item \label{ass:standing:positivity} There exists \(\rho_*>0\) such that
    \(
      \rho(\Phi(\theta,0))\geqslant \rho_*\), for every 
      \(\theta\in\Theta_{\Compct_{\Strat}}\) .
  \end{enumerate}
\end{assumption}

The density \(\rho\) is a priori defined only \(\dd\vol_{\Man}\)-almost everywhere on \(\Man\). The
local regularity condition in item~\ref{ass:standing:density} fixes a \(C^r\) representative in the
chosen corner chart. When \(\Compct_{\Strat}\) lies in a boundary or corner stratum, the values
\(\rho(\Phi(\theta,0))\) in item~\ref{ass:standing:positivity} are understood through the trace of
this representative on the boundary faces of the chart.

The four items play different roles. The chart hypothesis \ref{ass:standing:chart} is purely
geometric and is used to construct adapted frames, tubular coordinates, and translated corner
charts. The measure-representation \ref{ass:standing:measure} and density-regularity hypotheses
\ref{ass:standing:density} enter through the local amplitude \(\mathcal A=\rho J\), the coefficients
\(\mathsf C_0\) and \(\mathsf C_1\), and the local kernel expansion. The positivity hypothesis
\ref{ass:standing:positivity} is used to pass from the density expansion to the logarithmic
expansion, through the uniform lower bound on \(\mathsf C_0\). The measure item
\ref{ass:standing:measure}, but not the local density regularity \ref{ass:standing:density}, is used
to control the far-field contribution to the full heat regularization \(p_\sigma\): what enters
there is that \(q\) is a probability measure on \(\Man\). Near \(\Compct_{\Strat}\), the positivity
hypothesis \ref{ass:standing:positivity} makes \(\Man\) agree locally with the support of \(q\).

The integer \(r\) measures the smoothness of the density in local corner coordinates and is the
regularity that limits the order of the asymptotic expansions. Throughout the main results we assume
one additional derivative on the corner chart, namely \(\Phi\in C^{r+1}\). This one-derivative
surplus is needed because the geometric quantities constructed from the chart (adapted
frames, Jacobian factors, and tubular coordinates) depend on derivatives of \(\Phi\).

The contractibility of \(\Theta_{\mathrm{out}}\) will be used only to choose adapted frames along
the stratum piece \(\Strat_{\mathrm{out}}\). The standing assumption above is sufficient for
constructing the local coordinate package, the localized kernel, and the tangent-cone expansion.

\begin{remark}[Locality and globalization]
  \label{rem:globalization}
  \Cref{ass:standing}~\ref{ass:standing:chart} covers \(\Compct_{\Strat}\) by a single corner chart
  over a contractible \(\Theta_{\mathrm{out}}\). This is a local hypothesis: it does not, for
  instance, cover an entire closed boundary circle by one chart. Accordingly, all expansions below
  are stated locally and hold locally uniformly over \(\Compct_{\Strat}\). No generality is lost. Indeed, on the one hand, one easily checks that the coefficients of the expansions 
  are intrinsic (see \Cref{rem:invariance} in the Appendix). On the other hand, a general compact subset of a
  stratum is handled by covering it with finitely many adapted corner charts, applying the results
  on each, and using that on overlaps the intrinsic coefficients agree; the finitely many uniform
  constants are then combined by taking maxima. We therefore state everything for a single chart and
  regard the globalization as routine.
\end{remark}

\subsection{Boundary-layer coordinates}

Near the stratum \(\Strat\), two transverse structures have to be kept separate. The first consists
of the directions in \(\mathcal C_x=\mathcal T_x\Man\cap (\mathcal T_x\Strat)^\perp\), which are
tangent to \(\Man\) but transverse to the stratum. These directions carry the boundary or corner
structure of the support. The second consists of the ambient normal directions \(\mathcal
N_x\Man\subset\mathbb R^d\), which encode the possible lower-dimensionality of the support. As shown
later (see \Cref{lem:ad_frame_comp_strat}), \Cref{ass:standing}~\ref{ass:standing:chart} allows us
to choose matrix-valued \(C^{r}\) mappings \(\bfS\), \(\bfC\), and \(\bfN\), defined on
\(\Strat_{\mathrm{out}}\), such that, for every \(x\in\Strat_{\mathrm{out}}\), the matrices
\(\bfS(x)\), \(\bfC(x)\), and \(\bfN(x)\) have dimensions \(d\times(m-c)\), \(d\times c\), and
\(d\times k\), respectively, and have orthonormal columns. The columns of these matrices are
orthonormal bases of \(\mathcal T_x\Strat\), \(\mathcal C_x\), and \(\mathcal N_x\Man=\mathcal
T_x\Man^\perp\), respectively. In particular, the concatenated matrix \(\bfQ(x)=[\,\bfS(x)\ \bfC(x)\
\bfN(x)\,]\) belongs to \(\mathsf O(d)\).

Thus, for \(x\in\Strat\), we write boundary-layer points as
\begin{equation}\label{eq:ysigma}
  y_\sigma = y_\sigma(a,x) = x+\sigma \bfC(x)a_{\mathcal C}+\sigma \bfN(x)a_{\mathcal N},
  \qquad a=(a_{\mathcal C},a_{\mathcal N})\in\mathbb R^c\times\mathbb R^k .
\end{equation}
Keeping \(a\) bounded is equivalent to looking in an \(O(\sigma)\)-neighborhood of the active
stratum, see \Cref{fig:boundary-layer-coordinates}. The limiting model is obtained by zooming in on
this \(O(\sigma)\)-neighborhood and rescaling distances by \(\sigma^{-1}\). In the smooth interior
of \(\Man\), this rescaling replaces the support by its tangent plane. Near a boundary or a corner,
the same rescaling leads to a limiting model that retains the one-sided constraints, replacing the
tangent plane by the inward tangent cone at \(\Strat\). Throughout, \(\pi(y)\) denotes the
orthogonal projection of \(y\) onto \(\Strat\). By \Cref{ass:standing}~\ref{ass:standing:chart}, and
the compactness of \(\Compct_{\Strat}\), after shrinking the tubular neighborhood if necessary, we
fix an open set \(\mathcal U\) on which the projection \(\pi:\mathcal U\to\Strat_{\mathrm{out}}\) is
well-defined and single-valued. Accordingly, for \(A>0\) and \(\sigma_0>0\) small enough, we denote
this boundary-layer regime by
\begin{equation}\label{eq:boundary-layer-regime}
  \mathcal Y_{A,\Compct_{\Strat},\sigma_0}
  = \big\{(y,\sigma)\in\mathcal U\times(0,\sigma_0]: \pi(y)\in \Compct_{\Strat},\
  \|y-\pi(y)\|\leqslant A\sigma\big\}.
\end{equation}

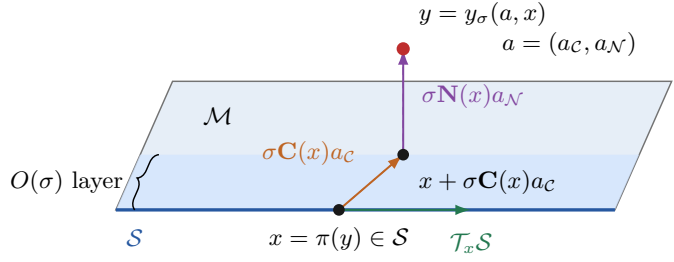
\begin{figure}[ht]
  \centering
  \vspace*{-15pt}
  \begin{tikzpicture}[scale=1.0]

  \coordinate (A) at (-3.30,-0.55);
  \coordinate (B) at ( 3.30,-0.55);
  \coordinate (C) at ( 4.05, 1.15);
  \coordinate (D) at (-2.55, 1.15);
  \coordinate (A1) at (-2.98, 0.18);
  \coordinate (B1) at ( 3.62, 0.18);

  \filldraw[bl-manifold] (A)--(B)--(C)--(D)--cycle;
  \fill[bl-layer] (A)--(B)--(B1)--(A1)--cycle;
  \draw[bl-stratum] (A)--(B);

  \coordinate (x)  at (-0.35,-0.55);
  \coordinate (xc) at ( 0.50, 0.18);
  \coordinate (y)  at ( 0.50, 1.58);

  \draw[bl-dash,BLCcolor!85!black] (x)--(xc);
  \draw[bl-dash,BLNcolor!85!black] (xc)--(y);

  \draw[bl-vec,BLTcolor] (x)--++(1.75,0)
      node[pos=1.0,below=5pt,bl-lab] {$\mathcal T_x\Strat$};
  \draw[bl-vec,BLCcolor] (x)--(xc)
      node[pos=0.52,above left=2pt,bl-lab] {$\sigma \bfC(x)a_{\mathcal C}$};
  \draw[bl-vec,BLNcolor] (xc)--(y)
      node[pos=0.55,right=3pt,bl-lab] {$\sigma \bfN(x)a_{\mathcal N}$};

  \node[bl-dot,label={[bl-lab,below=5pt]$x=\pi(y)\in \Strat$}] at (x) {};
  \node[bl-dot] at (xc) {};
  \node[bl-lab,below right=2pt] at (xc) {$x+\sigma \bfC(x)a_{\mathcal C}$};
  \node[bl-ydot,label={[bl-lab,above right=2pt]$y=y_\sigma(a,x)$}] at (y) {};

  \draw[decorate,decoration={brace,amplitude=4pt},line width=0.55pt]
      ($(B)-(6.35,0)$)--($(B1)-(6.35,0)$)
      node[midway,left=5pt,bl-lab] {$O(\sigma)$ layer};

  \node[bl-lab] at (-1.95,0.70) {$\Man$};
  \node[bl-lab,text=BLScolor] at (-3.05,-0.92) {$\Strat$};
  \node[bl-lab,align=left] at (2.70,1.65) {$a=(a_{\mathcal C},a_{\mathcal N})$};

  \end{tikzpicture}
  \vspace*{-8pt}
  \caption{\small Boundary-layer coordinates near \(\Strat\). The point \(x=\pi(y)\in\Strat\)
  records the location along the stratum. Coordinates \(a_{\mathcal C}\) and \(a_{\mathcal N}\)
  record transverse displacements: \(a_{\mathcal C}\) lies in the directions tangent to \(\Man\) but
  transverse to \(\Strat\), \(a_{\mathcal N}\) lies in the ambient-normal directions.}
  \label{fig:boundary-layer-coordinates}
  \vspace*{-15pt}
\end{figure}

To describe the coefficients of the expansion, let us recall that
\(p_\sigma(y)=\int_\Man\phi_\sigma(y-x') \,q(\dd x')\) and that the Gaussian kernel \(\phi_\sigma\)
contains the exponential of \(-\|y-x'\|^2/(2\sigma^2)\). For \(x=\pi(y)\in\Compct_{\Strat}\), thanks
to \Cref{ass:standing}~\ref{ass:standing:chart}, there is a unique
\(\theta\in\Theta_{\mathrm{out}}\) such that \(x=\Phi(\theta,0)\). To expand quantities of the form
\(F(x')\) appearing from the integration over \(x'\) from a local neighborhood of \(x\) in \(\Man\),
we use the translated corner chart
\begin{equation*}
  \Phi_x: \underbrace{ (\Theta_{\mathrm{out}}-\theta)\times[0,\varepsilon)^c }_{\mathcal U_x}
  \to \Man, \quad \Phi_x(\xi) = \Phi(\theta+\xi_{\mathcal S},\xi_{\mathcal C}),
  \quad \xi=(\xi_{\mathcal S},\xi_{\mathcal C}) \in\mathcal U_x\subset \mathbb H_c^m .
\end{equation*}
The two coordinate systems constructed above are summarized in \Cref{fig:obs-int-coord}. The
displacement of a point \(x'=\Phi_x(\xi)\in\Man\) from \(x\) can be expressed in the
frame \(\bfQ(x)\) by
\begin{equation}\label{eq:gh}
  \bfQ(x)^\top\bigl(\Phi_x(\xi)-x\bigr) = \bigl(g_x(\xi),h_x(\xi)\bigr)
  \in\mathbb R^m\times\mathbb R^k .
\end{equation}
Here \(g_x\) is the tangent component and \(h_x\) is the ambient normal component. At first order,
\(g_x\) is governed by the matrix \(\bfL\), while \(h_x\) has no linear part.

\begin{figure}[ht]
  \centering
  \begin{tikzpicture}[scale=1.0]

  \begin{scope}[shift={(-3.95,0)}]
   \node[bl-panel-title] at (0,1.75) {observation point};

   \coordinate (SA) at (-2.00,-0.75);
   \coordinate (SB) at ( 2.00,-0.75);
   \coordinate (MA) at (-1.70,-0.75);
   \coordinate (MB) at ( 1.85,-0.75);
   \coordinate (MC) at ( 2.25,0.35);
   \coordinate (MD) at (-1.25,0.55);
   \filldraw[bl-manifold] (MA)--(MB)--(MC)--(MD)--cycle;
   \draw[bl-stratum] (SA)--(SB);

   \coordinate (x) at (-0.30,-0.75);
   \coordinate (xu) at (0.45,0.00);
   \coordinate (y) at (0.45,1.20);
   \draw[bl-dash,BLCcolor!80!black] (x)--(xu);
   \draw[bl-dash,BLNcolor!80!black] (xu)--(y);
      \draw[bl-vec,BLCcolor] (x)--(xu)
      node[pos=0.52,above left=1pt,bl-smalllab] {$\sigma\bfC(x)a_{\mathcal C}$};
      \draw[bl-vec,BLNcolor] (xu)--(y)
      node[pos=0.55,right=2pt,bl-smalllab] {$\sigma\bfN(x)a_{\mathcal N}$};

   \node[bl-dot,label={[bl-smalllab,below]$x=\pi(y)$}] at (x) {};
   \node[bl-ydot,label={[bl-smalllab,right]$y$}] at (y) {};
   \node[bl-smalllab,text=BLScolor] at (-1.85,-1.00) {$\Strat$};
   \node[bl-smalllab] at (2.15,0.52) {$\Man$};

      \node[bl-lab,align=center] at (0,-1.68)
      {\(y=y_\sigma(a,x)\), \(x=\pi(y)\)\\[2pt]
      \(a=\sigma^{-1}\nu(y)\)};
  \end{scope}

  \draw[black!20,line width=0.5pt] (0,-2.02)--(0,2.02);

  \begin{scope}[shift={(3.95,0)}]
   \node[bl-panel-title] at (0,1.75) {integration variable};

   \coordinate (Oxi) at (-2.05,-1.10);
            \coordinate (Oxl) at (-2.05-1.35,-1.10);
   \fill[black!6] (Oxl)--++(2.50,0)--++(0,1.00)--++(-2.50,0)--cycle;
   \draw[bl-vec,black!65] (-2.05-1.35,-1.10)--(-2.05+1.35,-1.10)
   node[pos=0.95,below=1pt,bl-smalllab] {$\xi_{\mathcal S}$};
   \draw[bl-vec,black!65] (Oxi)--++(0,1.15)
   node[pos=0.95,left=1pt,bl-smalllab] {$\xi_{\mathcal C}\geqslant 0$};
   \node[bl-dot,label={[bl-smalllab,below left=1pt]$0$}] at (Oxi) {};
   \coordinate (xi) at (-1.45,-0.26);
   \node[
    bl-dot,
    label={[bl-smalllab,above left=1pt,xshift=30pt,bl-whitebox]$\xi=(\xi_{\mathcal S},\xi_{\mathcal C})$}
   ] at (xi) {};
   \node[bl-smalllab] at (-1.40,-1.55) {$\mathbb H_c^m$};

   \draw[bl-vec,black!65] (-0.28,-0.46)--(0.60,-0.46)
   node[midway,below=2pt,bl-smalllab] {$\Phi_x$};

   \coordinate (A) at (0.95,-0.85);
   \coordinate (B) at (2.95,-0.85);
   \coordinate (C) at (3.35,0.35);
   \coordinate (D) at (1.35,0.65);
   \filldraw[bl-manifold] (A)--(B)--(C)--(D)--cycle;
   \draw[bl-stratum] (A)--(B);

   \coordinate (xR) at (1.55,-0.85);
   \coordinate (XR) at (2.28,0.05);
   \draw[bl-dash,black!45] (xR)--(XR);
      \node[bl-dot,label={[bl-smalllab, below]
      $x=\Phi(\theta,0)$}] at (xR) {};
   \node[bl-dot] at (XR) {};
   \node[bl-smalllab, right=1pt] at (XR) {$\Phi_x(\xi)$};
   \node[bl-smalllab] at (3.33,0.72) {$\Man$};

   \node[bl-lab,align=center] at (1.55,-1.68)
   {$\Phi_x(\xi)=\Phi(\theta+\xi_{\mathcal S},
            \xi_{\mathcal C})$\\[2pt]
   $\xi_{\mathcal C}\in[0,\infty)^c$};
  \end{scope}
  \end{tikzpicture}
  \vspace*{-5pt}
  \caption{\small Two coordinate systems used near the same stratum. The observation point is
  written in tubular coordinates relative to $\Strat$, whereas the integration variable in the
  heat-kernel integral is written in the moving corner chart \(\Phi_x(\xi)
  =\Phi(\theta+\xi_{\mathcal S},\xi_{\mathcal C})\). }
  \vspace*{-20pt}
  \label{fig:obs-int-coord}
\end{figure}
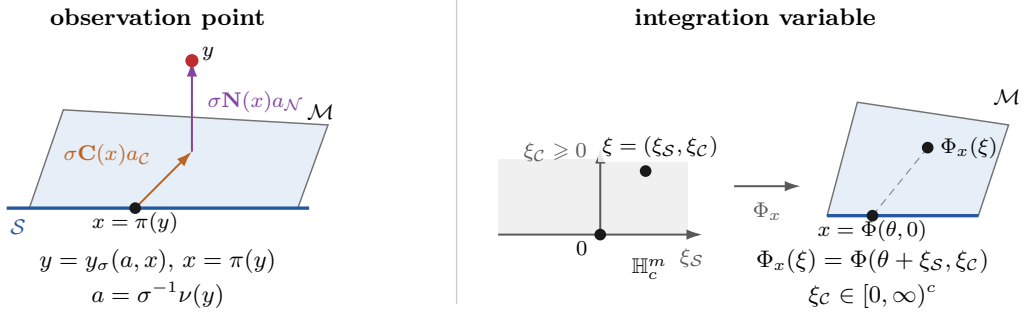

\subsection{The linearized cone coefficient}

Given that \([\,\bfS(x)\;\bfC(x)\,]\) is an adapted orthonormal tangent frame in \(\mathcal
T_x\Strat\oplus \mathcal C_x\), there is a matrix \(\bfL(x)=[\,\bfL_{\mathcal S}(x);\ \bfL_{\mathcal
C}(x)\,] \in\mathsf{GL}(m)\) such that \(\Diff\Phi_x(0)=\bfS(x)\bfL_{\mathcal S}(x)
+\bfC(x)\bfL_{\mathcal C}(x)\). In adapted tangent coordinates, the linearized inward cone is
\(\bfL(x)\mathbb H_c^m\). Thus, replacing the observation point \(y\) by \(y_\sigma(a,x)\) given by
\eqref{eq:ysigma}, yields the limiting quadratic exponent (\textit{i.e.}, the limit when
\(\sigma\downarrow0\) of \(\|y-x'\|^2/(2\sigma^2)\))
\begin{equation}\label{eq:Psi0}
  \Psi(\zeta;a,x) = \frac12 \left[ \|\bfL_{\mathcal S}(x)\,\zeta\|^2
  + \|a_{\mathcal C}-\bfL_{\mathcal C}(x)\,\zeta\|^2 + \|a_{\mathcal N}\|^2 \right].
\end{equation}
Indeed, in adapted coordinates, the linearized support has stratum component \(\bfL_{\mathcal
S}(x)\,\zeta\), transverse tangent component \(\bfL_{\mathcal C}(x)\,\zeta\), and no ambient-normal
component. The observation point has corresponding transverse coordinates \(a_{\mathcal C}\) and
\(a_{\mathcal N}\).

At leading order, the density \(\rho\) is replaced by \(\rho(x)\), the volume element contributes
the Jacobian factor \(|\det \bfL(x)|\), and the Gaussian factor contributes
\(e^{-\Psi(\zeta;a,x)}\), leading to \(p_\sigma\bigl(y_\sigma(a,x)\bigr)\approx
\sigma^{-k}(2\pi)^{-d/2}\mathsf C_0(a,x)\) with
\begin{equation}\label{eq:C0}
  \mathsf C_0(a,x) = \rho(x)|\det \bfL(x)| \int_{\mathbb H_c^m} e^{-\Psi(\zeta;a,x)} \,\dd\zeta .
\end{equation}
This coefficient is the central geometric quantity. The factor \(\rho(x)\) records the density of
\(q\) along the support, \(|\det \bfL(x)|\) is the leading volume Jacobian in adapted tangent
coordinates, and the integral is the Gaussian mass of the linearized inward cone as seen from the
rescaled observation point \(a\). The dependence on \(a_{\mathcal C}\) captures boundary and corner
effects, while the dependence on \(a_{\mathcal N}\) captures the ambient codimension of the support.

Although \(\mathsf C_0\) is written through the chart datum \(\bfL(x)\) and the frames
\(\bfC(x),\bfN(x)\), the coefficient it defines is intrinsic --- independent of the adapted corner
chart and equivariant under rotations of the adapted frames --- and the same holds for the
correction coefficient \(\mathsf C_1\) introduced below. This chart-independence, which underlies the
globalization of \Cref{rem:globalization}, is established in \Cref{rem:invariance}.

\begin{theorem}[First-order expansion for the heat regularization]
  \label{thm:first_coef_kernel}
  Let \(A>0\) and assume \Cref{ass:standing} with \(r\geqslant 1\). Then there exists \(\sigma_0 >
  0\) such that, for every \((y,\sigma)\in\mathcal Y_{A,\Compct_{\Strat},\sigma_0}\), \(x=\pi(y)\),
  \(a = a(y,\sigma)\), it holds that
  \begin{align*}
    p_\sigma(y) &= \sigma^{-k}(2\pi)^{-d/2} \left[ \mathsf C_0(a,x)
    + \sigma\mathscr E_0(a,x,\sigma) \right], \\
    \log p_\sigma(y) &= - k\log\sigma - (d/2)\log (2\pi) + \mathsf L_0(a,x)
    + \sigma\mathscr E_{0, \log} (a,x,\sigma),
  \end{align*}
  where \(\mathsf L_0(a,x)=\log \mathsf C_0(a,x)\), and \(\mathscr E_0(a,x,\sigma)\), \(\mathscr
  E_{0,\log}(a,x,\sigma)\) have all mixed derivatives up to order \(r-1\), in the variables \((a,
  x,\sigma)\), uniformly bounded for \(x\in\Compct_{\Strat}\), \(\|a\|\leqslant A\), and
  \(0<\sigma\leqslant\sigma_0\). Derivatives in \(x\) are computed in the stratum coordinate
  \(\theta\), after writing \(x=\Phi(\theta,0)\); throughout, the precise meaning of these mixed
  derivatives in \((a,x,\sigma)\) --- over the closed ball in \(a\), and by pullback in \(\theta\)
  --- is fixed in \Cref{rem:parameter_variable_derivatives}.
\end{theorem}

\subsection{The first correction coefficient}

The full expansion has a first correction beyond this linearized cone model, expressed through the
coefficient \(\mathsf C_1\) defined below in \eqref{eq:def_C1}. As expected, its expression involves
higher-order differentials of the chart \(\Phi\), the density \(\rho\), and related quantities which
we now define precisely. The Jacobian of the local corner chart \(\Phi_x\) and the corresponding
amplitude are defined by
\begin{equation}\label{eq:defAx}
  J_x(\xi)
  = \sqrt{ \det\bigl( \Diff \Phi_x(\xi)^\top
  \Diff \Phi_x(\xi) \bigr)},\quad J_x(0) = |\det \bfL(x)|,
  \quad \mathcal A_x(\xi)=\rho\bigl(\Phi_x(\xi)\bigr) \,J_x(\xi).
\end{equation}
They both belong to the class \(C^{r}\). We also define the quadratic chart coefficients by
\begin{equation*}
  \mathfrak g_{2,x}[\zeta,\zeta] = \Diff^2 g_x(0)[\zeta,\zeta]\in \mathbb R^m,
  \qquad \mathfrak h_{2,x}[\zeta,\zeta] = \Diff^2 h_x(0)[\zeta,\zeta]\in \mathbb R^k,
\end{equation*}
where \(g_x\) and \(h_x\) are given by \eqref{eq:gh}. The coefficient \(\mathfrak g_{2,x}\) measures
the quadratic correction to the tangent component of the chart, while \(\mathfrak h_{2,x}\) measures
the first nonzero normal displacement of \(\Man\) away from its tangent space. Since \(g_x\) and
\(h_x\) are \(C^{r+1}\) in \(\xi\), these coefficients are \(C^{r-1}\) in the stratum coordinates.
The first correction to the rescaled Gaussian phase is encoded by
\begin{equation}\label{eq:def_Lambda}
  \Lambda_x(\zeta;a) = \langle \bfL(x)\zeta, \mathfrak g_{2,x}[\zeta,\zeta] \rangle
  - \langle a_{\mathcal C}, \Pi_{\mathcal C}\mathfrak g_{2,x}[\zeta,\zeta] \rangle
  - \langle a_{\mathcal N}, \mathfrak h_{2,x}[\zeta,\zeta] \rangle ,
\end{equation}
where \(\Pi_{\mathcal C}:\mathbb R^m\to\mathbb R^c\) is the projection onto the last \(c\) adapted
tangent coordinates. More precisely, after the rescaling \(\xi=\sigma\zeta\), the Gaussian phase
satisfies
\begin{equation*}
  \frac{\|y_\sigma(a,x)-\Phi_x(\sigma\zeta)\|^2}{2\sigma^2} = \Psi(\zeta;a,x)
  + \frac{\sigma}{2}\Lambda_x(\zeta;a) + O(\sigma^2).
\end{equation*}
Thus \(\Lambda_x\) records how the quadratic Taylor terms of the chart modify the leading
linearized-cone Gaussian kernel. The first correction coefficient is obtained by integrating this
phase correction, together with the first variation of the amplitude \(\mathcal
A_x=\rho\circ\Phi_x\,J_x\), over the limiting cone \(\mathbb H_c^m\). For \(x\in\Compct_{\Strat}\) and
\(a=(a_{\mathcal C},a_{\mathcal N})\in \mathbb R^c\times\mathbb R^k\), define
\begin{equation}\label{eq:def_C1}
  \mathsf C_1(a,x)
  = \int_{\mathbb H_c^m} \langle \nabla_\xi \mathcal A_x(0) ,\zeta\rangle e^{-\Psi(\zeta;a,x)}
  \,\dd\zeta
  - \frac12 \rho(x)\,|\det \bfL(x)| \int_{\mathbb H_c^m} e^{-\Psi(\zeta;a,x)} \Lambda_x(\zeta;a)
  \,\dd\zeta .
\end{equation}
The coefficient \(\mathsf C_1\) collects the corrections of order \(\sigma\) to the leading
tangent-cone integral. These corrections have two sources. First, the local parametrization is not
exactly equal to its linearization: \(\Phi_x(\sigma\zeta)\) differs from \(x+\sigma
\Diff\Phi_x(0)\zeta\) by a quadratic Taylor term, and this changes the Gaussian exponent. Second,
the amplitude in the local integral, namely the product of the density \(\rho\) and the volume
Jacobian \(J_x\), is not exactly equal to its value \(\rho(x)|\det \bfL(x)|\) at the stratum point.
Its first-order variation contributes to \(\mathsf C_1\) as well.

\begin{theorem}[Second-order expansion for the heat regularization]
  \label{thm:2nd_coef_kernel}
  Let \(A>0\) and assume \Cref{ass:standing} with \(r\geqslant 2\). Then there exists \(\sigma_0 >
  0\) such that, for every \((y,\sigma)\in\mathcal Y_{A,\Compct_{\Strat},\sigma_0}\), \(x=\pi(y)\),
  \(a = a(y,\sigma)\), it holds that
  \begin{align*}
    p_\sigma(y) &= \sigma^{-k} (2\pi)^{-d/2} \left[ \mathsf C_0(a,x) + \sigma\mathsf C_1(a,x)
    + \sigma^2\mathscr E_1 (a,x,\sigma) \right], \\
    \log p_\sigma(y) &= -k\log\sigma - (d/2) \log(2\pi) + \mathsf L_0(a,x) + \sigma\mathsf L_1(a,x)
    + \sigma^2\mathscr E_{1,\log} (a,x,\sigma),
  \end{align*}
  where \(\mathsf L_0(a,x)=\log \mathsf C_0(a,x)\), \(\mathsf L_1(a,x)=\mathsf C_1(a,x)/\mathsf
  C_0(a,x)\), and \(\mathscr E_1(a,x,\sigma)\), \(\mathscr E_{1,\log}(a,x,\sigma)\) have all mixed
  derivatives up to order \(r-2\), in the variables \(a\), \(x\), and \(\sigma\), uniformly bounded
  for \(x\in\Compct_{\Strat}\), \(\|a\|\leqslant A\), and \(0<\sigma\leqslant\sigma_0\). Derivatives
  in \(x\) are computed in the stratum coordinate \(\theta\), after writing \(x=\Phi(\theta,0)\).
\end{theorem}

\subsection{Expansions of the score and its derivatives}

To derive expansions for the score, the log-Hessian and the scale-derivative of the score, we
differentiate the logarithmic conical-layer expansion obtained in
\Cref{thm:first_coef_kernel,thm:2nd_coef_kernel}. The singular powers of \(\sigma\) appear when one
returns from the boundary-layer variables \((x,a)\) to the original observation variable \(y\).
Indeed, by the tubular-coordinate construction \eqref{eq:ysigma}, the rescaled transverse coordinate
satisfies \(a(y,\sigma)=\sigma^{-1}\nu(y)\), where \(\nu(y)\) is the unscaled transverse coordinate
of \(y\) relative to the stratum. Thus \(\Diff_y a(y,\sigma)\) is of order \(\sigma^{-1}\), and
\(\partial_\sigma a(y, \sigma)=-\sigma^{-1}a(y,\sigma)\). Consequently, first derivatives in \(y\)
may produce one factor of \(\sigma^{-1 }\), while second derivatives may produce two.

The expansions of the score and its derivatives contain Jacobian factors originating from the change
of the variables \((a,x)\rightsquigarrow y_\sigma(a,x)\) given by \eqref{eq:ysigma}. Let
\(\theta(y)\) be the unique element of \(\Theta_{\rm out}\) such that
\begin{equation}\label{eq:theta-nu-def}
  (\theta(y),0_c)=\Phi^{-1}(\pi(y)).
\end{equation}
We also recall that \(\nu(y)\in \mathbb R^{c+k}\) is given by \( y = \pi(y) + [\,\bfC(\pi(y))\
\bfN(\pi (y))\,] \,\nu(y)\) so that \(a(y,\sigma)=\nu(y)/\sigma\). Here \(\theta(y)\) is the local
coordinate of the projection \(\pi(y)\) on the stratum and \(\nu(y)\) collects the unscaled tubular
coordinates transverse to the stratum. The singular powers of \(\sigma\) come from this change of
variables. Let
\begin{equation}\label{def:J1}
  J_\nu(y)=\Diff_y\nu(y)\in\mathbb R^{(c+k)\times d},
  \qquad J_\theta(y)=\Diff_y\theta(y)\in\mathbb R^{(m-c)\times d}.
\end{equation}
Define the second-derivative tensors \(\mathsf Q^\nu_y\) and \(\mathsf Q^\theta_y\) by
\begin{equation*}
  \mathsf Q^\nu_y[v] = \sum_{\mu=1}^{c+k}v_\mu\,\nabla_y^2\nu_\mu(y)
  \in\mathbb R^{d\times d},\quad\text{and}\quad \mathsf Q^\theta_y[w]
  = \sum_{i=1}^{m-c}w_i\,\nabla_y^2\theta_i(y)\in \mathbb R^{d\times d},
\end{equation*}
for \(v\in\mathbb R^{c+k}\) and \(w\in\mathbb R^{m-c}\). It is convenient to express all coefficient
fields in the rescaled variables \(y = y_\sigma(a,x)\) given by \eqref{eq:ysigma} on the set
\(\|a\|\leqslant A\), \(x\in\Compct_{\Strat}\), \( 0<\sigma\leqslant\sigma_0\), and define
\begin{equation}\label{def:J2}
  \mathbf J_\nu(a,x,\sigma)=J_\nu(y_\sigma(a,x)),
  \qquad \mathbf J_\theta(a,x,\sigma)=J_\theta(y_\sigma (a,x)).
\end{equation}
Likewise, for \(v\in\mathbb R^{c+k}\) and \(w\in\mathbb R^{m-c}\), set
\begin{equation}\label{def:Q2}
  \mathbf Q_\nu(a,x,\sigma)[v] = \mathsf Q^\nu_{y_\sigma(a,x)}[v],
  \qquad \mathbf Q_\theta(a,x,\sigma)[w] = \mathsf Q^\theta_{y_\sigma(a,x)}[w].
\end{equation}
Under \Cref{ass:standing}, and for the values of \(r\) imposed in the statements below, these
reconstructed coefficient fields have the finite differentiability required there.

To better explain the origin of different coefficients appearing in the expansions presented in the
theorems below, we recall a chain rule, further justified in \Cref{app:logarithmic_differentiation}.
If \( F^\sharp(y, \sigma) = F\bigl(a(y,\sigma),\pi(y), \sigma\bigr)\), then \( \nabla_yF^\sharp =
J_\theta(y)^\top\nabla_\theta F + (1/\sigma) J_\nu(y)^\top\nabla_aF. \) There is a similar but more
involved rule for the Hessian. Above, all derivatives of \(F\) on the right are evaluated at
\((a,x,\sigma)=\bigl(a(y,\sigma),\pi(y),\sigma\bigr)\). This identity makes the scale separation
explicit: \(a\)-derivatives carry powers of \(\sigma^{-1}\), while \(\theta\)-derivatives do not. We
state the first- and second-order expansions separately because they require different orders of
smoothness.

The density and logarithmic coefficient fields \(\mathsf C_0,\mathsf C_1,\mathsf L_0,\mathsf L_1\)
depend only on \((a,x)\) and are independent of \(\sigma\). By contrast, the differentiated
coefficient fields \(\mathsf S_i,\mathsf H_i,\dot{\mathsf S}_i\) appearing below are written after
reconstructing the ambient observation point \(y=y_\sigma(a,x)\), and may therefore retain a
bounded, admissible \(\sigma\)-dependence through the fields \(\mathbf J_\theta,\mathbf
J_\nu,\mathbf Q_\theta,\mathbf Q_\nu\) of \eqref{def:J2} and \eqref{def:Q2}. This dependence is harmless for the
asymptotic expansion, and the retained \(\sigma\)-dependence serves only to preserve the uniform
derivative counts. Under \Cref{ass:standing} with \(r\geqslant1\), the maps \(J_\nu,J_\theta\) are
continuous on \(\mathcal U\) and \(y_\sigma(a,x)\to x\) as \(\sigma\downarrow0\), so these
reconstructed fields always have limits; evaluating at \(\sigma=0\) yields the
\(\sigma\)-independent leading profiles. For instance, \(\mathsf S_0(a,x,0)=\mathbf
J_\nu(a,x,0)^\top\nabla_a\mathsf L_0(a,x)\), and similarly for \(\mathsf H_0\) and \(\dot{\mathsf
S}_0\).

\begin{theorem}[First-order expansions for the score and its derivatives]
  \label{thm:first-order-score}
  Let \(A>0\) and assume \Cref{ass:standing} with \(r\geqslant r_0\) (the value of \(r_0\) is
  specified below for each expansion). Then there exists \(\sigma_0 > 0\) such that, for every
  \((y,\sigma)\in\mathcal Y_{A,\Compct_{\Strat},\sigma_0}\), \(x=\pi(y)\), \(a = a(y,\sigma)\), it
  holds that
  \begin{align*}
    \score(y) &= \sigma^{-1}\big(\mathsf S_0(a,x,\sigma)
    + \sigma\mathscr R_{0,\score}(a,x,\sigma)\big), \qquad r_0=2 \\
    \Hess(y) &= \sigma^{-2} \big( \mathsf H_0(a,x,\sigma)
    + \sigma\mathscr R_{0, \Hess}(a,x,\sigma)\big), \qquad r_0=3 \\
    \Vel(y) &= \sigma^{-2} \big(\dot{\mathsf S}_0 (a,x,\sigma)
    + \sigma \mathscr R_{0,\Vel}(a,x,\sigma)\big), \qquad r_0=3.
  \end{align*}
  Here, the coefficients \(\mathsf S_0\), \(\mathsf H_0\) and \(\dot{\mathsf S}_0\) are explicitly
  given by the formulas below and have all mixed derivatives up to order \(r-r_0+1\) in the
  variables \(a\), \(x\), and \(\sigma\), uniformly bounded for \(x\in\Compct_{\Strat}\), \(\|a\|
  \leqslant A\), and \(0<\sigma\leqslant\sigma_0\):
  \begin{align*}
    \mathsf S_0 (a,x,\sigma) &= \mathbf J_\nu^\top (a,x,\sigma)\nabla_a \mathsf L_0(a,x), \\
    \mathsf H_0 (a,x,\sigma) &= \mathbf J_\nu^\top(a,x,\sigma)
    \Diff_a^2 \mathsf L_0(a,x) \,\mathbf J_\nu(a,x,\sigma), \\
    \dot{\mathsf S}_0(a,x,\sigma) &= -\mathsf S_0 (a,x, \sigma)
    - (\Diff_a\mathsf S_0(a,x,\sigma))[a].
  \end{align*}
  The functions \(\mathscr R_{0,\score}\), \(\mathscr R_{0,\Hess}\) and \(\mathscr R_{0,\Vel}\) have
  all mixed derivatives up to order \(r-r_0\) in the variables \(a\), \(x\), and \(\sigma\),
  uniformly bounded for \(x\in\Compct_{\Strat}\), \(\|a\| \leqslant A\), and
  \(0<\sigma\leqslant\sigma_0\). Derivatives in \(x\) are computed in the stratum coordinate
  \(\theta\), after writing \(x=\Phi(\theta,0)\).
\end{theorem}

We remark that the most singular parts of the score and log-Hessian are determined entirely by the
Gaussian mass of the linearized inward tangent cone. Curvature of the support and variation of the
density enter only at lower orders.

\begin{theorem}[Second-order expansions for the score and its derivatives]
  \label{thm:2nd-order-score}
  Let \(A>0\) and assume \Cref{ass:standing} with \(r\geqslant r_0\) (the value of \(r_0\) is
  specified below for each expansion). Then there exists \(\sigma_0 > 0\) such that, for every
  \((y,\sigma)\in\mathcal Y_{A,\Compct_{\Strat},\sigma_0}\), \(x=\pi(y)\), \(a = a(y,\sigma)\), it
  holds that
  \begin{align*}
    \score(y) &= \sigma^{-1}\big(\mathsf S_0(a,x,\sigma) + \sigma \mathsf S_1(a,x,\sigma)
    + \sigma^2\mathscr R_{1,\score}(a,x,\sigma)\big), \qquad r_0=3 \\
    \Hess(y) &= \sigma^{-2} \big( \mathsf H_0(a,x,\sigma) + \sigma\mathsf H_1(a,x, \sigma)
    + \sigma^2\mathscr R_{1,\Hess}(a,x,\sigma)\big), \qquad r_0=4 \\
    \Vel(y) &= \sigma^{-2} \big(\dot{\mathsf S}_0 (a,x,\sigma)
    + \sigma \dot{\mathsf S}_1(a,x,\sigma) + \sigma^2\mathscr R_{1,\Vel}(a,x,\sigma)\big),
    \qquad r_0=4.
  \end{align*}
  Here, the coefficients \(\mathsf S_0\), \(\mathsf S_1\), \(\mathsf H_0\), \(\mathsf H_1\),
  \(\dot{\mathsf S}_0\) and \(\dot{\mathsf S}_1\) are explicitly given in
  \Cref{thm:first-order-score} or by the formulas below and have all mixed derivatives up to order
  \(r-r_0+1\) in the variables \(a\), \(x\), and \(\sigma\), uniformly bounded for
  \(x\in\Compct_{\Strat}\), \(\|a\| \leqslant A\), and \(0<\sigma\leqslant\sigma_0\).
  \begin{align*}
    \mathsf S_1 (a,x,\sigma)
    &= \mathbf J_{\theta}^\top (a,x,\sigma) \nabla_{\theta}\mathsf L_0 (a,x)
    + \mathbf J_\nu^\top(a,x,\sigma) \nabla_a\mathsf L_1(a,x), \\
    \mathsf H_1 (a,x,\sigma)
    &= \mathbf J_{\theta}^\top \Diff_{\theta a}^2\mathsf L_0\,\mathbf J_\nu
    + \mathbf J_\nu^\top \Diff_{a\theta}^2\mathsf L_0\, \mathbf J_{\theta}
    + \mathbf Q_\nu[\nabla_a\mathsf L_0]
    + \mathbf J_\nu^\top \Diff_a^2\mathsf L_1\,\mathbf J_\nu, \\
    \dot{\mathsf S}_1 (a,x,\sigma) &= \partial_\sigma{\mathsf
    S}_0(a,x,\sigma) - (\Diff_a\mathsf S_1(a,x,\sigma))[a].
  \end{align*}
  The functions \(\mathscr R_{1,\score}\), \(\mathscr R_{1,\Hess}\) and \(\mathscr R_{1,\Vel}\) have
  all mixed derivatives up to order \(r-r_0\) in the variables \(a\), \(x\), and \(\sigma\),
  uniformly bounded for \(x\in\Compct_{\Strat}\), \(\|a\| \leqslant A\), and
  \(0<\sigma\leqslant\sigma_0\). Derivatives in \(x\) are computed in the stratum coordinate
  \(\theta\), after writing \(x=\Phi(\theta,0)\).
\end{theorem}

In the formula for \(\mathsf H_1\) above, the arguments \((a,x,\sigma)\) and \((a,x)\) are omitted
on the right-hand side in order to avoid an overly long display.

%% file: sections/part_cases.tex

\section{Particular cases}
\label{sec:part_cases}

We spell out three limiting cases of the boundary-layer expansions. These cases are useful because
they separate ordinary heat smoothing, the effects caused by a lower-dimensional support, and the
one-sided behavior near a smooth boundary. In the first case, the support is all of \(\mathbb R^d\).
There is no transverse coordinate and no singular behavior. In the second case, there is no boundary
coordinate, but the ambient normal coordinate remains. This produces the singular normal terms in
the score, in the log-Hessian, and in the scale derivative of the score. In the third case, the
support is a full-dimensional ellipsoid, so only the boundary coordinate is present.

\paragraph{Smooth positive density on \(\mathbb R^d\).}

Suppose first that \(q\) has a smooth strictly positive density \(\rho\) with respect to Lebesgue
measure on \(\mathbb R^d\). In the notation of the boundary-layer setup this corresponds to
\(\Man=\mathbb R^d\), \(m=d\), \(c=0\), and \(k=0\). Thus there is no boundary coordinate
\(a_{\mathcal C}\), no ambient-normal coordinate \(a_{\mathcal N}\), no cone effect, and no singular
prefactor \(\sigma^{-k}\). The heat regularization is the usual Euclidean heat semigroup,
\(p_\sigma=e^{(\sigma^2/2)\Delta}\rho\).

Hence the boundary-layer expansion is not needed in this setting. A direct Taylor expansion of the
heat semigroup gives, on compact subsets and assuming enough smoothness,
\begin{equation*}
  p_\sigma(y) = \rho(y) + \frac{\sigma^2}{2}\Delta\rho(y) + O(\sigma^4).
\end{equation*}
Since \(\rho\) is strictly positive and continuous, it is bounded away from zero on every compact
subset. Therefore
\begin{equation*}
  \log p_\sigma(y) = \log\rho(y) + \frac{\sigma^2}{2} \frac{\Delta\rho(y)}{\rho(y)} + O(\sigma^4).
\end{equation*}
Differentiating the same expansion gives, again on compact subsets and with the corresponding
smoothness assumptions,
\begin{equation*}
  \score(y) = \scorezero(y) + \frac{\sigma^2}{2} \nabla\!\left(\frac{\Delta\rho}{\rho}\right)(y)
  + O(\sigma^4), \quad \Hess(y) = \Hesszero(y)
  + \frac{\sigma^2}{2} \nabla^2\!\left(\frac{\Delta\rho}{\rho}\right)(y) + O(\sigma^4),
\end{equation*}
where \(\scorezero(y) = \nabla\log\rho(y)\) and \(\Hesszero(y) = \nabla^2\log\rho(y)\). The
fixed-\(y\) scale derivative of the score is regular:
\begin{equation*}
  \Vel(y) = \sigma \nabla\!\left(\frac{\Delta\rho}{\rho}\right)(y) + O(\sigma^3).
\end{equation*}
These formulas are consistent with the coefficients of \Cref{sec:main_results}. Since \(c=k=0\), the
transverse variable \(a\) is empty. Thus \(\mathsf S_0\), \(\mathsf H_0\), and \(\dot{\mathsf S}_0\)
vanish. In Euclidean coordinates one may take \(\bfL=\bfI_d\), and
\begin{equation*}
  \mathsf C_0(x) = (2\pi)^{d/2}\rho(x), \qquad \mathsf L_0(x) = \frac d2\log(2\pi) + \log\rho(x).
\end{equation*}
The prefactor \((2\pi)^{-d/2}\) in the general formula then gives
\(p_\sigma(y)=\rho(y)+O(\sigma^2)\). The order-\(\sigma\) coefficient vanishes by the symmetry of
the full Gaussian tangent plane.

\paragraph{Smooth embedded manifold without boundary.}

Assume next that \(\Man\subset\mathbb R^d\) is a smooth \(m\)-dimensional embedded manifold without
boundary or corners, and that \(q(\dd x)=\rho(x)\,\dd\vol_{\Man}(x)\), with \(\rho>0\) on the
compact set under consideration. A simple example is the unit circle \(\mathbb S^1=\{x\in\mathbb
R^2:\|x\|=1\}\). The only stratum is \(\Strat=\Man\). Thus \(c=0\), \(\mathbb H_0^m=\mathbb R^m\),
\(k=d-m\), and \(a=a_{\mathcal N}\in\mathbb R^k\). There is no boundary coordinate \(a_{\mathcal
C}\). Points in the \(O(\sigma)\) normal layer around \(\Man\) have the form
\begin{equation*}
  y_\sigma(a,x) = x + \sigma\bfN(x)a, \qquad x\in\Man, \quad \|a\|\leqslant A.
\end{equation*}
Since the inward tangent cone is the full tangent space, the leading Gaussian integral is taken over
all of \(\mathbb R^m\):
\begin{equation*}
  \int_{\mathbb R^m} \exp\!\left(-\frac12\|\bfL(x)\zeta\|^2\right) \,\dd\zeta
  = \frac{(2\pi)^{m/2}}{|\det \bfL(x)|}.
\end{equation*}
Thus the determinant in the leading Jacobian cancels with the determinant produced by the Gaussian
integral, and
\begin{equation*}
  \mathsf C_0(a,x) = (2\pi)^{m/2} \rho(x) e^{-\|a\|^2/2}.
\end{equation*}
Consequently, \(\mathsf L_0(a,x) = \frac m2\log(2\pi) + \log\rho(x) - \frac12\|a\|^2\). In
particular, \(\nabla_a\mathsf L_0=-a\) and \(\Diff_a^2\mathsf L_0=-\bfI_k\). These identities are
the source of the leading normal attraction in the score and of the leading negative normal block in
the log-Hessian.

The first correction has a simple geometric meaning. In normal geodesic coordinates on \(\Man\), the
odd first-order variation of \(\rho\) integrates to zero over the full tangent space. The even
quadratic normal displacement of \(\Man\), encoded by the second fundamental form \(\mathrm{II}\) in
\eqref{eq:II}, survives through its trace \(\bsh_{\Man}\) defined in \eqref{eq:HM}. Hence
\begin{equation*}
  \mathsf L_1(a,x) = \frac12 \langle \bfN(x)a,\bsh_{\Man}(x)\rangle.
\end{equation*}
Combining the preceding identities with \Cref{thm:2nd_coef_kernel} gives
\begin{equation*}
  p_\sigma\bigl(x+\sigma\bfN(x)a\bigr) = \frac{\rho(x)e^{-\|a\|^2/2}}{(2\pi\sigma^2)^{k/2}} \left[ 1
  + \frac{\sigma}{2} \langle \bfN(x)a,\bsh_{\Man}(x)\rangle + O(\sigma^2) \right].
\end{equation*}
Equivalently,
\begin{equation*}
  \log p_\sigma\bigl(x+\sigma\bfN(x)a\bigr) = - \frac k2\log(2\pi\sigma^2) + \log\rho(x)
  - \frac12\|a\|^2 + \frac{\sigma}{2} \langle \bfN(x)a,\bsh_{\Man}(x)\rangle + O(\sigma^2).
\end{equation*}

The differentiated expansions can now be read from \Cref{thm:2nd-order-score}. Since
\(\nabla_a\mathsf L_0=-a\), the leading score coefficient satisfies \(\mathsf
S_0(a,x,0)=-\bfN(x)a\). The next coefficient has two intrinsic contributions: the tangential score
of the density and the mean-curvature bias. Thus
\begin{equation*}
  \score\bigl(x+\sigma\bfN(x)a\bigr) = -\sigma^{-1}\bfN(x)a + \nabla_{\Man}\log\rho(x)
  + (1/2)\bsh_{\Man}(x) + O(\sigma).
\end{equation*}
The term \(-\sigma^{-1}\bfN(x)a\) is the geometric
denoising term: at distance \(\sigma\|a\|\) from \(\Man\), the score points normally back toward the
support. The constant term follows the density along \(\Man\) and includes the mean-curvature
correction caused by integrating over a curved submanifold.

For the log-Hessian, recall that \(\bfP_{\mathcal N}(x)=\bfN(x)\bfN(x)^\top\) is the orthogonal
projection onto \(\mathcal N_x\Man\). The coefficient of the most singular term is \(-\bfP_{\mathcal
N}(x)\); this term comes from the Hessian of the normal Gaussian confinement. The first curvature
correction is the symmetric bilinear form \(\mathsf B_{a,x}\) defined by
\begin{equation*}
  \mathsf B_{a,x}(\tau_1+n_1,\tau_2+n_2) = \langle \bfN(x)a, \mathrm{II}_x(\tau_1,\tau_2)\rangle,
\end{equation*}
where \(\tau_i\in\mathcal T_x\Man\) and \(n_i\in\mathcal N_x\Man\). Equivalently, \(\mathsf
B_{a,x}\) is the self-adjoint ambient matrix whose normal-normal and tangent-normal blocks vanish
and whose tangent-tangent quadratic form is given by the last display. With the convention for
\(\mathrm{II}\) used in \eqref{eq:II},
\begin{equation*}
  \Hess\bigl(x+\sigma\bfN(x)a\bigr) = -\sigma^{-2}\bfP_{\mathcal N}(x) + \sigma^{-1}\mathsf B_{a,x}
  + O(1).
\end{equation*}
In particular, when \(a=0\), the displayed \(\sigma^{-1}\) curvature term vanishes.

Finally, consider the fixed-\(y\) scale derivative of the score. In the normal layer, the unscaled
normal coordinate is \(\eta(y)=\sigma a\), and \(y-\pi(y)=\bfN(x)\eta(y)=\sigma\bfN(x)a\). Holding
\(y\) fixed while differentiating with respect to \(\sigma\) therefore differentiates both the
explicit factor \(\sigma^{-1}\) and the rescaled coordinate \(a=\eta(y)/\sigma\). This yields
\begin{equation*}
  \Vel(y) = 2\sigma^{-2}\bfN(x)a + O(1) = 2\sigma^{-3}\bigl(y-\pi(y)\bigr) + O(1),
  \qquad y=x+\sigma\bfN(x)a.
\end{equation*}

\paragraph{Uniform law on a full-dimensional ellipsoid.}

Let \(\bfA\in\mathbb R^{d\times d}\) be symmetric positive definite and set
\begin{equation*}
  \mathcal E_{\bfA} = \{z\in\mathbb R^d:z^\top\bfA z\leqslant 1\},
  \qquad \Sigma_{\bfA}=\partial\mathcal E_{\bfA}.
\end{equation*}
Assume that \(q\) is the uniform law on \(\mathcal E_{\bfA}\). Its Lebesgue density on the support
is the constant
\begin{equation*}
  \rho_{\bfA} = \frac{1}{\operatorname{vol}_d(\mathcal E_{\bfA})}
  = \frac{\sqrt{\det\bfA}}{\omega_d}, \qquad \omega_d=\frac{\pi^{d/2}}{\Gamma(d/2+1)}.
\end{equation*}
For \(x\in\Sigma_{\bfA}\), the inward unit normal is \(-\bfA x/\|\bfA x\|\). We therefore
parametrize the \(O(\sigma)\) boundary layer by
\begin{equation*}
  y=y_\sigma(a,x) = x-\sigma a\frac{\bfA x}{\|\bfA x\|}, \qquad x\in\Sigma_{\bfA},
\end{equation*}
where positive \(a\) points inside the ellipsoid. This matches the general corner-chart
convention, in which the orthant constraint \(\xi_{\mathcal C}\geqslant 0\) selects the inward
normal direction, so the tangent half-space retains the argument \(+a\) and the leading mass is
\(\Phi_{\rm N}(a)\) rather than \(\Phi_{\rm N}(-a)\); the sign of \(a_{\mathcal C}\) is otherwise a
frame choice. This is the case \(m=d\), \(c=1\), and \(k=0\).
Let
\begin{equation*}
  \phi_{\rm N}(a)=(2\pi)^{-1/2}e^{-a^2/2},
  \qquad \Phi_{\rm N}(a)=\int_{-\infty}^a\phi_{\rm N}(t)\,\dd t,
  \qquad \lambda(a)=\frac{\phi_{\rm N}(a)}{\Phi_{\rm N}(a)}.
\end{equation*}
The leading coefficient of \Cref{thm:first_coef_kernel} is
\begin{equation*}
  \mathsf C_0(a,x) = \rho_{\bfA}(2\pi)^{d/2}\Phi_{\rm N}(a), \qquad \mathsf L_0(a,x)
  = \log\rho_{\bfA} +\frac d2\log(2\pi) +\log\Phi_{\rm N}(a).
\end{equation*}
Hence, uniformly for bounded \(a\) and \(x\in\Sigma_{\bfA}\),
\begin{equation*}
  p_\sigma\bigl(y_\sigma(a,x)\bigr) = \rho_{\bfA}\Phi_{\rm N}(a)+O(\sigma),
  \qquad \log p_\sigma\bigl(y_\sigma(a,x)\bigr) = \log\rho_{\bfA}+\log\Phi_{\rm N}(a)+O(\sigma).
\end{equation*}
The first-order differentiated coefficients of \Cref{thm:first-order-score}, which here are
independent of \(\sigma\), are
\begin{equation*}
  \mathsf S_0(a,x) = -\lambda(a)\frac{\bfA x}{\|\bfA x\|}, \quad \mathsf H_0(a,x)
  = \lambda'(a) \frac{(\bfA x)(\bfA x)^\top}{\|\bfA x\|^2}, \quad \dot{\mathsf S}_0(a,x)
  = \bigl[\lambda(a)+a\lambda'(a)\bigr] \frac{\bfA x}{\|\bfA x\|},
\end{equation*}
where \(\lambda'(a)=-\lambda(a)\bigl(a+\lambda(a)\bigr)\). Since \(\lambda=\phi_{\rm N}/\Phi_{\rm
N}>0\) and \(\lambda'(a)=(\log\Phi_{\rm N})''(a)<0\) by the log-concavity of \(\Phi_{\rm N}\), the
rank-one Hessian \(\mathsf H_0\preceq 0\). Equivalently,
\begin{align*}
  \score\bigl(y_\sigma(a,x)\bigr) &= -\sigma^{-1}\lambda(a)\frac{\bfA x}{\|\bfA x\|} +O(1), \\
  \Hess\bigl(y_\sigma(a,x)\bigr)
  &= \sigma^{-2}\lambda'(a) \frac{(\bfA x)(\bfA x)^\top}{\|\bfA x\|^2} +O(\sigma^{-1}), \\
  \Vel\bigl(y_\sigma(a,x)\bigr)
  &= \sigma^{-2}\bigl[\lambda(a)+a\lambda'(a)\bigr] \frac{\bfA x}{\|\bfA x\|} +O(\sigma^{-1}).
\end{align*}

These formulas show that the leading boundary behavior is universal: after zooming at scale
\(\sigma\), the ellipsoid is replaced by its tangent half-space, the density is multiplied by the
Gaussian mass \(\Phi_{\rm N}(a)\), and the score becomes a singular inward normal attraction of size
\(\sigma^{-1}\lambda(a)\). The geometry of the ellipsoid appears only through the local outward
normal \(\bfA x/\|\bfA x\|\), with curvature contributing at lower order.

%% file: sections/proof_sketch.tex

\section{Sketch of the proofs of the main results}
\label{sec:sketch_proofs}

In this section, we explain the main steps of the proofs of the second-order expansions stated in
\Cref{thm:2nd_coef_kernel} and \Cref{thm:2nd-order-score}. The proofs of the first-order expansions
of \Cref{thm:first_coef_kernel} and \Cref{thm:first-order-score} follow the same lines. The detailed
proofs are given in the appendix.

To obtain the desired expansion of \(p_\sigma\) and its first- and second-order derivatives, we
decompose this density as the sum of two terms:
\begin{equation}\label{eq:loc-far}
  p_\sigma(y)
  = \underbrace{ \int_{\Man} \chi_{\pi(y)} (x')\, \phi_\sigma(y-x')\, q(\dd x')}_{p_{\textrm{loc},
  \sigma}(y)} + \underbrace{ \int_{\Man} \big( 1 -
  \chi_{\pi(y)} (x')\big)\, \phi_\sigma(y-x')\, q(\dd x') }_{p_{\textrm{far},\sigma}(y)}.
\end{equation}
Here, \(\chi_x(x')\) is a truncation function (introduced later in this section) such that the
integrand of \(p_{\textrm{loc},\sigma}\) vanishes when \(x'\) is far from \(x = \pi(y)\), whereas
the integrand of \(p_{\textrm{far}, \sigma}\) is concentrated on \(x'\) sufficiently far from \(x\).

Let the chart \(\Phi\) and the open set \(\Theta_{\mathrm{out}} \Subset \mathcal V\) be as in
\Cref{ass:standing}~\ref{ass:standing:chart}, so that \(\Compct_{\Strat}\subset \Strat_{\mathrm{out}} =
\Phi(\Theta_{\mathrm{out}} \times\{0\})\). Recall that  \(\varphi: \Theta_{\mathrm{out}}\to
\Strat_{\mathrm{out}}\) and  \(\Theta_{\Compct_{\Strat}} \subset \Theta_{ \mathrm{out}} \) are defined by
\begin{equation*}
  \varphi(\theta) = \Phi(\theta,0), \qquad \Theta_{\Compct_{\Strat}} = \varphi^{-1}(\Compct_{\Strat}).
\end{equation*}
Thus, \(\varphi\) parametrizes the part of the stratum we are interested in. Near each base point
\(x = \varphi(\theta)\) we use the translated corner chart
\begin{equation*}
  \Phi_x(\xi) = \Phi(\theta+\xi_{\mathcal S},\xi_{\mathcal C}),
  \qquad \xi=(\xi_{\mathcal S},\xi_{\mathcal C})\in\mathbb H_c^m .
\end{equation*}
This is the parametrization of the integration variable \(x'\) in the local component
\(p_{\textrm{loc},\sigma}\) of the target density. We show that for each \(x\),  \(\Phi_x\) is a
diffeomorphism defined on \(\mathbb H_c^m\cap \mathbb B_{4R}^m\), for some \(R>0\), onto an open
neighborhood of \(x\) in \(\Man\). In addition, \(\Phi_x\) is \(C^{r + 1}\).

We choose a \(C^\infty\) function \(\chi:\mathbb R_{\geqslant 0}\to [0,1]\) such that \(\chi(t) =
1\) if \(t\leqslant 1\) and \(\chi(t) = 0\) if \(t\geqslant 4\). Using this function, we define the
transported cutoff \(\chi_x: \Man \to [0,1]\) by
\begin{equation}\label{eq:chi_x}
  \chi_x(x') = \chi\bigg(\frac{\|\Phi_x^{-1}(x')\|^2}{R^2} \bigg),
  \qquad \text{if }\qquad x'\in \Phi_x(\mathbb H_c^m\cap \mathbb B_{4R}^m)
\end{equation}
and \(\chi_x(x') = 0\) for all the other \(x'\in\Man\). This zero extension is well defined because
the support of \(\chi\) is compactly contained in \(\mathbb B_{4R}^m\). In words, \(\chi_x\) is the
fixed cutoff \(\chi\), but transported to \(\Man\) through the corner chart centered at the moving
base point \(x\).

In the adapted frame at \(x\), \(x' = \Phi_x (\xi) \) and the measure \(\rho\,\dd\vol_{\Man}\)
written in the chart \(\Phi_x\) has density \(\mathcal A_x(\xi) = \rho\bigl( \Phi_x(\xi)
\bigr)\,J_x(\xi)\) defined by \eqref{eq:defAx}. Thus, \(\rho(x')\,\dd\vol_{\Man}(x') = \mathcal
A_x(\xi)\, \dd\xi\). After the change of variable \(x' = \Phi_x(\xi)\), we get
\begin{equation}\label{eq:loc2}
  p_{{\rm loc},\sigma}(y)
  = \int_{\mathbb H_c^m} \chi(\|\xi\|^2/R^2)\, \phi_\sigma\bigl(y-\Phi_x(\xi)\bigr)\,
  \mathcal A_x(\xi)\,\dd\xi, \qquad x=\pi(y), \qquad \sigma>0.
\end{equation}
To ease notation, we write \(\bar\chi_R(\xi) = \chi\big( \|\xi\|^2/R^2\big)\). Because
\(\operatorname{supp} \bar\chi_R\Subset \mathbb B_{4R}^m\), all quantities in the following integral
are evaluated where the chart is defined. Equivalently, the integral may be read as an integral over
\(\mathbb H_c^m\cap \mathbb B_{2R}^m\) because the cutoff vanishes outside the ball \(\mathbb
B_{2R}^m\). The cutoff is therefore supported strictly inside the uniform corner chart. This will
allow us to differentiate and rescale the local integral without encountering the edge of the chosen
coordinate patch.

The cutoff is equal to one on \(\|\xi\| \leqslant R\). We prove in the appendix that there exists
\(\delta_0>0\), independent of \(x\in\Compct_{\Strat}\), such that every point of \(\Man\) within
distance \(2\delta_0\) of \(x\) is represented in the region where the transported cutoff is
identically one. Hence every contribution surviving the factor \(1-\chi_x\) is uniformly separated
from the base point. This separation is the mechanism behind the exponential smallness of the far
term in \eqref{eq:loc-far}.

\subsection{Conical layer and scaled representation}
\label{subsec:rescaled_variables_scaled_exponent}

We now pass from the local coordinates to the scaled variables in which the boundary layer has a
nontrivial limit. The heat kernel has width \(\sigma\). Hence, near a codimension-\(c\) stratum
\(\Compct_{\Strat}\), the relevant observation points are those lying in a tubular neighborhood of
\(\Strat\) of size \(O(\sigma)\). In this regime the Gaussian sees both the ambient normal
directions and the inward corner constraint. Hence, as we will see, the limiting model is conical.

For \(A>0\) and \(\sigma_0\in(0,R/A)\), and writing \(\mathcal U\) for the tubular neighborhood on
which the projection \(\pi\) is single-valued, define the \(O(\sigma)\)-tubular layer around the
stratum, recalling the boundary-layer regime of \eqref{eq:boundary-layer-regime},
\begin{equation*}
  \mathcal Y_{A,\Compct_{\Strat},\sigma_0}
  = \Bigl\{ (y,\sigma)\in \mathcal U\times (0,\sigma_0]: \pi(y)\in
  \Compct_{\Strat};\ \|y-\pi(y)\|\leqslant A\sigma \Bigr\}.
\end{equation*}
The vector \(y-\pi(y)\) can be represented in an orthonormal basis \(\bfQ(x)\) adapted to the local
geometry of \(\Man\) near \(x\). This representation has a component that is normal to the tangent
space of \(\Man\) and another component within the tangent space but orthogonal to the tangent of
the stratum. After rescaling by \(\sigma\), these two components are denoted by \(a_{\mathcal
N}\in\mathbb R^k\) and \( a_{\mathcal C}\in\mathbb R^c\). We then denote
\begin{equation}\label{eq:a-y-sigma}
  a(y,\sigma) = \begin{bmatrix} a_{\mathcal C}(y,\sigma)\\ a_{\mathcal N}(y,\sigma) \end{bmatrix}
  \in \mathbb R^c\times \mathbb R^k, \quad\text{and have}\quad y- x
  = \sigma \bfQ(x)\begin{bmatrix} 0_{m-c}\\ a_{\mathcal C}\\ a_{\mathcal N} \end{bmatrix},
\end{equation}
with \(x=\pi(y)\). Thus the layer is precisely the region in which the rescaled displacement \(a\)
remains bounded as \(\sigma \downarrow 0\). The integration variable on \(\Man\) will be
parametrized by corner coordinates using the same basis \(\bfQ(x)\) by \(\xi=( \xi_{\mathcal S},
\xi_{\mathcal C})\in\mathbb H_c^m\), where \(\xi_{\mathcal C}\in[0,\infty)^c\). Recall that the
converse parametrization, that of an observation point in the layer by \((a,x,\sigma)\) with \(a =
(a_{\mathcal C},a_{\mathcal N})\in \mathbb R^c\times \mathbb R^k\), \(x\in \Compct_{\Strat}\),
\(\sigma \geqslant 0\), is given by \eqref{eq:ysigma}.

\subsection{Change of variable of integration}

We next rescale the integration variable \(\xi\) in the integral \eqref{eq:loc2} of the localized
kernel. We set \(\xi=\sigma\zeta\), where \(\zeta\in\mathbb H_c^m\). This change of variable aims to
stabilize the Gaussian term \(\phi_\sigma(y-\Phi_x (\xi)) \). We define the scaled exponent
\begin{equation}\label{eq:scaled_exp}
  \Psi_\sigma(\zeta;a,x)
  = \frac{\big\|y_\sigma(a,x) - \Phi_x (\sigma\zeta)\big\|^2}{2\sigma^2}, \text{ so that }
  \phi_{\sigma}\big(y-\Phi_x(\sigma\zeta)\big)
  = \frac{\exp\big(-\Psi_\sigma(\zeta;a,x)\big)}{ (2\pi\sigma^2)^{d/2}}.
\end{equation}
Using the notation introduced in \eqref{eq:defAx}, the change of variable formula leads to
\begin{equation}\label{eq:loc3}
  p_{{\rm loc},\sigma}(y)
  = \frac{\sigma^m}{(2\pi\sigma^2)^{d/2}} \int_{\mathbb H_c^m} \bar\chi_R\big(\sigma\zeta\big)\,
  e^{- \Psi_\sigma(\zeta;a,x)}\mathcal A_{x} (\sigma\zeta)\,\dd\zeta, \qquad k=d-m,
\end{equation}
where \(x = \pi(y)\) and \(a = a(y,\sigma)\) is given by \eqref{eq:a-y-sigma}. To be more precise, we
have used that
\begin{equation*}
  \phi_\sigma\bigl(y-\Phi_x(\sigma\zeta)\bigr) = (2\pi\sigma^2)^{-d/2} e^{-\Psi_\sigma(\zeta;a,x)},
  \qquad \dd \xi=\sigma^m\,\dd\zeta,
\end{equation*}
and therefore the prefactor is \( \sigma^{m-d} =\sigma^{-k}\). The factor
\(\bar\chi_R(\sigma\zeta)\) in \eqref{eq:loc2} keeps the whole scaled integral inside the original
uniform corner chart, since \(\bar\chi_R(\sigma\zeta)\neq 0\) implies that
\(\sigma\|\zeta\|\leqslant 2R\). Thus the rescaling introduces no new geometric domain; it only
magnifies the fixed local chart at scale \(\sigma\).

The remaining task is to expand the scaled kernel \(e^{-\Psi_\sigma(\zeta; a,x)}\mathcal
A_x(\sigma\zeta)\) uniformly for bounded \(a\) and \(x\in\Compct_{\Strat}\). The limiting object is
obtained by freezing the geometry at the base point \(x\). In the chart coordinates, this means
replacing \(\Phi_x(\sigma \zeta)\) by its first-order approximation \(x + \sigma [\,\bfS(x)\
\bfC(x)\,]\bfL(x)\zeta\). Equivalently, the support of the measure is approximated by the inward
tangent cone \(\mathcal T_x^+\Man = [\, \bfS(x)\ \bfC(x)\,]\bfL(x)\mathbb H_c^m\). This produces the
leading conical Gaussian profile. The first correction comes from two sources: the quadratic part of
the chart, which perturbs the exponent, and the first variation of the amplitude \(\mathcal A_x\).

The expansion is carried out at the level of kernels before integration in \(\zeta\). The smoothness
of the remainder terms is established on domains of the form
\begin{equation*}
  \mathcal D_{R,A,\Compct_{\Strat},\sigma_0}
  = \Bigl\{ (\zeta,a,x,\sigma): \zeta\in\mathbb H_c^m,\ \|a\|\leqslant A,\ x\in\Compct_{\Strat},\
  0<\sigma\leqslant\sigma_0,\ \sigma\|\zeta\|<R \Bigr\},
\end{equation*}
where \(R\in(0, \infty]\). The condition \(\sigma\|\zeta\| <R\) ensures that the unscaled variable
\(\xi=\sigma\zeta\) remains inside the fixed corner chart.

\subsection{Taylor expansion of the exponent and the chart}
\label{subsec:taylor_expansion_chart_exponent}

To simplify the scaled exponent, we represent the point \(x' = \Phi_x(\xi)\) in the basis
\(\bfQ(x)\) by the vector \(\Delta_x(\xi)\), which means that
\begin{equation*}
  \Delta_x(\xi) = \bfQ(x)^\top(\Phi_x(\xi) - x) = \bfQ(x)^\top(\Phi_x(\xi) - \Phi_x(0)).
\end{equation*}
Under \Cref{ass:standing}, \ref{ass:standing:chart} and~\ref{ass:standing:density}, this mapping is
\(C^{r}\) in \(x\) and \(C^{r+1}\) in \(\xi\). The lower degree of smoothness with respect to \(x\)
is explained by the fact that \(\bfQ\) encodes information on the tangent space of \(\Man\) at
\(x\), which is defined through the differential of \(\Phi\). The latter has finite regularity
\(r\). Since \(\Delta_x(0_m) = 0_d\), the Taylor expansion leads to \(\Delta_x(\sigma\zeta) = \sigma
\Diff_\xi \Delta_x(0)\zeta + (\frac12)\sigma^2\Diff^2 \Delta_x(0)[\zeta,\zeta] + O(\sigma^3)\),
provided that \(r\geqslant 2\). Using the notation introduced above, we have
\begin{align}
  2\Psi_\sigma(\zeta;a,x) &= \big\|x+\sigma \bfQ (x)[\, 0_{m-c};\,a\,] - \big(x
  + \bfQ(x)\Delta_x(\sigma\zeta) \big) \big\|^2/(\sigma^2) \notag \\
  &= \big\|[\, 0_{m-c};\,a\,] - \sigma^{-1}\Delta_x(\sigma \zeta)\big\|^2 \notag \\
  &= \big\|[\, 0_{m-c};\,a\,] -\Diff_\xi \Delta_x(0)\zeta
  - (\frac\sigma2)\Diff^2_\xi\Delta_x(0)[\zeta,\zeta] \big\|^2 + O(\sigma^2) \notag \\
  &= 2\Psi(\zeta;a,x) + \sigma \underbrace{\big(\Diff_\xi \Delta_x(0) \zeta
  - [\, 0_{m-c};\,a\,]\big)^\top \Diff^2_\xi \Delta_x(0)[\zeta,\zeta]}_{=\Lambda_x(\zeta;a)}
  + O(\sigma^2), \label{eq:expansion_Psi}
\end{align}
see \eqref{eq:Psi0} for the definition of the linearized exponent \(\Psi\) and \eqref{eq:def_Lambda}
for the definition of \(\Lambda_x\). Since \(x\mapsto \bfL(x)\) is \(C^r\), the coefficients of the
quadratic model exponent \(\Psi\) have parameter regularity \(C^r\). Moreover,
\(\Diff_\xi^2\Delta_x(0)\) has parameter regularity \(C^{r-1}\), and therefore the coefficients of
\(\Lambda_x(\zeta;a)\) have parameter regularity \(C^{r-1}\). The loss to order \(r-2\) occurs only
in the second-order Taylor remainder of the exponent.

Another factor present in the integrand of the local kernel \eqref{eq:loc3} that we have to expand
is the amplitude \(\mathcal A_x(0)\) defined by \eqref{eq:defAx}, which can also be written as
\begin{equation*}
  \mathcal A_x(0) = \rho(x)\,J_x(0) = \rho(x)\,|\det \bfL(x)|.
\end{equation*}
Since the amplitude takes real values, its differential is given by its gradient, that is
\(\Diff_\xi\mathcal A_x(0)[\xi] = \nabla_\xi \mathcal A_x(0)^\top\xi\). Because \(\rho>0\) on
\(\Compct_{\Strat}\) and \(\bfL(x)\in\mathsf{GL}(m)\), the quantity \(\mathcal
A_x(0)=\rho(x)|\det\bfL(x)|\) is uniformly bounded away from zero for \(x\in\Compct_{\Strat}\). By
\ref{ass:standing:chart} and~\ref{ass:standing:density} of \Cref{ass:standing}, the amplitude
\(\mathcal A_x\) has the finite regularity inherited from \(\rho\), the chart, and the volume
Jacobian. In particular, if this regularity is at least \(C^2\) in the chart variable, Taylor
expansion yields
\begin{equation}\label{eq:ampl_exp0}
  \mathcal A_x(\sigma\zeta) = \rho(x)J_x(0) + \sigma \,\nabla_\xi \mathcal A_x(0)^\top \zeta
  + O(\sigma^2).
\end{equation}
Combining \cref{eq:expansion_Psi,eq:ampl_exp0}, we arrive at
\begin{equation}\label{eq:exp_exp_A}
  e^{-\Psi_\sigma(\zeta;a,x)}\mathcal A_x(\sigma\zeta) = e^{-\Psi(\zeta;a,x)}\Big(\mathcal A_x(0)
  +\sigma \kappa_{1,x}(\zeta;a)\Big) + O(\sigma^2),
\end{equation}
where we have used the notation \(\kappa_{1,x}(\zeta;a) = \nabla_\xi \mathcal A_x(0)^\top \zeta -
(\frac12)\mathcal A_x(0)\Lambda_x(\zeta;a)\). Although \(\kappa_{1,x}\) contains the first
variations of both the amplitude and the exponent, the important point is that, for
fixed \(a\) and \(x\), it is a polynomial in \(\zeta\) of degree at most three.

\subsection[Proof of Theorem~\ref{thm:2nd_coef_kernel}: integration and logarithm]
{Proof of \Cref{thm:2nd_coef_kernel}: integration and passage to logarithm}
\label{sec:integration_and_local_coefficients}

The fixed local cutoff \(\bar\chi_R(\sigma\zeta)\) only guarantees that \(\sigma\|\zeta\|\leqslant
2R\), so it allows \(\|\zeta\|\) as large as order \(\sigma^{-1}\). On that whole region the Taylor
expansion of \(e^{- \Psi_\sigma}\) around \(e^{-\Psi}\) is not a uniform perturbative expansion. For
this reason we introduce a second cutoff, denoted \(\vartheta_\sigma\). It is defined by
\begin{equation*}
  \vartheta_\sigma(\zeta) = \chi(\sigma^{1/2}\|\zeta\|^2).
\end{equation*}
Here \(\chi\in C_c^\infty(\mathbb R_{\geqslant 0})\) is the same function as in \eqref{eq:chi_x}
satisfying \(0 \leqslant \chi(t)\leqslant 1\), \(\chi=1\) on \(t\leqslant 1\), and \(\chi=0\) on \(t
\geqslant 4\). Thus \(\vartheta_\sigma=1\) for \(\|\zeta\| \leqslant \sigma^{-1/4}\) and
\(\vartheta_\sigma=0\) for \(\|\zeta\| \geqslant 2\sigma^{-1/4}\). Since \(\sigma^{-1/4} \) tends to
infinity, the inner region exhausts the whole tangent-cone variable space as \(\sigma\downarrow 0\).
At the same time it grows slowly enough that the cubic exponent perturbation remains small.
Moreover, a suitable choice of \(\sigma_0\) ensures that
\begin{equation}\label{eq:supp}
  \operatorname{supp}\vartheta_\sigma \subset \{\|\zeta\|\leqslant 2\sigma^{-1/4}\}
  \subset \{\sigma\|\zeta\|\leqslant R\} \subset \big\{\zeta: \chi\big(\sigma^2\|\zeta\|^2/R^2\big)
  = 1\big\},
\end{equation}
for all \(\sigma\in(0,\sigma_0]\). We show in the appendix that the remainder term in the expansion
\begin{equation*}
  \vartheta_\sigma(\zeta)e^{-\Psi_\sigma(\zeta;a,x)}\mathcal A_x(\sigma\zeta)
  = \vartheta_\sigma(\zeta) e^{-\Psi(\zeta;a,x)}\Big(\mathcal A_x(0)
  +\sigma \kappa_{1,x}(\zeta;a)\Big) + O(\sigma^2),
\end{equation*}
derived from \eqref{eq:exp_exp_A}, is sufficiently smooth and has exponentially small
tails when \(\zeta\) becomes large. This is a key point: after the insertion of
\(\vartheta_\sigma\), the Taylor remainder can be integrated in \(\zeta\) uniformly in the
parameters. The role of \(\vartheta_\sigma\) is only temporary. In the next step, the model
coefficients are defined by integrating over the full cone \(\mathbb H_c^m\); the difference between
\(\mathbb H_c^m\) and the growing inner region is absorbed into exponentially small tails.

To ease notation, recall that \(\bar\chi_R(\xi) = \chi\big( \|\xi\|^2/R^2\big)\) and let us
introduce three tails, defined precisely in
\cref{eq:def_Tail_ex,eq:def_Tail_0,eq:def_Tail_0,eq:def_Tail_1}. The exact tail
\(\operatorname{Tail}_{\mathrm{ex}}(a,x,\sigma)\) is the contribution of the truncated kernel
outside the growing inner region \(\{\vartheta_\sigma=1\}\); the model tails
\(\operatorname{Tail}_{0}(a,x,\sigma)\) and \(\operatorname{Tail}_{1}(a,x,\sigma)\) are the parts of
the full-cone coefficients \(\mathsf C_0\) and \(\mathsf C_1\) that lie outside this region, namely
the integrands \(\mathcal A_x(0)e^{-\Psi}\) and \(e^{-\Psi}\kappa_{1,x}\) weighted by
\((1-\vartheta_\sigma)\). In view of \cref{eq:supp,eq:exp_exp_A}, the scaled local integral
\[
  I_\sigma(a,x)
  =
  \int_{\mathbb H_c^m}\bar\chi_R(\sigma\zeta)\,
  e^{-\Psi_\sigma(\zeta;a,x)}\mathcal A_x(\sigma\zeta)\,\dd\zeta,
  \qquad
  p_{{\rm loc},\sigma}(y)=\sigma^{-k}(2\pi)^{-d/2}I_\sigma(a,x),
\]
admits the exact decomposition
\begin{equation*}
  I_\sigma(a,x)
  = \mathsf C_0(a,x)+\sigma\mathsf C_1(a,x)+\sigma^2\mathscr E_{\mathrm{in}}(a,x,\sigma)
  + \operatorname{Tail}_{\mathrm{ex}} - \operatorname{Tail}_{0}
  - \sigma\operatorname{Tail}_{1}.
\end{equation*}
As proved in the appendix, the exact and model tails are exponentially small in the admissible
classes, even after multiplication by the powers of \(\sigma^{-1}\) needed to absorb them into the
\(\sigma^2\)-remainder. Hence they can be included in \(\sigma^2\mathscr E_{\rm loc}\). Multiplying by
\(\sigma^{-k}(2\pi)^{-d/2}\) gives
\begin{equation*}
  p_{{\rm loc},\sigma}(y) = \sigma^{-k}(2\pi)^{-d/2} \Bigl[ \mathsf C_0(a,x)
  + \sigma\mathsf C_1(a,x) + \sigma^2\mathscr E_{\rm loc}(a,x,\sigma) \Bigr],
\end{equation*}
for every \((y,\sigma)\in\mathcal Y_{A,\Compct_{\Strat}, \sigma_0}\), \(x=\pi(y)\), and
\(a=a(y,\sigma)\).

To assess the complementary contribution
\begin{equation*}
  p_{{\rm far},\sigma}(y)
  = \int_{\Man} \bigl(1-\chi_{\pi(y)}(x')\bigr) \phi_\sigma(y-x')\, \rho(x')\,\dd\vol_{\Man}(x'),
\end{equation*}
we use the fact that \(\chi_{\pi(y)}(x') = 1\) if \(\|\pi(y) -x'\| \leqslant 2\delta_0\). Therefore,
\begin{equation*}
  p_{{\rm far},\sigma}(y)
  \leqslant \int_{\{x':\|\pi(y)-x'\|>2\delta_0\}} \phi_\sigma(y-x')\, \rho(x')\,\dd\vol_{\Man}(x') .
\end{equation*}
Since \(\|y - \pi(y)\| \leqslant \sigma_0 A\), a careful choice \(\sigma_0\) implies that \(\|y-x'\|
\geqslant \delta_0\) whenever \(\|\pi(y) - x'\|> 2\delta_0\). Therefore, the far term is
exponentially small on the scale \(e^{-c/\sigma^2}\). The same reasoning remains valid after
differentiating. Derivatives may fall either on the Gaussian or on the transported cutoff.
Derivatives of the Gaussian produce only powers of \(\sigma^{-1}\), absorbed by the exponential
decay. Derivatives of the transported cutoff are bounded and supported in the same separated region.
This concludes the proof of the first claim of \Cref{thm:2nd_coef_kernel}. To obtain the second
claim, it suffices to check that the dominating term, \(\mathsf C_0\) is uniformly bounded away from
zero on \((x,a)\in \Compct_{\Strat} \times \overline{\mathbb B}_A\) and the other terms of the
expansion, \(\mathsf C_1\) and \(\mathscr E_1\), have the required finite regularity and are uniformly
bounded. Then, the desired result follows from the Taylor formula applied to the function \(t\mapsto
\log(1+t)\).

\subsection[Proof of Theorem~\ref{thm:2nd-order-score}: differentiating the log-expansion]
{Proof of \Cref{thm:2nd-order-score}: differentiating the log-expansion}
\label{ssec:diff_log_exp}

The starting point for the proof of the expansions of the score and its derivatives is the expansion
\begin{equation}\label{eq:exp_log}
  \log p_\sigma(y) = -k\log\sigma - (d/2) \log(2\pi) + \mathsf L_0(a,x) + \sigma\mathsf L_1(a,x)
  + \sigma^2\mathscr E_{1,\log} (a,x,\sigma),
\end{equation}
established in \Cref{thm:2nd_coef_kernel}, in which \(\mathscr E_{1,\log}\in\mathfrak B_{r-2}\) in
the boundary-layer variables \((a,x,\sigma)\). Thus its admissible derivatives in
\((a,x,\sigma)\) are uniformly bounded, while ambient \(y\)-derivatives are obtained only after
applying the conical-layer chain rules, which produce the explicit powers of \(\sigma^{-1}\)
appearing in the score and Hessian expansions. For the score expansion, \(r\geqslant3\) is
sufficient to differentiate the logarithmic expansion once in the ambient variable. For the Hessian
and the scale derivative of the score, one additional admissible derivative is needed, giving the
threshold \(r\geqslant4\), in agreement with \Cref{thm:2nd-order-score}. Since these germs
are expressed in the adapted variables \(x=\pi(y)\) and \(a=a(y,\sigma)\), we need to translate
derivatives in \(y\) into derivatives in the stratum variable \(x\) and in the rescaled transverse
coordinate \(a\). Recalling \(\theta(y),\nu(y)\) from \eqref{eq:theta-nu-def} and using
\(\varphi(\theta)=\Phi(\theta,0)\), define
\begin{equation*}
  \theta(y)=\varphi^{-1}(\pi(y))\in\Theta_{\Compct_{\Strat}},
  \qquad \nu(y)= [\,\bfC(x)\ \bfN(x) \,]^\top\big(y-\pi(y)\big) \in\mathbb R^{c+k}
\end{equation*}
so that \(a(y,\sigma)=\nu(y)/\sigma\). Here \(\theta(y)\) is the local coordinate of the projection
\(\pi(y)\) on the stratum and \(\nu(y)\) collects the unscaled tubular coordinates transverse to the
stratum. The maps \(\nu\) and \(\theta\) have the finite differentiability required on the relevant
compact region of the tubular neighborhood \(\mathcal Y_{A,\Compct_{\Strat},\sigma_0}\). Recall that
the corresponding differentials are denoted by \(J_\theta\) and \(J_\nu\), see \cref{def:J1,def:J2}.

The explicit terms \(-k\log\sigma\) and \(-(d/2) \log(2\pi)\) in \eqref{eq:exp_log} being
independent of \(y\), their derivatives vanish. By contrast, every derivative falling on the
rescaled transverse coordinate \(a(y,\sigma)\) produces one factor \(\sigma^{-1}\), while
derivatives falling on \(\pi(y)\), or equivalently on the stratum coordinate \(\theta(y)\), remain
uniformly bounded. To obtain the score expansion, we simply use the chain rule:
\begin{equation}\label{eq:chain1}
  \nabla_yF(a(y,\sigma),\pi(y),\sigma) = J_\theta(y)^\top\nabla_\theta F
  + (1/\sigma) J_\nu(y)^\top\nabla_aF
\end{equation}
that we apply to \(F = \mathsf L_0\) and \(F = \mathsf L_1\). This leads to the coefficients
\(\mathsf S_0\) and \(\mathsf S_1\). The gradient and the differential with respect to \(\theta\) on
the right-hand side of \eqref{eq:chain1} should be understood as those of the mapping
\(\theta\mapsto F(a,\varphi( \theta),\sigma)\), \textit{i.e.}, the function obtained by replacing
\(x\) by its parameterization \(\varphi (\theta)\).

For the scale derivative of the score, we start from the score expansion
\begin{equation*}
  \score(y) = \frac1\sigma{\mathsf S}_0(a,x,\sigma) + \mathsf S_1(a,x,\sigma)
  + \sigma\mathscr R_{1,\score}(a,x,\sigma)
\end{equation*}
and use the chain rule with the fact that \(\partial_\sigma a(y,\sigma) = -\sigma^{-1} a(y,
\sigma)\). Hence, for each score coefficient, \(\partial_\sigma \mathsf S_j(a(y,\sigma),x,\sigma) =
\partial_\sigma\mathsf S_j(a,x,\sigma) -\sigma^{-1} (\Diff_a\mathsf S_j(a,x,\sigma))[a]\).
Grouping powers of \(\sigma\) leads to the coefficients defined in \Cref{thm:2nd-order-score}.

%% file: sections/conclusion.tex

\section{Conclusion and perspectives}
\label{sec:conclusion}

We have proved boundary-layer asymptotic expansions for Gaussian regularizations of measures
supported on manifolds with corners. The main point of the analysis is that, near a boundary or
corner stratum, the small-noise limit is not governed by a smooth Euclidean density. It is governed
by the Gaussian mass of the inward tangent cone. This cone determines the leading density
coefficient, and its logarithmic derivatives determine the leading score, log-Hessian, and scale
derivative of the score.

One concrete consequence concerns Gaussian denoising. If \(Y_\sigma=X+\sigma Z\), with \(X\sim q\)
and an independent vector \(Z\sim\mathcal N(0,\bfI_d)\), then Tweedie's identities give
\begin{equation}\label{eq:tweedie2}
  \mathbb E[X\,|\,Y_\sigma=y] = y+\sigma^2\score(y), \quad \operatorname{Cov}(X\,|\,Y_\sigma=y)
  = \sigma^4\Hess(y)+\sigma^2\bfI_d .
\end{equation}
Thus the score determines the denoising displacement, while the log-Hessian determines both the
posterior covariance and the local linearization of the denoising map
\citep{robbins1956empirical,efron2011tweedie, vincent2011connection}. In the smooth
embedded-manifold case without boundary, for example, if \(y=x+\sigma\bfN(x)a\), our expansion gives
\begin{equation*}
  \mathbb E[X\,|\,Y_\sigma=y] = x + \sigma^2 \left( \nabla_{\Man}\log\rho(x)
  + \frac12\bsh_{\Man}(x) \right) + O(\sigma^3),
\end{equation*}
and \(\Diff_y \mathbb E[X\,|\,Y_\sigma=y] = \bfI_d+\sigma^2\Hess(y) = \bfP_{\mathcal T}(x) +
O(\sigma)\), where \(\bfP_{\mathcal T}(x)\) is the orthogonal projection onto \(\mathcal T_x\Man\).
Hence, at small noise, the population denoiser first collapses normally onto the support, then moves
along the support according to the intrinsic density and the mean curvature. Near a boundary or a
corner, the same role is played by the tangent-cone profile \(\mathsf C_0\): the denoiser is no
longer described by a linear normal projection alone, but by the conditional mean inside a
half-space or corner cone.

This interpretation also gives a geometric use of the score and the Hessian. In the boundaryless
case, the leading log-Hessian is \(-\sigma^{-2}\bfP_{\mathcal N}(x)\), so its large negative
eigenspace identifies the ambient normal space. Equivalently, the near-null directions of
\(\bfI_d+\sigma^2\Hess(y)\) distinguish normal directions, whereas the leading nonzero part
identifies the tangent space. This is closely related to recent uses of trained diffusion scores for
recovering normal bundles and intrinsic dimension \citep{stanczuk2024intrinsic}. The present results
refine this picture by showing what changes at boundaries and corners: additional singular
directions appear inside the support, controlled by derivatives of the Gaussian mass of the inward
tangent cone. Thus the same score or Hessian field can, in principle, distinguish an interior point
of a manifold from a boundary point or a corner point.

A second concrete application concerns score-based and diffusion generative models, where the
learned population target is a score field of a Gaussian-perturbed data law
\citep{vincent2011connection, sohl2015nonequilibrium,ho2020ddpm,song2021scorebased, karras2022edm}.
Along the heat path, the canonical velocity is \(-\sigma\score\), its spatial derivative is
\(-\sigma\Hess\), and its scale variation involves \(\Vel=\partial_\sigma\score\). Therefore, our
expansions give a local model for the stiffness and geometry of the reverse dynamics at small noise.
This is relevant for high-order samplers and Taylor-type discretizations, where derivatives of the
score, Jacobian-vector products, or scale derivatives of the learned field intervene
\citep{lu2022dpmsolver,lu2022dpmsolverpp,zhang2023deis,
dockhorn2022genie,tachibana2021quasitaylor,kim2024ctm}.

The Hessian expansion also gives a local interpretation of recent DDPM discretization assumptions.
In particular, \citet{arsenyan2025assessing} obtain fast Wasserstein discretization bounds under a
condition on the Gaussian-smoothed law which, through the second-order Tweedie identity
\eqref{eq:tweedie2}, is a covariance or log-Hessian control. Assumptions on the conditional
covariance are equivalent to assumptions on the Hessian of \(\log p_\sigma\). Our small-noise
expansion provides a geometric mechanism for such conditions: near a stratum,
\begin{equation*}
  \operatorname{Cov} (X\,|\,X+\sigma Z=y_\sigma(a,x))
  = \sigma^2 \bigl( \bfI_d+\mathsf H_0(a,x,\sigma) \bigr) + O(\sigma^3),
\end{equation*}
where the leading matrix is determined by the tangent cone. Hence, at least in the low-noise regime,
the Hessian conditions used to obtain fast DDPM discretization can be understood through the local
geometry of the data support.

Conversely, the expansion of \(\Vel=\partial_\sigma\score\) shows that regularity assumptions in the
noise variable cannot hold uniformly down to \(\sigma=0\) for genuinely lower-dimensional data. Some
DDIM analyses assume regular dependence of the score on the diffusion time or noise level to control
deterministic-sampler discretization errors \citep{chen2023restoration,yu2025wasserstein}. Our
results show that, in a boundary layer around a lower-dimensional manifold, the
\(\sigma\)-derivative of the score has order \(\sigma^{-2}\). Thus any Lipschitz-in-noise constant
must blow up as the terminal noise level tends to zero, especially near lower-dimensional supports,
boundaries, and corners. This gives a geometric explanation for low-noise Lipschitz singularities in
diffusion models \citep{yang2024lipschitz}.

The same Hessian expansion is relevant for inverse problems with diffusion priors. In second-order
Tweedie approximations, the posterior covariance is expressed through the Hessian of the
log-density, and recent methods use this information to improve posterior sampling or image
restoration \citep{boys2024tmpd,rout2024beyondfirstorder}. Our formulas show that, when the prior is
concentrated near a singular support, this covariance has a degenerate small-noise structure:
tangent directions, normal directions, and boundary-cone directions scale differently. Accounting
for this anisotropy may be important when diffusion priors are used for constrained or intrinsically
low-dimensional data.

Several directions remain open. A first one is to continue the expansion to higher orders. This
would require a systematic description of higher jets of the support, the density, and the volume
element, and would give more precise local models for high-order score derivatives. A second
direction is to analyze transition regimes where several strata interact, for example when the
observation point approaches a corner at a scale different from its distance to an adjacent boundary
face. A third direction is to weaken the geometric assumptions, allowing more general stratified
supports, nonsmooth densities, or mixtures of components of different dimensions.

Finally, the results suggest inverse geometric questions. The regularized density and its
logarithmic derivatives retain information about the tangent cone, the ambient codimension, the
active boundary constraints, the density on the support, and curvature corrections. It is therefore
natural to ask whether these objects can be reconstructed from small-noise samples or from learned
score fields. In this sense, the boundary-layer expansion provides not only an asymptotic
description of Gaussian smoothing, but also a possible bridge between singular geometric inference,
denoising, and the analysis of generative models.

%% file: sections/appendices.tex

{\scriptsize
\setlength{\tabcolsep}{3pt}
\renewcommand{\arraystretch}{0.92}
\setlength{\LTpre}{0pt}
\setlength{\LTpost}{0pt}
\setlength{\extrarowheight}{0pt}
\begin{longtable}{@{}p{0.25\textwidth}p{0.72\textwidth}@{}}
\caption{Complementary notation used in the appendix proofs.}
\label{tab:appendix-notation}\\
\toprule
Notation & Meaning in the appendix proofs \\
\midrule
\endfirsthead
\multicolumn{2}{@{}l}{\itshape Table~\ref{tab:appendix-notation} (continued)}\\
\toprule
Notation & Meaning in the appendix proofs \\
\midrule
\endhead
\midrule
\multicolumn{2}{r@{}}{\itshape continued on next page}\\
\endfoot
\bottomrule
\endlastfoot

\(\theta\), \(x=\varphi(\theta)\)
& Stratum coordinate and associated point of
\(\Strat_{\rm out}\). The chart is fixed in
\Cref{ass:standing}~\ref{ass:standing:chart}; derivatives in \(x\) are understood through the
pullback by \(\varphi\), as explained in
\Cref{rem:parameter_variable_derivatives}. \\

\(\zeta=(\zeta_{\mathcal S},\zeta_{\mathcal C})\),
\(\xi=\sigma\zeta\)
& Scaled integration variable in
\(\mathbb H_c^m=\mathbb R^{m-c}\times[0,\infty)^c\) and
corresponding local chart variable. The variable \(\zeta\) is
introduced in \Cref{subsec:admissible_class_definitions}, and the
scaled domains are defined in \eqref{eq:mathcal-D}. \\

\(\mathcal D_{R,A,\Compct_{\Strat},\sigma_0}\)
& Proof domain for functions of \((\zeta,a,x,\sigma)\):
\(\sigma\|\zeta\|<R\), \(\|a\|\leqslant A\),
\(x\in\Compct_{\Strat}\), and \(0<\sigma\leqslant\sigma_0\);
see \eqref{eq:mathcal-D}. \\

\(\mathcal D_{A,\Compct_{\Strat},\sigma_0}\)
& Parameter domain for functions of \((a,x,\sigma)\), with the
same bounds on \(a,x,\sigma\). It is defined immediately after
\eqref{eq:mathcal-D}. \\

\(\partial_{a,\theta,\sigma}^{\gamma}\)
& Mixed parameter derivative
\(\partial_a^\alpha\partial_\theta^\beta\partial_\sigma^j\),
where \(\gamma=(\alpha,\beta,j)\). This convention is introduced
in \Cref{subsec:admissible_class_definitions}, just before
\Cref{rem:parameter_variable_derivatives}. \\

\(\mathfrak B_\ell\), \(\mathfrak E_\ell^\omega\)
& Bounded admissible functions, respectively exponentially small
admissible functions, with derivatives up to order \(\ell\);
see \Cref{def:admissible_classes}. \\

\(\mathfrak P_\ell\), \(\mathfrak G_\ell\)
& Polynomial-growth functions in \(\zeta\), respectively functions
dominated by a polynomial times a canonical conical Gaussian, with
parameter derivatives up to order \(\ell\);
see \Cref{def:admissible_classes}. \\

\(\Gamma_{c_0}(\zeta;a)\)
& Canonical conical Gaussian
\(\exp\{-c_0(\|\zeta_{\mathcal S}\|^2+
\|a_{\mathcal C}-\zeta_{\mathcal C}\|^2+
\|a_{\mathcal N}\|^2)\}\). It is introduced just before
\eqref{eq:mathcal-D} and used in the definition of
\(\mathfrak G_\ell\) in \Cref{def:admissible_classes}. \\

\(\bar\chi_R\), \(\vartheta_\sigma\)
& Fixed local cutoff
\(\bar\chi_R(\xi)=\chi(\|\xi\|^2/R^2)\) and growing inner cutoff
in the scaled variable \(\zeta\). The growing cutoff is introduced
in \Cref{lem:admissibility_of_growing_cutoff}; the vector cutoff
\(\bar\chi_R\) is introduced before
\Cref{lem:gaus_tail_beyond_window} and used in
\Cref{lem:scaled_kern_gaus_class}. \\

\(\Psi(\zeta;a,x)\)
& Model exponent obtained from the linearized conical support;
see \eqref{eq:Psi1}. The same leading exponent is introduced in
the main text in \eqref{eq:Psi0}. \\

\(\Psi_\sigma(\zeta;a,x)\)
& Exact scaled exponent
\(\|y_\sigma(a,x)-M(\theta,\sigma\zeta)\|^2/(2\sigma^2)\),
expanded in \eqref{eq:exponent_exp}. \\

\(\mathcal A\)
& Local amplitude, equal to the density times the chart-volume
Jacobian; see \eqref{eq:defAx}. In stratum coordinates the
base-point form \(\mathcal A_x\) and the two-argument form satisfy
\(\mathcal A_x(\xi)=\mathcal A(\theta,\xi)\) for \(x=\varphi(\theta)\).
Its scaled Taylor expansion is \eqref{eq:ampl_exp}. \\

\(\mathscr K_\sigma\)
& Exact scaled local kernel
\(e^{-\Psi_\sigma(\zeta;a,x)}\mathcal A_x(\sigma\zeta)\).
It is introduced just before \Cref{lem:inner_exp_gaussian_rem},
and its inner expansion is \eqref{eq:exp_inner}. \\

\(I_\sigma(a,x)\)
& Scaled local integral after factoring out
\(\sigma^{-k}(2\pi)^{-d/2}\); see \eqref{eq:I_sigma}. \\

\(\mathsf C_0,\mathsf C_1\)
& Density coefficients obtained by integrating the leading and
first-order conical kernels. The leading coefficient is defined in
\eqref{eq:C0}, and the first correction is defined in
\eqref{eq:def_C1}. Their regularity is proved in
\Cref{lem:coeff_kernels_in_Br}. \\

\(\mathsf L_0,\mathsf L_1\)
& Log-density coefficients
\(\mathsf L_0=\log\mathsf C_0\) and
\(\mathsf L_1=\mathsf C_1/\mathsf C_0\). They are introduced in
\Cref{lem:uniform_positivity_of_C0}; the logarithmic expansion is
based on \eqref{eq:log_of_exp}. \\

\(\operatorname{Tail}_{\rm ex}\)
& Exact tail outside the growing inner region. It is introduced in
the tail decomposition preceding \Cref{lem:decomp_loc_integral},
in the same display as \eqref{eq:def_Tail_0} and
\eqref{eq:def_Tail_1}. \\

\(\operatorname{Tail}_0\), \(\operatorname{Tail}_1\)
& Model tails produced when truncated model integrals are replaced
by the full-cone coefficient integrals defining
\(\mathsf C_0\) and \(\mathsf C_1\); see
\eqref{eq:def_Tail_0} and \eqref{eq:def_Tail_1}. \\

\(\mathbf J_\theta,\mathbf J_\nu\)
& Reconstructed first-order chain-rule coefficients for the
coordinates \(\theta(y)\) and \(\nu(y)\), evaluated along
\(y=y_\sigma(a,x)\). The underlying Jacobians are defined in
\eqref{def:J1}, the reconstructed versions in \eqref{def:J2},
and the gradient chain rule is \eqref{eq:app_chain_rule_gradient}. \\

\(\mathbf Q_\theta,\mathbf Q_\nu\)
& Reconstructed second-order chain-rule coefficients coming from
the Hessians of \(\theta(y)\) and \(\nu(y)\), evaluated along
\(y=y_\sigma(a,x)\). They are introduced immediately after
\eqref{def:J2}, and enter the Hessian chain rule
\eqref{eq:app_chain_rule_hessian}; their admissibility is proved
in \Cref{lem:reconstructed_differential_coefficients_in_Br}. \\

\(\mathsf S_i,\mathsf H_i,\dot{\mathsf S}_i\)
& Coefficients in the score, Hessian, and scale-derivative
expansions, obtained by applying the chain rules to
\(\mathsf L_0\) and \(\mathsf L_1\). The first-order coefficients
\((i=0)\) are defined in \Cref{thm:first-order-score}, and the
second-order coefficients \((i=1)\) are defined in
\Cref{thm:2nd-order-score}. \\

\end{longtable}
}
\clearpage

\paragraph{Assumptions used in the appendix proofs.}
Throughout the appendix we indicate explicitly which items of \Cref{ass:standing} each statement
uses. Recall the roles of the four items: item~\ref{ass:standing:measure} gives the measure
representation \(q=\rho\,\dd\vol_{\Man}\) and the probability normalization; item
\ref{ass:standing:chart} is the local corner-chart hypothesis; item~\ref{ass:standing:density} is
the \(C^r\) regularity of the local density representative; and item~\ref{ass:standing:positivity}
is the lower bound on \(\rho\) along \(\Compct_{\Strat}\). Each lemma and corollary states, at its
outset, the precise items on which it depends.

\section{Geometric ingredients}
\label{app:geometric_preliminaries}

This appendix proves the coordinate package used in the proofs of the main results. The argument is
split into three steps: we first construct the adapted orthonormal frame along the stratum, then
introduce orthogonal tubular coordinates around the stratum, and finally choose a uniform translated
corner parametrization for the integration variable. Throughout this section, we use only the
geometric part of \Cref{ass:standing}, namely item~\ref{ass:standing:chart}, with its stated chart
regularity \(\Phi\in C^{r+1}\). Let us write
\begin{equation*}
  \varphi(\theta)=\Phi(\theta,0), \qquad \Theta_{\Compct_{\Strat}} = \varphi^{-1}(\Compct_{\Strat}).
\end{equation*}
Derivatives in the base point \(x\in\Compct_{\Strat}\) are always understood after the pullback
\(x=\varphi(\theta)\). Notation introduced throughout the appendix proofs is collected for reference
in \Cref{tab:appendix-notation}.

We shall use the following standard consequence of vector-bundle triviality. It is included here to
make the smooth completion of the adapted tangent frame explicit.

\begin{lemma}[Smooth orthogonal completion]
  \label{lem:smooth_orthogonal_completion}
  Let \(\mathcal B\) be a contractible paracompact \(C^s\) manifold and let \(E\subset \mathcal
  B\times\mathbb R^d\) be a \(C^s\) rank-\(m\) subbundle, \(0\leqslant m \leqslant d\). Assume that
  \(E\) is equipped with a \(C^s\) orthonormal frame \(q_1,\ldots,q_m\). Then there are \(C^s\) maps
  \(n_1,\ldots,n_{d-m}: \mathcal B \to \mathbb R^d\) such that, for every \(b\in\mathcal B\),
  \(q_1(b),\ldots,q_m(b), n_1(b),\ldots, n_{d-m} (b)\) is an orthonormal basis of \(\mathbb R^d\).
\end{lemma}

\begin{proof}
Let \(E^\perp\) be the Euclidean orthogonal complement of \(E\) in the trivial bundle \(\mathcal
B\times\mathbb R^d\). It is a \(C^s\) rank-\((d-m)\) subbundle and
\begin{equation*}
  \mathcal B\times\mathbb R^d = E\oplus^\perp E^\perp.
\end{equation*}
Since \(\mathcal B\) is contractible and paracompact, the standard \(C^s\) vector-bundle triviality
theorem over contractible paracompact \(C^s\) bases implies that \(E^\perp\) is \(C^s\)-trivial.
Equivalently, one may first use the usual topological triviality over contractible paracompact
bases, see for example \cite[Ch.~3, Cor.~4.8]{husemoller1994fibre}, and then upgrade the
trivializing bundle maps to class \(C^s\) by the standard smoothing argument: a continuous section
or bundle map is approximated by a \(C^s\) one in the fine \(C^0\) topology using \(C^s\) partitions
of unity, and a sufficiently close approximation remains a trivialization; see
\cite[Ch.~2--4]{hirsch1976differential} for the underlying approximation theorems. Thus \(E^\perp\)
admits a global \(C^s\) frame. Applying Gram--Schmidt inside the fibers of \(E^\perp\) gives a
\(C^s\) orthonormal frame \(n_1,\ldots,n_{d-m}\). Since each \(n_j\) lies in \(E^\perp\), this frame
completes \(q_1,\ldots,q_m\) to an orthonormal frame of the ambient trivial bundle.
\end{proof}

\begin{lemma}[Adapted orthonormal frame and linearized corner data]
  \label{lem:ad_frame_comp_strat}
  Assume \AssChart. One can choose \(C^{r}\) maps
  \begin{equation*}
    \bfS: \Strat_{\mathrm{out}}\to \mathbb R^{d\times(m-c)},\quad
    \bfC: \Strat_{\mathrm{out}}\to \mathbb R^{d\times c},\quad
    \bfN: \Strat_{\mathrm{out}}\to \mathbb R^{d\times k},\quad
    \bfL: \Strat_{\mathrm{out}}\to \mathsf{GL}(m),
  \end{equation*}
  such that, for every \(x=\varphi(\theta)\in\Strat_{\mathrm{out}}\), the following properties hold.
  \begin{enumerate}[itemsep=0pt]
    \item The columns of \(\bfS(x)\),  \(\bfC(x)\) and \( \bfN(x)\) form an orthonormal basis of
    \(\mathcal T_x\Strat\), \(\mathcal T_x\Man\cap(\mathcal T_x\Strat)^\perp\) and \(\mathcal
    N_x\Man=(\mathcal T_x\Man)^\perp\), respectively.

    \item With \(\bfQ_{\Man}(x)=[\,\bfS(x)\ \bfC(x)\,] \in\mathbb R^{d\times m}\) and \(
    \bfQ(x)=[\,\bfS(x)\ \bfC(x)\ \bfN(x)\,] \in\mathsf O(d)\), the matrix \(\bfA_{\mathrm{ch}}(x) =
    \Diff_\xi \Phi(\theta,0)\in\mathbb R^{d\times m}\) satisfies
    \(\bfA_{\mathrm{ch}}(x)=\bfQ_{\Man}(x)\bfL(x)\).

    \item The inward tangent cone is \(\mathcal T_x^+\Man = \bfQ_{\Man}(x)\bfL(x) \mathbb H_c^m\).
    \item The matrix field \(\bfL\) satisfies: \(\sup_{x\in\Strat_{\mathrm{out}}}
    \bigl(\|\bfL(x)\|+\|\bfL(x)^{-1}\| \bigr)<\infty\).
  \end{enumerate}
\end{lemma}

\begin{proof}
Put \(x=\varphi(\theta)\) for some \(\theta\in\Theta_{\rm out}\). The first \((m-c)\) chart
derivatives \(\Diff_{\xi_{\mathcal S}} \Phi(\theta,0)\) span \(\mathcal T_x\Strat\), and all \(m\)
chart derivatives \(\Diff_{\xi} \Phi(\theta,0)\) span \(\mathcal T_x\Man\). Hence
\begin{equation*}
  \mathcal T_x\Man = \mathcal T_x\Strat \oplus\mathcal C^{\mathrm{ch}}_x,
  \qquad \mathcal C^{\mathrm{ch}}_x = \operatorname{Im} \Diff_{\xi_{\mathcal C}}\Phi(\theta,0).
\end{equation*}
Let
\begin{equation*}
  \mathcal C_x = \mathcal T_x\Man\cap(\mathcal T_x\Strat)^\perp.
\end{equation*}
Then \(\mathcal T_x\Man=\mathcal T_x\Strat\oplus^\perp \mathcal C_x\). Both \(\mathcal T_x\Strat\)
and \(\mathcal C_x\) are \(C^{r}\) subbundles over \(\Strat_{\mathrm{out}}\). Pulling them back by
the diffeomorphism \(\varphi:\Theta_{\mathrm{out}}\to\Strat_{\mathrm{out}}\) gives \(C^{r}\) vector
bundles over \(\Theta_{\mathrm{out}}\). The latter set is paracompact, being an open subset of
\(\mathbb R^{m-c}\), and is contractible by \Cref{ass:standing}~\ref{ass:standing:chart}. Hence, the
standard smooth vector-bundle triviality theorem over contractible paracompact bases gives global
\(C^{r}\) frames for the pullback bundles. Transporting these frames through \(\varphi\) and
applying Gram--Schmidt fiberwise gives \(C^{r}\) orthonormal frames \(\bfS\) for \(\mathcal
T_x\Strat\) and \(\bfC\) for \(\mathcal C_x\).

Now apply \Cref{lem:smooth_orthogonal_completion} to the rank-\(m\) subbundle \(\mathcal
T_x\Man|_{\Strat_{\mathrm{out}}}\subset \Strat_{\mathrm{out}}\times\mathbb R^d\), equipped with the
orthonormal frame \([\,\bfS(x)\ \bfC(x)\,]\). We obtain a \(C^{r}\) orthonormal frame \(\bfN(x)\) of
its Euclidean orthogonal complement, namely \(\mathcal N_x\Man\). Therefore
\begin{equation*}
  \bfQ_{\Man}(x)=[\,\bfS(x)\ \bfC(x)\,],
  \qquad \bfQ(x)=[\,\bfS(x)\ \bfC(x)\ \bfN(x)\,]\in\mathsf O(d).
\end{equation*}
Recall that
\[
  \bfA_{\mathrm{ch}}(x)=\Diff_\xi\Phi(\theta,0),
  \qquad x=\varphi(\theta).
\]
Its columns form a basis
of \(\mathcal T_x\Man\), whereas the columns of \(\bfQ_{\Man}(x)\) form an orthonormal basis of the
same space. Thus there is a unique matrix \(\bfL(x)\in\mathsf{GL}(m)\) such that
\(\bfA_{\mathrm{ch}}(x)=\bfQ_{\Man}(x)\bfL(x)\). Equivalently, \(\bfL(x)=\bfQ_{\Man}(x)^\top \bfA_{
\mathrm{ch}}(x)\), so \(\bfL\) is \(C^{r}\).

The chart identifies the inward tangent directions at \(x=\Phi(\theta,0)\) with
\(\bfA_{\mathrm{ch}}(x)\mathbb H_c^m\). Hence
\begin{equation*}
  \mathcal T_x^+\Man = \bfQ_{\Man}(x)\bfL(x)\mathbb H_c^m.
\end{equation*}
Finally, since \(\overline{\Theta}_{\mathrm{out}}\Subset\mathcal V\), the set
\(\overline{\Theta}_{\mathrm{out}}\times\{0\}\) remains inside the domain on which \(\Phi\) is a
\(C^{r+1}\) corner chart. Thus the matrix field \(\bfA_{\mathrm{ch}}\) extends continuously to the
compact set \(\overline{\Theta}_{\mathrm{out}}\), and \(\bfA_{\mathrm{ch}}(\varphi(\theta))\) has
rank \(m\) for every \(\theta\in\overline{\Theta}_{\mathrm{out}}\). By compactness, its smallest
singular value is bounded away from zero, and its largest singular value is bounded above. Hence the
singular values of \(\bfA_{\mathrm{ch}}\) are uniformly bounded above and away from zero. Since
\(\bfQ_{\Man}(x)^\top\bfQ_{\Man}(x)=\bfI_m\),
\begin{equation*}
  \bfA_{\mathrm{ch}}(x)^\top \bfA_{\mathrm{ch}}(x) = \bfL(x)^\top\bfL(x).
\end{equation*}
Thus \(\bfA_{\mathrm{ch}}(x)\) and \(\bfL(x)\) have the same singular values, which gives the
uniform bounds for \(\bfL\) and \(\bfL^{-1}\).
\end{proof}

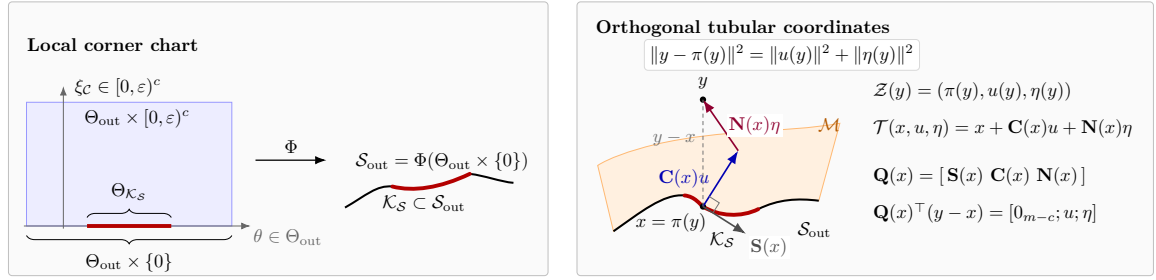
\begin{figure}[ht]
  \centering
  \resizebox{\textwidth}{!}{%
  \begin{tikzpicture}[
  >=Latex,
  font=\small,
  panel/.style={rounded corners=2pt, draw=black!25,
    fill=black!2},
  axis/.style={black!60, -Latex},
  chartfill/.style={fill=blue!7, draw=blue!40},
  manifold/.style={fill=orange!12, draw=orange!55, line
    width=.35pt},
  stratum/.style={draw=black, line width=1.1pt},
  compact/.style={draw=red!70!black, line width=2.1pt},
  vec/.style={-Latex, line width=.9pt},
  dline/.style={dash pattern=on 2.2pt off 1.8pt,
    black!55},
  overbrace/.style={decorate, decoration={brace,
    amplitude=4.5pt}},
  underbrace/.style={decorate, decoration={brace,
    amplitude=4.5pt, mirror}}
]

\begin{scope}[shift={(0,0)}]
  \node[panel, minimum width=10.25cm, minimum
    height=5.2cm,
    anchor=south west] at (-.35,-.95) {};
  \node[anchor=west] at (-.12,3.45)
    {\textbf{Local corner chart}};

  \draw[chartfill] (0,0) rectangle (3.9,2.35);
  \draw[axis] (-.05,0) -- (4.25,0)
    node[below right=-2pt]
      {$\theta\in\Theta_{\mathrm{out}}$};
  \draw[axis] (0.7,-.05) -- (0.7,2.65);
  \node[anchor=west] at (0.78,2.64)
    {$\xi_{\mathcal C}\in[0,\varepsilon)^c$};
  \node at (2.1,2.05)
    {$\Theta_{\mathrm{out}}\times[0,\varepsilon)^c$};

  \draw[underbrace] (0,-.18) -- (3.9,-.18)
    node[midway, below=5.5pt]
    {$\Theta_{\mathrm{out}}\times\{0\}$};

  \draw[compact] (1.15,0) -- (2.75,0);
  \draw[overbrace] (1.15,.16) -- (2.75,.16)
    node[midway, above=5.5pt]
    {$\Theta_{\Compct_{\Strat}}$};

  \draw[->, line width=.7pt] (4.35,1.25) -- (5.7,1.25)
    node[midway, above] {$\Phi$};

  \draw[stratum]
    (6.05,.35)
    .. controls (6.45,.58) and (6.65,.78) .. (6.95,.75)
    .. controls (7.35,.65) and (7.85,.72) .. (8.42,.99)
    .. controls (8.70,.96) and (9.00,.84) .. (9.25,.82);
  \draw[compact, line cap=round]
    (6.95,.75)
    .. controls (7.35,.65) and (7.85,.72) .. (8.42,.99);
  \node[anchor=west] at (6.12,1.22)
    {$\Strat_{\mathrm{out}}=
      \Phi(\Theta_{\mathrm{out}}\times\{0\})$};
  \node[below] at (7.55,.72)
    {$\Compct_{\Strat}\subset\Strat_{\mathrm{out}}$};
\end{scope}
\noindent
\begin{scope}[shift={(10.80,-0.4)}]
  \node[panel, minimum width=10.95cm, minimum
    height=5.2cm,
    anchor=south west] at (-.35,-.55) {};
  \node[anchor=west] at (-.12,4.15)
    {\textbf{Orthogonal tubular coordinates}};

  \path[manifold]
    (0.35,.30)
    .. controls (.85,.65) and (1.30,1.02) .. (1.70,.98)
    .. controls (1.90,.96) and (1.99,.81) .. (2.05,.77)
    .. controls (2.35,.58) and (2.78,.56) .. (3.10,.72)
    .. controls (3.48,.91) and (3.85,.93) .. (4.25,.88)
    -- (4.65,2.35)
    .. controls (3.55,2.16) and (2.45,2.25) .. (1.15,1.95)
    .. controls (.65,1.83) and (.25,1.68) .. (.05,1.55)
    -- cycle;
  \node[orange!70!black] at (4.45,2.33) {$\Man$};

  \draw[stratum]
    (0.35,.30)
    .. controls (.85,.65) and (1.30,1.02) .. (1.70,.98)
    .. controls (1.90,.96) and (1.99,.81) .. (2.05,.77)
    .. controls (2.35,.58) and (2.78,.56) .. (3.10,.72)
    .. controls (3.48,.91) and (3.85,.93) .. (4.25,.88);
  \draw[compact, line cap=round]
    (1.70,.98)
    .. controls (1.90,.96) and (1.99,.81) .. (2.05,.77)
    .. controls (2.35,.58) and (2.78,.56) .. (3.10,.72);
  \node[below right=7pt] at (3.47,.86)
    {$\Strat_{\mathrm{out}}$};
  \node[below] at (2.42,.42) {$\Compct_{\Strat}$};

  \coordinate (x) at (2.05,.77);
  \coordinate (cu) at (2.72,1.83);
  \coordinate (y) at (2.05,2.80);

  \fill (x) circle (1.7pt) node[below left=-1.5pt] at
    (2.2,.72) {$x=\pi(y)$};
  \fill (y) circle (1.7pt) node[above=2pt] {$y$};
  \draw[vec, blue!70!black] (x) -- (cu)
    node[midway, left=2pt] {$\bfC(x)u$};
  \draw[vec, purple!75!black] (cu) -- (y)
    node[midway, right=3pt, fill=white, inner sep=1pt]
      {$\bfN(x)\eta$};
  \draw[dline] (x) -- (y)
    node[midway, above left] {$y-x$};

  \draw[vec, black!70] (x) -- ($(x)+(.82,-.52)$)
    node[below right=1pt, fill=white, inner sep=1pt]
      {$\bfS(x)$};

  \draw[black!55, line width=.45pt]
    ($(x)+(.19,-.12)$) -- ($(x)+(.31,.07)$) --
      ($(x)+(.12,.19)$);

  \node[align=left, anchor=west] at (5.15,2.95)
    {$\mathcal Z(y)=(\pi(y),u(y),\eta(y))$};
  \node[align=left, anchor=west] at (5.15,2.25)
    {$\mathcal T(x,u,\eta)=x+\bfC(x)u+\bfN(x)\eta$};
  \node[align=left, anchor=west] at (5.15,1.35)
    {$\bfQ(x)=[\,\bfS(x)\ \bfC(x)\ \bfN(x)\,]$};
  \node[align=left, anchor=west] at (5.15,.65)
    {$\bfQ(x)^\top(y-x)=[0_{m-c};u;\eta]$};

  \node[align=center, fill=white, draw=black!15, rounded
    corners=2pt,
    inner sep=3pt] at (3.55,3.65)
    {$\|y-\pi(y)\|^2=\|u(y)\|^2+\|\eta(y)\|^2$};
\end{scope}
  \end{tikzpicture}
  }

  \caption{\small Local corner and tubular coordinates near \(\Compct_{\Strat}\). \(\Phi\)
  identifies \(\Theta_{\mathrm{out}}\times\{0\}\) with \(\Strat_{\mathrm{out}}\). Tubular
  coordinates write \(y=x+\bfC(x)u+\bfN(x)\eta\), with \(x=\pi(y)\) and
  \(\bfQ(x)^\top(y-x)=[0_{m-c};u;\eta]\).}
  \label{fig:orth-tub-coord}
\end{figure}

\begin{lemma}[Orthogonal tubular coordinates near \(\Compct_{\Strat}\)]
  \label{lem:orth_tub_coord}
  Assume \AssChart. There are an open neighborhood \(\mathcal{U}\subset\mathbb R^d\) of
  \(\Compct_{\Strat}\), an open neighborhood \(\mathcal W\subset\Strat_{\mathrm{out}}\times \mathbb
  R^c\times\mathbb R^k\) of the zero section over \(\Compct_{\Strat}\), and mutually inverse
  \(C^{r}\) diffeomorphisms
  \begin{equation*}
    \mathcal Z:\mathcal{U}\to\mathcal W,\ y\mapsto (x,u, \eta),
    \quad \text{and}\quad \mathcal T:\mathcal W\to\mathcal{U},\ (x,u,\eta) \mapsto y
  \end{equation*}
  of the form
  \begin{equation*}
    \mathcal Z(y)=\bigl(\pi(y),u(y),\eta(y)\bigr),
    \qquad \mathcal T(x,u,\eta)=x+\bfC(x)u+\bfN(x)\eta.
  \end{equation*}
  Here \(\pi:\mathcal{U}\to\Strat_{\mathrm{out}}\), \(u:\mathcal{U}\to\mathbb R^c\), and
  \(\eta:\mathcal{U}\to\mathbb R^k\) are \(C^{r}\), and \(\pi(y) = \pi_{\mathcal S}(y)\) is the
  nearest-point projection onto \(\Strat\) in this tubular neighborhood.

  Consequently, for every \(y\in\mathcal{U}\),
  \begin{equation*}
    y = \pi(y)+\bfC(\pi(y))u(y)+\bfN(\pi(y))\eta(y), \quad \bfQ(\pi(y))^\top\bigl(y-\pi(y)\bigr)
    = [0_{m-c};u(y);\eta(y)].
  \end{equation*}
  In particular,
  \begin{equation*}
    \|y-\pi(y)\|^2=\|u(y)\|^2+\|\eta(y)\|^2
  \end{equation*}
  implying that \(u(y)=0_c\) and \(\eta(y)=0_k\) for every \(y\in\Strat\cap\mathcal{U}\).
\end{lemma}

\begin{proof}
For each \(x\in\Strat_{\mathrm{out}}\), the columns of \(\bfC(x)\) and \(\bfN(x)\) form an
orthonormal basis of \(\mathcal N_x\Strat = \mathcal C_x\oplus^\perp \mathcal N_x\Man\). Hence,
\begin{equation*}
  \mathcal I(x,u,\eta) = \bigl(x,\bfC(x)u+\bfN(x)\eta\bigr)
\end{equation*}
defines a \(C^{r}\) vector-bundle trivialization from \(\Strat_{\mathrm{out}}\times\mathbb
R^c\times\mathbb R^k\) onto \(\mathcal N\Strat|_{\Strat_{\mathrm{out}}}\).

Apply the nearest-point tubular neighborhood theorem to the embedded stratum \(\Strat\), near
\(\Compct_{\Strat}\subset\Strat_{\mathrm{out}}\); see
\cite[Theorem~6.24 and Proposition~6.25]{lee2013smooth} for the smooth case. Since \(\Strat\) is only of
class \(C^{r+1}\) here, we use the finite-regularity form: the nearest-point projection onto a
\(C^{r+1}\) submanifold is of class \(C^{r}\) on a tubular neighborhood, and the corresponding
normal-exponential map is a \(C^{r}\) diffeomorphism; see \cite{foote1984regularity} and
\cite[Sec.~3.2]{krantzparks2002implicit}. After shrinking, there are a neighborhood
\(\mathcal W_{\mathcal N\Strat}\subset \mathcal N \Strat |_{\Strat_{\mathrm{out}}}\) of the zero
section over \(\Compct_{\Strat}\) and a neighborhood \(\mathcal{U} \subset\mathbb R^d\) of
\(\Compct_{\Strat}\) such that
\begin{equation*}
  \mathcal E(x,\nu)=x+\nu
\end{equation*}
is a \(C^{r}\) diffeomorphism from \(\mathcal W_{\mathcal N\Strat}\) to \(\mathcal{U}\). Its inverse
is
\begin{equation*}
  \mathcal E^{-1}(y) = \bigl(\pi_{\Strat}(y),y-\pi_{\Strat}(y)\bigr),
\end{equation*}
where \(\pi_{\Strat}(y)\) is the nearest point of \(y\) on \(\Strat\) in this neighborhood. Set
\(\mathcal W=\mathcal I^{-1}(\mathcal W_{\mathcal N\Strat})\) and \(\mathcal T=\mathcal E\circ\mathcal
I\). Then \(\mathcal T:\mathcal W\to\mathcal{U}\) is a \(C^{r}\) diffeomorphism, with
\begin{equation*}
  \mathcal T(x,u,\eta) = x+\bfC(x)u+\bfN(x)\eta.
\end{equation*}
Let \(\mathcal Z=\mathcal T^{-1}\), and write \(\mathcal Z(y)=(\pi(y),u(y),\eta(y))\). Comparing
\(\mathcal E^{-1}(y)\) with \(\mathcal I(\mathcal Z(y))\) gives \(\pi(y)=\pi_{\Strat}(y)\) and
\begin{equation*}
  y-\pi(y) = \bfC(\pi(y))u(y)+\bfN(\pi(y))\eta(y).
\end{equation*}
Finally, \(\bfQ(\pi(y))\) is orthogonal and \(y-\pi(y)\in\mathcal N_{\pi(y)}\Strat\). Therefore
\begin{equation*}
  \bfQ(\pi(y))^\top\bigl(y-\pi(y)\bigr) = [0_{m-c};u(y);\eta(y)],
\end{equation*}
and the Pythagorean theorem gives \(\|y-\pi(y)\|^2=\|u(y)\|^2+\|\eta(y)\|^2\). If
\(y\in\Strat\cap\mathcal{U}\), its nearest point is itself, so \(y-\pi(y)=0\), and the last identity
forces \(u(y)=0\) and \(\eta(y)=0\).
\end{proof}

\begin{lemma}[Uniform corner parametrization near \(\Compct_{\Strat}\)]
  \label{lem:unif_par_comp_strat}
  Assume \AssChart. There are an open set \(\Theta^\circ_{\mathrm{out}}\subset\mathbb
  R^{m-c}\) and numbers \(R>0\), \(\delta_0>0\) such that
  \begin{equation*}
    \Theta_{\Compct_{\Strat}}\subset \Theta^\circ_{\mathrm{out}} \Subset\Theta_{\mathrm{out}},
    \qquad \Theta^\circ_{\mathrm{out}} + \mathbb B_{4R}^{m-c} \subset \Theta_{\rm out}
  \end{equation*}
  and the mapping
  \begin{equation*}
    M: \Theta^\circ_{\mathrm{out}} \times (\mathbb H_c^m\cap \mathbb B_{4R}^m)
    \to \Man,\qquad (\theta,\xi)\mapsto \Phi(\theta+\xi_{\mathcal S},\xi_{\mathcal C})
  \end{equation*}
  is well defined and of class \(C^{r+1}\). In addition, if we set
  \begin{equation*}
    \varphi(\theta)=\Phi(\theta,0), \qquad \bfQ(\theta)=\bfQ(\varphi(\theta)),
    \qquad \bfL(\theta)=\bfL(\varphi(\theta)),
  \end{equation*}
  and define \(\Delta: \Theta^\circ_{\mathrm{out}} \times (\mathbb H_c^m\cap \mathbb B_{4R}^m) \to
  \mathbb R^d = \mathbb R^{m-c}\times\mathbb R^c\times\mathbb R^k\) by
  \begin{equation}\label{eq:Delta}
    \Delta(\theta,\xi) = \bfQ(\theta)^\top\bigl(M(\theta,\xi)-\varphi(\theta) \bigr)
  \end{equation}
  then the following properties hold.

  \begin{enumerate}
    \item \label{item:Cr-Delta}
    The map
    \[
      \Delta:\Theta^\circ_{\mathrm{out}}\times
      (\mathbb H_c^m\cap\mathbb B_{4R}^m)\to\mathbb R^d
    \]
    is \(C^r\). Moreover, all mixed derivatives
    \[
      \partial_\theta^\beta\partial_\xi^\gamma M(\theta,\xi)
    \]
    with \(|\beta|+|\gamma|\leqslant r+1\), and all mixed derivatives
    \[
      \partial_\theta^\beta\partial_\xi^\gamma \Delta(\theta,\xi)
    \]
    with \(|\beta|+|\gamma|\leqslant r\), are uniformly bounded on
    \[
      \Theta^\circ_{\mathrm{out}}\times
      (\mathbb H_c^m\cap\mathbb B_{4R}^m).
    \]
    In addition, \(\Delta\) has the following anisotropic regularity: whenever
    \(|\gamma|\geqslant1\) and \(|\beta|+|\gamma|\leqslant r+1\), the derivative
    \[
      \partial_\theta^\beta\partial_\xi^\gamma\Delta(\theta,\xi)
    \]
    exists and is uniformly bounded on the same set.

    \item \label{item:Delta_at_0} At \(\xi=0\), one has
    \begin{equation*}
      M(\theta,0_m)=\varphi(\theta), \qquad \Delta(\theta,0_m) = 0_d,
      \qquad \Diff_\xi \Delta(\theta, 0)=[\,\bfL(\theta);\ \mathbf 0_{k\times m}\,].
    \end{equation*}

    \item \label{item:Crplusone-corner-chart} For each \(\theta\in\Theta^\circ_{\mathrm{out}}\), the map \( M(\theta,\cdot): \mathbb
    H_c^m\cap\mathbb B_{4R}^m \to\Man\) is a \(C^{r+1}\) corner chart onto a relatively open
    neighborhood of \(\varphi(\theta)\) in \(\Man\).

    \item \label{item:Cr-J} For \(\theta\in\Theta^\circ_{\mathrm{out}}\), define \(J(\theta,\xi) = \det
    \bigl(\Diff_\xi \Delta (\theta,\xi)^\top \Diff_\xi\Delta(\theta,\xi) \bigr)^{1/2}\). Then \(J\)
    is jointly \(C^{r}\) in \((\theta,\xi)\), all mixed derivatives
    \begin{equation*}
      \partial_\theta^\beta\partial_\xi^\gamma J(\theta,\xi), \qquad |\beta|+|\gamma|\leqslant r,
    \end{equation*}
    are uniformly bounded on \( \Theta^\circ_{\mathrm{out}} \times (\mathbb H_c^m\cap\mathbb
    B_{4R}^m)\), and \( J(\theta, 0)=|\det\bfL(\theta)|\).

    \item \label{item:subset_of_M} For every \(x=\varphi(\theta)\in\Compct_{\Strat}\), \(\Man\cap \mathbb
    B_{2\delta_0}(x)\subset M\bigl(\theta, \mathbb H_c^m\cap \mathbb B_{R}^m\bigr)\).
  \end{enumerate}
\end{lemma}

\begin{proof}
Choose an open set \(\Theta^\circ_{\mathrm{out}}\) such that \(\Theta_{\Compct_{\Strat}}\subset
\Theta^\circ_{\mathrm{out}} \Subset\Theta_{\mathrm{out}}\). Let
\begin{equation}\label{eq:d*}
  d_*
  = \dist\bigl( \overline{\Theta}^\circ_{\mathrm{out}},
  \mathbb R^{m-c}\setminus\Theta_{\mathrm{out}} \bigr)>0,
\end{equation}
and choose \(R>0\) such that \(4R<\min\{d_*,\varepsilon \}\), where \(\varepsilon>0\) is given by
\Cref{ass:standing}~\ref{ass:standing:chart}. It follows from \cref{eq:d*} that the
\(4R\)-neighborhood of \(\Theta^\circ_{\rm out}\) lies in \(\Theta_{\rm out}\). Therefore, if
\(\xi=(\xi_{\mathcal S},\xi_{\mathcal C})\in \mathbb H_c^m\) satisfies \(\|\xi\|<4R\) and
\(\theta\in \Theta^\circ_{\mathrm{out}}\), we have \(\xi_{\mathcal C} \in[0,\varepsilon)^c\) and
\(\theta+\xi_{\mathcal S}\in \Theta_{\mathrm{out}}\). Thus
\begin{equation*}
  M(\theta,\xi)=\Phi(\theta+\xi_{\mathcal S},\xi_{\mathcal C})
\end{equation*}
is well defined on the stated domain \(\Theta^\circ_{\mathrm{out} } \times (\mathbb H_c^m\cap
\mathbb B_{4R}^m)\).

The compact set of chart arguments
\begin{equation*}
  \Compct^m
  = \{(\theta+\xi_{\mathcal S},\xi_{\mathcal C}) : (\theta,\xi)\in
  \overline{\Theta}^\circ_{\mathrm{out}} \times (\mathbb H_c^m\cap \overline{\mathbb B}_{4R}^m)\}
  \subset\mathbb R^m
\end{equation*}
lies in \(\Theta_{\mathrm{out}}\times[0,\varepsilon)^c\). Since \(\Phi\) is \(C^{r+1}\), all
derivatives of \(\Phi\) up to order \(r+1\) are continuous and uniformly bounded in \(\Compct^m\).
This implies that \(M\) and its derivatives up to order \(r+1\) are continuous and uniformly bounded
in \(\overline{\Theta}^\circ_{\mathrm{out}} \times (\mathbb H_c^m\cap \overline{\mathbb
B}_{4R}^m)\).

On the other hand, by \Cref{lem:ad_frame_comp_strat}, the frame field \(\bfQ\) is \(C^{r}\) on
\(\Strat_{\rm out}\). Since \(\varphi:\Theta^\circ_{\rm out}\to\Strat_{\rm out}\) is \(C^{r+1}\),
the pullback \(\theta\mapsto \bfQ(\varphi(\theta))\) is \(C^r\). Therefore
\[
  \Delta(\theta,\xi)
  =
  \bfQ(\theta)^\top\bigl(M(\theta,\xi)-\varphi(\theta)\bigr)
\]
is \(C^r\), and all its mixed derivatives of total order at most \(r\) are uniformly bounded on the
compact chart region.

It remains to record the extra regularity in the \(\xi\)-directions. Let \(|\gamma|\geqslant1\) and
\(|\beta|+|\gamma|\leqslant r+1\). Since the factor \(\varphi(\theta)\) is independent of \(\xi\),
we have
\[
  \partial_\theta^\beta\partial_\xi^\gamma\Delta(\theta,\xi)
  =
  \sum_{\beta_1+\beta_2=\beta}
  c_{\beta_1,\beta_2}
  \bigl(\partial_\theta^{\beta_1}\bfQ(\theta)^\top\bigr)
  \bigl(\partial_\theta^{\beta_2}\partial_\xi^\gamma M(\theta,\xi)\bigr).
\]
Here \(|\beta_1|\leqslant|\beta|\leqslant r\), so the derivatives of \(\bfQ\) are available, and
\[
  |\beta_2|+|\gamma|\leqslant|\beta|+|\gamma|\leqslant r+1,
\]
so the corresponding derivatives of \(M\) are available. All factors are uniformly bounded on the
compact chart region. This proves both the ordinary \(C^r\) statement and the anisotropic
\(\xi\)-regularity asserted in item~\ref{item:Cr-Delta}.

At \(\xi=0\), \(\Phi(\theta,0)=\varphi(\theta)\), hence \(\Delta(\theta,0_m) = 0_d\). Moreover,
\begin{equation*}
  \Diff_\xi \Delta(\theta,0) = \bfQ(\varphi(\theta))^\top \bfA_{\mathrm{ch}}(\varphi(\theta)).
\end{equation*}
Using \(\bfA_{\mathrm{ch}}=\bfQ_{\Man}\bfL\) and \(\bfQ = [\,\bfQ_{\Man}\ \bfN\,]\), we get
\begin{equation*}
  \Diff_\xi\Delta(\theta,0) = [\,\bfL(\theta);\ \mathbf 0_{k\times m}\,].
\end{equation*}
This proves the second item of the lemma.

For fixed \(\theta\), the map \(\xi = (\xi_{\mathcal S}, \xi_{\mathcal C})\mapsto(\theta +
\xi_{\mathcal S}, \xi_{ \mathcal C})\) is a smooth diffeomorphism from \(\mathbb H_c^m\cap\mathbb
B_{4R}^m\) onto an open subset of the corner chart. Composing with \(\Phi\) shows that \(M(\theta,
\cdot)\) is a \(C^{r+1}\) corner chart onto a relatively open
neighborhood of \(\varphi(\theta)\) in \(\Man\). This proves the third item.

For the Jacobian, set
\begin{equation*}
  G(\theta,\xi) = \Diff_\xi \Delta(\theta,\xi)^\top \Diff_\xi \Delta(\theta,\xi) .
\end{equation*}
Since \(\bfQ(\theta)\) is orthogonal and does not depend on \(\xi\), \(G(\theta,\xi) = \Diff_\xi
M(\theta,\xi)^\top \Diff_\xi M(\theta,\xi)\). The differential of the corner chart has rank \(m\) on
the compact closed region \(\|\xi\|\le4R\). Thus \(G\) is uniformly positive definite there. Hence
\(\sqrt{\det G}\) is jointly \(C^{r}\), and all derivatives up to order \(r\) are uniformly bounded.
At \(\xi=0\), the previous computation gives
\[
  G(\theta,0)=\bfL(\theta)^\top\bfL(\theta),
\]
so
\[
  J(\theta,0)=|\det\bfL(\theta)|.
\]
This proves the fourth item of the lemma.

It remains to choose a uniform inner radius. First, by compactness of \(\Theta_{\Compct_{\Strat}}\)
and continuity of the chart, there is \(\delta'_0>0\) such that, for every
\(\theta\in\Theta_{\Compct_{\Strat}}\),
\begin{equation}\label{eq:intermediate_delta}
  \Man\cap\mathbb B^d_{2\delta'_0}(\varphi(\theta))
  \subset M(\theta, \mathbb H_c^m\cap\mathbb B_{4R}^m).
\end{equation}
Indeed, otherwise one could find \((\theta_n,z_n) \in \Theta_{\Compct_{\Strat}}\times \Man\) such
that \(\theta_n \to\theta_* \in \Theta_{\Compct_{\Strat}}\), \(z_n\to \varphi(\theta_*)\) and
\(z_n\notin M({\theta_n}, \mathbb H_c^m\cap\mathbb B_{4R}^m)\). Writing \(z_n=\Phi(\xi_n)\) for
\(n\) large enough, continuity of \(\Phi^{-1}\) yields \(\xi_n\to(\theta_*,0_c)\). Hence \(\xi_n -
(\theta_n,0_c) \to 0_m\), which is a contradiction.

Next define
\begin{equation*}
  m_{R}
  = \min_{\substack{\theta\in\Theta_{\Compct_{\Strat}}, \xi\in\mathbb H_c^m \\
  R\le\|\xi\|\leqslant 4R}} \|M(\theta,\xi) - \varphi(\theta)\|.
\end{equation*}
The minimum is strictly positive because the set is compact and each \(M(\theta,\cdot)\) is
injective. Choose \(\delta_0 >0\) with \(2\delta_0 < \min\{2\delta'_0,m_{R}\}\). If \(z \in
\Man\,\cap\,\mathbb B_{2\delta_0}(\varphi( \theta))\), then \cref{eq:intermediate_delta} gives
\(z=M(\theta,\xi)\) with \(\|\xi\|<4R\). The definition of \(m_{R}\) excludes \(\|\xi\|\geqslant
R\), and therefore \(\|\xi\|<R\). This proves the final inclusion.
\end{proof}

\paragraph{The reconstructed observation map.}
We now record the simple consequence of the tubular coordinates that is used to pass between ambient
observation points in the conical layer and the rescaled variables \((a,x)\). The point is that,
uniformly for \(x\in\Compct_{\Strat}\) and bounded \(a\), the point obtained by moving a distance
\(\sigma a_{\mathcal C}\) in the \(\mathcal C \)-directions and \(\sigma a_{\mathcal N}\) in the
\(\mathcal N\)-directions remains in the tubular neighborhood \(\mathcal W\) provided by
\Cref{lem:orth_tub_coord} for all \(\sigma\) small enough. In these coordinates the identities are
then immediate from the fact that \(\mathcal Z\) and \(\mathcal T\) are inverse maps.

\begin{lemma}[Rescaled tubular coordinate identities]
  \label{lem:rec_map_prop}
  Assume \AssChart. Fix \(A>0\). There exists \(\sigma_0>0\) such that the map
  \( \mathcal T_{\rm rescaled}:(a,x,\sigma) \longmapsto y_\sigma(a,x)\) defined on
  \(\overline{\mathbb B}_A^{c+k}\times \Compct_{\Strat} \times (0,\sigma_0]\) and taking values in
  \(\mathcal U\) is well defined and of class \(C^{r}\). Moreover,
  \begin{equation*}
    \bfQ(x)^\top\bigl(y_\sigma(a,x)-x\bigr) = (0_{m-c},\sigma a_{\mathcal C},\sigma a_{\mathcal N}),
  \end{equation*}
  and
  \begin{equation}\label{eq:Y_tub_coord}
    \pi\bigl(y_\sigma(a,x)\bigr)=x, \qquad u\bigl(y_\sigma(a,x)\bigr)=\sigma a_{\mathcal C},
    \qquad \eta\bigl(y_\sigma(a,x)\bigr)=\sigma a_{\mathcal N}.
  \end{equation}
  Consequently, for \((a,\sigma)\in\overline{\mathbb B}_A^{c+k}\times (0, \sigma_0]\) and for every
  \((y, \sigma)\in \mathcal Y_{A,\Compct_{\Strat}, \sigma_0}\),
  \begin{equation}\label{eq:Y_resc_param}
    a\bigl(y_\sigma(a,x),\sigma\bigr)=a ,\qquad \text{and}
    \qquad y_\sigma\bigl(a(y,\sigma),\pi(y)\bigr)=y.
  \end{equation}
\end{lemma}

\begin{proof}
The maps \(\mathcal Z\) and \(\mathcal T\) are inverse \(C^{r}\) diffeomorphisms between
\(\mathcal{U}\) and \(\mathcal W\). For \(\sigma_0\) small enough, \((x, \sigma a_{\mathcal
C},\sigma a_{\mathcal N})\in \mathcal W\) whenever \(x \in \Compct_{\Strat}\), \(\|a\|\leqslant A\),
and \(0\leqslant \sigma\leqslant \sigma_0\). Thus \(y_\sigma(a,x) = \mathcal T(x,\sigma a_{\mathcal
C}, \sigma a_{\mathcal N})\) is well defined and \(C^{r}\). Since \( y_\sigma(a,x)-x = \sigma
\bfC(x) a_{\mathcal C} + \sigma\bfN(x)a_{\mathcal N}\), orthogonality of \(\bfQ(x)=[\,\bfS(x)\
\bfC(x)\ \bfN(x)\,]\) gives
\begin{equation*}
  \bfQ(x)^\top\bigl(y_\sigma(a,x)-x\bigr)
  = [0_{m-c};\ \sigma a_{\mathcal C}; \sigma a_{\mathcal N}].
\end{equation*}
Also, \(\mathcal Z\bigl(y_\sigma(a,x)\bigr) = \mathcal Z\bigl(\mathcal T(x,\sigma a_{\mathcal
C},\sigma a_{ \mathcal N})\bigr) = (x,\sigma a_{\mathcal C},\sigma a_{\mathcal N})\). This proves
\cref{eq:Y_tub_coord} and, after dividing the last two components by \(\sigma\),
\cref{eq:Y_resc_param}.

Finally, if \((y,\sigma)\in\mathcal Y_{A,\Compct_{\Strat}, \sigma_0}\), then \(u(y)=\sigma
a_{\mathcal C}(y, \sigma)\) and \(\eta(y)=\sigma a_{\mathcal N}(y,\sigma)\). Therefore \(
y_\sigma\bigl(a(y,\sigma),\pi(y)\bigr) = \mathcal T\bigl(\pi(y),u(y),\eta(y)\bigr) = \mathcal
T(\mathcal Z(y)) = y. \)
\end{proof}

\section{Admissible classes and uniform calculus}
\label{app:admissible_classes}

This appendix introduces the admissible classes used to formulate all uniform remainder estimates in
the conical layer. The classes separate three kinds of control: bounded parameter dependence in
\((a,x,\sigma)\), polynomial growth in the rescaled kernel variable \(\zeta\), and Gaussian decay
with respect to the canonical conical model. The closure properties proved below allow the later
appendix estimates to be written in a compact form, without repeating elementary product,
composition, and integration arguments.

\subsection{Definitions}
\label{subsec:admissible_class_definitions}

The scaled kernel variable is \(\zeta=(\zeta_{\mathcal S},\zeta_{\mathcal C})\in \mathbb H_c^m\),
with \(\zeta_{\mathcal S}\in\mathbb R^{m-c}\) and \( \zeta_{\mathcal C} \in [0,\infty)^c\), while
\((a,x,\sigma)\) are regarded as parameters. We write \(a=(a_{\mathcal C},a_{\mathcal N})\) and
denote by \(\Pi_{\mathcal S}\) and \(\Pi_{\mathcal C}\) the projections onto the first \(m-c\) and
last \(c\) tangent coordinates, respectively. For \(c_0>0\), define the canonical conical Gaussian
\(\Gamma_{c_0}(\zeta;a) = \exp\!\left( - c_0\bigl[ \|\zeta_{\mathcal S}\|^2 + \|a_{\mathcal
C}-\zeta_{\mathcal C}\|^2 + \|a_{\mathcal N}\|^2 \bigr] \right)\). For \(A>0\), \(\sigma_0>0\), and
\(R\in(0,\infty]\), set
\begin{equation}\label{eq:mathcal-D}
  \mathcal D_{R,A,\Compct_{\Strat},\sigma_0}
  = \Bigl\{ (\zeta,a,x,\sigma): \zeta\in\mathbb H_c^m, \sigma\|\zeta\|<R,
  \|a\|\leqslant A, x\in\Compct_{\Strat}, 0<\sigma\leqslant\sigma_0 \Bigr\}.
\end{equation}
The corresponding parameter domain is
\begin{equation*}
  \mathcal D_{A,\Compct_{\Strat},\sigma_0}
  = \Bigl\{ (a,x,\sigma): \|a\|\leqslant A, x\in\Compct_{\Strat}, 0<\sigma\leqslant\sigma_0 \Bigr\}.
\end{equation*}
For a multi-index triple \(\gamma=(\alpha,\beta,j)\), where \(\alpha\) differentiates in \(a\),
\(\beta\) differentiates in \(\theta\), and \(j\) differentiates in \(\sigma\), write
\begin{equation*}
  |\gamma|=|\alpha|+|\beta|+j, \qquad \partial_{a,\theta,\sigma}^{\gamma}
  = \partial_a^\alpha\partial_\theta^\beta\partial_\sigma^j.
\end{equation*}

\begin{remark}[Parameter-variable derivatives]
  \label{rem:parameter_variable_derivatives}
  All derivatives in the base variable \(x\in\Compct_{\Strat}\) are understood after pullback by the
  fixed stratum chart \(\varphi:\Theta^\circ_{\mathrm{out}} \supset\Theta_{\Compct_{\Strat}}
  \longrightarrow\Compct_{\Strat}\). Thus \(x\)-derivatives mean ordinary \(\theta\)-derivatives of
  local extensions, restricted back to \(\Theta_{\Compct_{\Strat}}\). The same convention applies
  to the bounded rescaled transverse coordinate \(a \in \overline{\mathbb B}_A\). All estimates
  below are therefore estimates on compact parameter sets, with derivatives computed using
  open-neighborhood extensions.
\end{remark}

\begin{definition}[Admissible classes]
  \label{def:admissible_classes}
  Let \(\ell\in\mathbb N\) and \(\omega>0\). Fix \(A,\sigma_0>0\) and \(R\in(0,\infty]\).

  \begin{enumerate}
    \item A function \(B:\overline{\mathbb B}_A\times \Compct_{\Strat}\times (0,\sigma_0]\to\mathbb
    R\) belongs to \(\mathfrak B_\ell\) if its pullback \( B(a,\varphi (\theta),\sigma)\) is
    \(C^\ell\) in \((a,\theta, \sigma)\) and, for every \(\gamma\) with \(|\gamma| \leqslant \ell\),
    \begin{equation*}
      \sup_{{\mathcal D}_{A,\Compct_{\Strat},\sigma_0}} \bigl| \partial_{a,\theta,\sigma}^{\gamma}
      B(a,\varphi(\theta),\sigma) \bigr|<\infty.
    \end{equation*}

    \item A function \(E:\overline{\mathbb B}_A \times\Compct_{\Strat}\times (0,\sigma_0]\to\mathbb
    R\) belongs to \(\mathfrak E_\ell^\omega\) if its pullback \( E(a, \varphi (\theta),\sigma)\) is
    \(C^\ell\) in \((a, \theta, \sigma)\), and there exists \(c>0\) such that, for every \(\gamma\)
    with \(|\gamma|\leqslant\ell\), there is a constant \(C_\gamma>0\) satisfying
    \begin{equation*}
      \bigl|\partial_{a,\theta,\sigma}^{ \gamma} E(a,\varphi(\theta),\sigma) \bigr|
      \leqslant C_\gamma e^{-c/\sigma^{\omega}},\quad
      \forall (a,\varphi(\theta),\sigma)\in {\mathcal D}_{A,\Compct_{\Strat}, \sigma_0}.
    \end{equation*}

    \item A function \(P: \mathcal D_{R, A, \Compct_{\Strat}, \sigma_0} \to\mathbb R\) belongs to
    \(\mathfrak P_\ell \) if it is measurable in \(\zeta\) and its pullback
    \(P(\zeta;a,\varphi(\theta),\sigma)\) is \(C^\ell\) in the
    parameter variables \((a,\theta, \sigma)\), for every fixed \(\zeta\), and there exists
    \(n\geqslant 0\) such that
    \begin{equation*}
      \bigl| \partial_{a,\theta,\sigma}^{\gamma} P(\zeta; a,\varphi(\theta),\sigma) \bigr|
      \leqslant C_\gamma \bigl(1 + \|\zeta\|^n\bigr)
    \end{equation*}
    for every \(\gamma\) with \(|\gamma|\leqslant\ell\), uniformly on \(\mathcal D_{R, A,
    \Compct_{\Strat}, \sigma_0} \).

    \item A function \(G: \mathcal D_{R, A,\Compct_{\Strat}, \sigma_0} \to\mathbb R\) belongs to
    \(\mathfrak G_\ell\) if it is measurable in \(\zeta\) and its pullback
    \(G(\zeta;a,\varphi(\theta),\sigma)\) is \(C^\ell\) in the
    parameter variables \((a,\theta,\sigma) \), for every fixed \(\zeta\), and there exist \(c_0>0,
    n\geqslant 0\) such that
    \begin{equation*}
      \bigl| \partial_{a,\theta,\sigma}^{\gamma} G(\zeta;a,\varphi(\theta),\sigma) \bigr|
      \leqslant C_\gamma\bigl(1+\|\zeta\|^n\bigr)\Gamma_{c_0}(\zeta;a)
    \end{equation*}
    for every \(\gamma\) with \(|\gamma|\leqslant\ell\), uniformly on \({\mathcal
    D}_{R,A,\Compct_{\Strat},\sigma_0}\).
  \end{enumerate}
\end{definition}

No derivatives in \(\zeta\) are imposed in these definitions, only measurability, which together
with the \(C^\ell\) dependence on \((a,\theta,\sigma)\) guarantees the joint measurability needed
for the integration rule of \Cref{lem:prop_adm_classes}-\ref{item:integr_Gauss_kern} and for
differentiation under the integral sign. This costs nothing here: every kernel arising in the paper
is in fact continuous in \(\zeta\), hence jointly continuous in \((\zeta,a,\theta,\sigma)\). The
open-neighborhood condition
ensures that these are ordinary derivatives of functions defined beyond the boundary faces of the
domains on which the estimates are imposed. Unless stated otherwise, the same notation is used
componentwise for finite-dimensional vector-valued functions.

The next subsection collects the elementary stability properties of these classes.

\subsection{Closure properties}
\label{subsec:admissible_class_closure_properties}

We present the basic rules that make the admissible classes useful in practice. They show that
admissibility is preserved under the operations appearing in the kernel expansion: addition,
multiplication, multiplication by harmless powers of \(\sigma\), smooth composition, and
Gaussian-weighted integration in \(\zeta\). All estimates are uniform over \((a,x,\sigma) \in
\mathcal D_{A,\Compct_{\Strat},\sigma_0}\), and, for kernel-level classes, over the corresponding
truncated kernel domain.

\begin{lemma}[Basic calculus of the admissible classes]
  \label{lem:prop_adm_classes}
  Fix \(A,\sigma_0>0\), and \(R\in(0,\infty]\). Let \(\omega,\omega'>0\) and \(\ell\in\mathbb N\).
  All classes below are taken with respect to the same domain \( \mathcal D_{A, \Compct_{\Strat},
  \sigma_0}\) or \( \mathcal D_{R,A, \Compct_{\Strat}, \sigma_0}\). The following stability
  properties hold.

  \vspace*{-10pt}
  \begin{enumerate}
    \item \label{item:vect_space_and_mon}
    The classes \(\mathfrak B_\ell\), \(\mathfrak E_\ell^\omega\), \(\mathfrak P_\ell\) and
    \(\mathfrak G_\ell\) are vector spaces. Moreover, \(\omega\geqslant\omega'\) entails that
    \(\mathfrak E_\ell^{\omega}\subset\mathfrak E_\ell^{\omega'}\subset \mathfrak B_\ell\).

    \item \label{item:products}
    Let \(B\in\mathfrak B_\ell\), \(E\in\mathfrak E_\ell^\omega\), \(P\in\mathfrak P_\ell\) and
    \(G\in\mathfrak G_\ell\). Then the product functions \(BE,BP\) and \(BG\) satisfy
    \begin{equation*}
      BE\in\mathfrak E_\ell^\omega, \qquad BP\in\mathfrak P_\ell, \qquad BG\in\mathfrak G_\ell.
    \end{equation*}
    If, in addition, for \(i=1,2\), \(E_i\in\mathfrak E_\ell^{\omega_i}\), \(P_i\in\mathfrak P_\ell\),
    \(G_i\in\mathfrak G_\ell\), then, with \(\bar\omega=\min\{\omega_1,\omega_2\}\),
    \begin{equation*}
      E_1E_2 \in\mathfrak E_\ell^{\bar\omega}, \quad P_1P_2 \in\mathfrak P_\ell,
      \quad P_1G_2\in\mathfrak G_\ell, \quad G_1G_2\in\mathfrak G_\ell.
    \end{equation*}

    \item \label{item:inversion_in_B_ell}
    If \(B\in\mathfrak B_\ell\) then \(B^{-1}\in\mathfrak B_\ell\) provided that
    \begin{equation*}
      \inf_{\|a\|\leqslant A,\ x\in\Compct_{\Strat},\ 0<\sigma\leqslant\sigma_0} |B(a,x,\sigma)|>0.
    \end{equation*}
    \item \label{item:insertion_of_the_canonical_Gaussian}
    If \(P\in\mathfrak P_\ell\), then, for every \(c_0>0\), \(\Gamma_{c_0}P \in\mathfrak G_\ell\).

    \item \label{item:integr_Gauss_kern}
    If \(R=\infty\), \(G\in\mathfrak G_\ell\) and \(B(a,x,\sigma) = \int_{\mathbb
    H_c^m}G(\zeta;a,x,\sigma)\,\dd\zeta\), then \(B\in\mathfrak B_\ell\).

    \item \label{item:smooth_composition_on_bounded_ranges}
    Let \(q\ge1\), let \(\Omega\subset\mathbb R^q\) be open, and let \(\psi\in C^\infty(\Omega)\).
    \begin{enumerate}
      \item[(a)] If \(B=(B_1,\dots,B_q)\),  with \(B_i\in
      \mathfrak B_\ell\), and the range of \(B\) on \(\mathcal D_{A,\Compct_{\Strat}, \sigma_0}\) is
      contained in a compact set \(K_\psi \Subset\Omega\), then \(\psi\circ B\in\mathfrak B_\ell\).

      \item[(b)] If \( P=(P_1,\dots,P_q)\) with \(P_i\in
      \mathfrak P_\ell\), and the range of \(P\) on \(\mathcal D_{R,A,\Compct_{\Strat},\sigma_0}\)
      is contained in a compact set \(K_\psi\Subset\Omega\), then \(\psi\circ P\in\mathfrak
      P_\ell\).
    \end{enumerate}
    The same conclusions hold without the compact-range assumption if \(\psi\) and all of its
    derivatives up to order \(\ell\) are globally bounded on \(\Omega\).

    \item
    \label{item:mult_div_power_sigma}
    Let \(B\in\mathfrak B_\ell\), \(E\in\mathfrak E_\ell^\omega\), \(P\in\mathfrak P_\ell\) and
    \(G\in\mathfrak G_\ell\). If \(q\in\mathbb N\), then
    \begin{equation*}
      \sigma^qB\in\mathfrak B_\ell, \qquad \sigma^qE\in\mathfrak E_\ell^\omega,
      \qquad \sigma^qP\in\mathfrak P_\ell, \qquad \sigma^qG\in\mathfrak G_\ell.
    \end{equation*}
    If \(M\in\mathbb N\) and \(E\in\mathfrak E_\ell^\omega\), then \(\sigma^{-M}E \in \mathfrak
    E_\ell^\omega\).
  \end{enumerate}
\end{lemma}

\begin{proof}
Constants may change from line to line, but are uniform on the relevant domain.

\medskip \noindent\textbf{(1) Vector spaces and monotonicity.} The vector-space properties follow
directly from the definitions, using the smaller of the two exponential decay constants in the case
of \(\mathfrak E_\ell^\omega\), and the smaller of the two Gaussian constants in the case of
\(\mathfrak G_\ell\). If \(\omega_1\geqslant\omega_2\), then, for \(0<\sigma\leqslant\sigma_0\),
\begin{equation*}
  \sigma^{-\omega_1} = \sigma^{-\omega_2}\sigma^{-(\omega_1-\omega_2)}
  \geqslant \sigma_0^{-(\omega_1-\omega_2)}\sigma^{-\omega_2}.
\end{equation*}
Thus every \(e^{-c\sigma^{-\omega_1}}\)-bound implies an \(e^{-c'\sigma^{-\omega_2}}\)-bound, with
\(c'=c\sigma_0^{-(\omega_1-\omega_2)}.\) Therefore \(\mathfrak E_\ell^{\omega_1}\subset\mathfrak
E_\ell^{\omega_2}\). Finally, since \(e^{-c\sigma^{-\omega}} \leqslant 1\), exponential bounds imply
uniform boundedness, and hence the inclusion \(\mathfrak E_\ell^\omega\subset\mathfrak B_\ell.\)

\medskip \noindent\textbf{(2) Products.} The proof is an immediate consequence of Leibniz' rule. For
\(|\gamma|\leqslant\ell\),
\begin{equation*}
  \partial^\gamma(FH)
  = \sum_{\gamma_1+\gamma_2=\gamma} C_{\gamma_1,\gamma_2}
  (\partial^{\gamma_1}F)(\partial^{\gamma_2}H).
\end{equation*}

If one factor lies in \(\mathfrak B_\ell\), its derivatives are uniformly bounded, so multiplying by
that factor preserves the defining estimates of \(\mathfrak E_\ell^\omega\), \(\mathfrak P_\ell\),
and \(\mathfrak G_\ell\). Hence
\begin{equation*}
  BE\in\mathfrak E_\ell^\omega,\qquad BP\in\mathfrak P_\ell,\qquad BG\in\mathfrak G_\ell.
\end{equation*}
For products of two exponential terms, first use item~\ref{item:vect_space_and_mon} to place both
factors in \(\mathfrak E_\ell^{\bar\omega}\), where \(\bar\omega=\min\{\omega_1,\omega_2\}.\) Then
each term in Leibniz' rule is bounded by \(C
e^{-c_1\sigma^{-\bar\omega}}e^{-c_2\sigma^{-\bar\omega}} \leqslant C e^{-c\sigma^{-\bar\omega}},\)
for some \(c>0\). Thus \(E_1E_2\in\mathfrak E_\ell^{\bar\omega}.\)

For polynomial factors, use \((1+\|\zeta\|^{n_1})(1+\|\zeta\|^{n_2}) \leqslant
C(1+\|\zeta\|^{n_1+n_2}).\) This gives \(P_1P_2\in\mathfrak P_\ell.\) If one factor is Gaussian, the
same polynomial estimate applies and the Gaussian factor is retained, so \(P_1G_2\in\mathfrak
G_\ell.\) Finally, if \(G_i\) has Gaussian constant \(c_i>0\), then
\(\Gamma_{c_1}(\zeta;a)\Gamma_{c_2}(\zeta;a) = \Gamma_{c_1+c_2}(\zeta;a),\) which gives
\(G_1G_2\in\mathfrak G_\ell.\)

\medskip \noindent\textbf{(3) Inversion in \(\mathfrak B_\ell\).} Since \(B\) is uniformly bounded
away from zero, the reciprocal map \(u\mapsto u^{-1}\) is smooth on an open set containing the range
of \(B\). By the Fa\`a di Bruno formula, every derivative \(\partial^\gamma(B^{-1})\),
\(|\gamma|\leqslant\ell\), is a finite sum of terms of the form
\begin{equation*}
  \frac{ \prod_{\mu=1}^N \partial^{\gamma_\mu}B }{ B^{N+1} },
  \qquad \sum_{\mu=1}^N|\gamma_\mu|\le|\gamma|.
\end{equation*}
The numerator is uniformly bounded because \(B\in\mathfrak B_\ell\), and the denominator is
uniformly bounded away from zero. Therefore \(B^{-1}\in\mathfrak B_\ell.\)

\medskip \noindent\textbf{(4) Insertion of the canonical Gaussian.} The Gaussian
\(\Gamma_{c_0}(\zeta;a)\) depends only on \((\zeta,a)\). For every multi-index \(\mu\) in the
\(a\)-variables, there exists a polynomial \(Q_\mu\), of degree at most \(|\mu|\), such that
\(\partial_a^\mu\Gamma_{c_0}(\zeta;a) = Q_\mu(a_{\mathcal C}-\zeta_{\mathcal C},a_{\mathcal
N})\Gamma_{c_0}(\zeta;a).\) Since \(\|a\|\leqslant A\), this polynomial satisfies
\(|Q_\mu(a_{\mathcal C}-\zeta_{\mathcal C},a_{\mathcal N})| \leqslant C_\mu(1+\|\zeta\|^{|\mu|}).\)
Hence, for \(|\mu|\leqslant\ell\), \(|\partial_a^\mu\Gamma_{c_0}(\zeta;a)| \leqslant
C_\mu(1+\|\zeta\|^{|\mu|})\Gamma_{c_0}(\zeta;a).\) This estimate, Leibniz' rule and the defining
bounds for \(P\) imply that there exists \(n\geqslant 0\) such that every admissible derivative of
\(\Gamma_{c_0}P\) is bounded by \(C(1+\|\zeta\|^{n+\ell})\Gamma_{c_0}(\zeta;a). \) Thus
\(\Gamma_{c_0}P \in \mathfrak G_\ell.\)

\medskip \noindent\textbf{(5) Integration of Gaussian kernels.} Let \(G\in\mathfrak G_\ell\). Then,
for some \(c_0>0\),
\begin{equation*}
  |\partial^\gamma G(\zeta;a,\varphi(\theta),\sigma)|
  \leqslant C_\gamma(1+\|\zeta\|^n)\Gamma_{c_0}(\zeta;a), \qquad |\gamma|\leqslant\ell.
\end{equation*}
Because \(\|a\|\leqslant A\), \(\|a_{\mathcal C}-\zeta_{\mathcal C}\|^2 \geqslant
\frac12\|\zeta_{\mathcal C}\|^2-\|a_{\mathcal C}\|^2 \geqslant \frac12\|\zeta_{\mathcal
C}\|^2-A^2.\) Therefore \(\Gamma_{c_0}(\zeta;a) \leqslant e^{c_0A^2} e^{-c_0\|\zeta_{\mathcal
S}\|^2} e^{-(c_0/2)\|\zeta_{\mathcal C}\|^2}.\) The right-hand side is integrable over \(\mathbb
H_c^m=\mathbb R^{m-c}\times[0,\infty)^c\), even after multiplication by \(1+\|\zeta\|^n\). Thus the
derivatives \(\partial^\gamma G\) admit a common integrable majorant, uniformly in
\((a,\theta,\sigma)\). Differentiation under the integral sign gives, for \(|\gamma|\leqslant\ell\),
\begin{equation*}
  \partial^\gamma B(a,\varphi(\theta),\sigma)
  = \int_{\mathbb H_c^m} \partial^\gamma G(\zeta;a,\varphi(\theta),\sigma)\,\dd\zeta,
\end{equation*}
and the right-hand side is uniformly bounded. Hence \(B\in\mathfrak B_\ell.\)

\medskip \noindent\textbf{(6) Smooth compositions.} We use the multivariate Fa\`a di Bruno formula.

For part (a), \(\operatorname{Im}B\subset K_\psi\Subset\Omega\). Hence all derivatives of \(\psi\)
up to order \(\ell\) are bounded on \(K_\psi\). Every derivative \(\partial^\gamma(\psi\circ B)\),
\(|\gamma|\leqslant\ell\), is a finite sum of products of such bounded derivatives of \(\psi\),
evaluated at \(B\), and derivatives of the components \(B_i\) of order at most \(\ell\). Since each
\(B_i\in\mathfrak B_\ell\), all these terms are uniformly bounded. Therefore \(\psi\circ
B\in\mathfrak B_\ell.\)

For part (b), the same formula applies. The derivatives of \(\psi\) are again bounded on the compact
range \(K_\psi\). A derivative of order at most \(\ell\) of \(\psi\circ P\) is a finite sum of terms
containing at most \(\ell\) factors of derivatives of the components \(P_i\). Each such factor is
bounded by \(C(1+\|\zeta\|^n)\), for some \(n\geqslant 0\), and therefore each term is bounded by
\(C (1 + \|\zeta\|^{n\ell} ) \). Thus \(\psi\circ P\in \mathfrak P_\ell.\) If the derivatives of
\(\psi\) up to order \(\ell\) are globally bounded on \(\Omega\), the same argument applies without
assuming compactness of the range.

\medskip \noindent\textbf{(7) Multiplication and division by powers of \(\sigma\).} Let
\(q\in\mathbb N\). For every \(i\leqslant\ell\), the derivative \(\partial_\sigma^i\sigma^q\) is
either zero or a nonnegative power of \(\sigma\), and is therefore uniformly bounded on
\((0,\sigma_0]\). Leibniz' rule then gives
\begin{equation*}
  \sigma^qB\in\mathfrak B_\ell, \qquad \sigma^qE\in\mathfrak E_\ell^\omega,
  \qquad \sigma^qP\in\mathfrak P_\ell, \qquad \sigma^qG\in\mathfrak G_\ell.
\end{equation*}
It remains to consider \(\sigma^{-M}E\), with \(M\in\mathbb N\) and \(E\in\mathfrak
E_\ell^\omega\). For \(|\gamma|=|\alpha|+|\beta|+j\leqslant\ell\), Leibniz' rule gives
\begin{equation*}
  \partial^\gamma(\sigma^{-M}E)
  = \sum_{i=0}^j c_{j,i}\, \sigma^{-M-i}\,
  \partial_a^\alpha\partial_\theta^\beta\partial_\sigma^{j-i}E.
\end{equation*}
Thus it suffices to absorb powers of \(\sigma^{-1}\) into the exponential decay. If \(L\geqslant
0\), then
\begin{equation*}
  \sigma^{-L}e^{-c\sigma^{-\omega}}
  = \bigl(\sigma^{-L}e^{-(c/2)\sigma^{-\omega}}\bigr) e^{-(c/2)\sigma^{-\omega}}.
\end{equation*}
With \(u=\sigma^{-\omega}\), the first factor becomes \(u^{L/\omega}e^{-(c/2)u},\) which is bounded
for \(u\geqslant \sigma_0^{-\omega}\). Hence \(\sigma^{-L}e^{-c\sigma^{-\omega}} \leqslant C_L
e^{-(c/2)\sigma^{-\omega}}, 0<\sigma \leqslant \sigma_0.\) Every admissible derivative of
\(\sigma^{-M} E\) therefore satisfies an exponential bound of the same type. Hence
\(\sigma^{-M}E\in\mathfrak E_\ell^\omega.\)
\end{proof}

\section{Analytic estimates for the local tangent-cone expansion}
\label{app:formal_local_expansion}

This appendix proves the analytic estimates used in the local tangent-cone expansion. The main task
here is to control the exact local kernel after the rescaling \(\xi=\sigma\zeta\). The estimates are
arranged in the order in which they enter the expansion. We first derive the scaled Taylor
expansions of the chart components, the amplitude, and the exponent. We then prove the support,
Gaussian-comparison, and tail estimates needed to turn these local Taylor expansions into a uniform
conical Laplace expansion.

\paragraph{Hypotheses in force.}
The geometric estimates for \(\Delta\), \(\Psi\), \(M\), \(\bfL\), and the conical Gaussian use only
\AssChart. Estimates involving the amplitude \(\mathcal A=\rho J\), the exact scaled kernel
\(\mathscr K_\sigma\), or the local integral \(I_\sigma\) additionally use
\Cref{ass:standing}~\ref{ass:standing:measure} and \ref{ass:standing:density}. The positivity
item~\ref{ass:standing:positivity} is not used in this section.

\subsection{Scaled Taylor expansions of the local factors}
\label{app:proof_of_lem:adm_scaled_exp}

We begin with the Taylor estimates for the local chart factors and for the amplitude. These
estimates are obtained by expanding in the original corner variable \(\xi\), then inserting the
blow-up \(\xi=\sigma\zeta\). The admissible-class bounds follow from the uniform derivative bounds
on the compact chart region. Recall that \(\Delta(\theta, \xi)\) is defined by \eqref{eq:Delta} and
\(\mathcal A(\theta,\xi) =\rho \bigl(M(\theta,\xi)\bigr) \,J(\theta,\xi)\). For
\(x=\varphi(\theta)\), this two-argument form is the stratum-coordinate expression of the base-point
amplitude of \eqref{eq:defAx}: since \(\Phi_x(\xi)=M(\theta,\xi)\) and \(J_x(\xi)=J(\theta,\xi)\),
we have \(\mathcal A_x(\xi)=\rho\bigl(\Phi_x(\xi)\bigr)\,J_x(\xi)=\mathcal A(\theta,\xi)\). As with
\(\bfL\), \(\bfQ\), and \(\Phi_x\), we retain the two-argument notation \(\mathcal A(\theta,\xi)\)
only in the derivative estimates of this subsection, where \(\theta\) is differentiated; elsewhere
we use the frozen-base-point form \(\mathcal A_x\).

\begin{lemma}[Admissible scaled Taylor expansion of the chart displacement]
  \label{lem:adm_scaled_exp_delta}
  Assume \AssChart{} with \(r\geqslant 2\). Let \(A,\sigma_0>0\) and \(R_0\in(0,4R)\). Then there
  exists \(\mathscr R_\Delta\in \mathfrak P_{r-2}\) defined on \(\mathcal
  D_{R_0,A,\Compct_{\Strat},\sigma_0}\), such that
  \begin{equation}\label{eq:exp_Delta}
    \sigma^{-1}\Delta(\theta,\sigma\zeta)
    = \begin{bmatrix} \bfL(\theta)\\ \mathbf 0_{k\times m}
    \end{bmatrix}\,\zeta + \frac\sigma2\Diff_\xi^2 \Delta(\theta, 0)[\zeta,\zeta]
    + \sigma^2 \mathscr R_\Delta(\zeta;a,\varphi(\theta),\sigma).
  \end{equation}
\end{lemma}

\begin{proof}
Fix \(x=\varphi(\theta)\in\Compct_{\Strat}\). Taylor's theorem in the \(\xi\)-variable, with
integral remainder, gives for \(\|\xi\|<4R\),
\begin{align*}
  \Delta(\theta,\xi) &= \Delta(\theta,0) + \Diff_\xi \Delta(\theta,0) \xi
  +\frac12 \Diff^2_\xi\Delta(\theta,0)[\xi,\xi]
  + \int_0^1\frac{(1-t)^2}{2}\,\Diff_\xi^3\Delta (\theta,t\xi)[\xi,\xi,\xi]\,dt, \\
  &= \begin{bmatrix} \bfL(x)\xi \\ 0_k \end{bmatrix} + \frac12\Diff_\xi^2\Delta(\theta,0) [\xi,\xi]
  + \int_0^1\frac{(1-t)^2}{2}\, \Diff_\xi^3 \Delta(\theta, t\xi)[\xi,\xi,\xi]\,dt,
\end{align*}
where in the second line we used item~\ref{item:Delta_at_0} of \Cref{lem:unif_par_comp_strat}.
Substituting \(\xi = \sigma\zeta\) and dividing by \(\sigma\) yields \cref{eq:exp_Delta}, with
\begin{equation}\label{eq:def_Rg_precise}
  \mathscr R_\Delta(\zeta;a,x,\sigma)
  =
  \int_0^1\frac{(1-t)^2}{2}\,
  \Diff_\xi^3 \Delta(\theta,t\sigma\zeta)[\zeta,\zeta,\zeta]\,\dd t .
\end{equation}
It remains to check admissibility. Since \(\sigma\| \zeta\|<R_0<4R\), all points \(t\sigma\zeta\),
\(t\in[0,1]\), remain in the compact \(\xi\)-region \(\{\xi\in\mathbb H_c^m:\ \|\xi\|\leqslant
R_0\}.\) We shall use the anisotropic bounds of \Cref{lem:unif_par_comp_strat}: whenever at least
one \(\xi\)-derivative is present and the total \((\theta,\xi)\)-order is at most \(r+1\), the
corresponding derivative of \(\Delta\) is uniformly bounded on this compact region.

Let \(|\alpha|+|\beta|+j\leqslant r-2\). The remainder \(\mathscr R_\Delta\) is independent of
\(a\), so \(\partial_a^\alpha\mathscr R_\Delta=0\) unless \(\alpha=0\). For \(\alpha=0\),
differentiating under the integral sign gives
\begin{align*}
  \partial_\theta^\beta\partial_\sigma^j
  \mathscr R_\Delta(\zeta;a,\varphi(\theta),\sigma)
  &=
  \int_0^1
  \frac{(1-t)^2}{2}\,t^j\,
  \bigl(\partial_\theta^\beta\Diff_\xi^{j+3}\Delta\bigr)
  (\theta,t\sigma\zeta)
  [\zeta,\ldots,\zeta]\,\dd t .
\end{align*}
The multilinear form is evaluated on \(j+3\) copies of \(\zeta\). Since
\[
  |\beta|+(j+3)\leqslant r+1,
\]
the anisotropic regularity in item~\ref{item:Cr-Delta} of \Cref{lem:unif_par_comp_strat} applies.
Hence
\[
  \bigl|
  \partial_a^\alpha\partial_\theta^\beta\partial_\sigma^j
  \mathscr R_\Delta(\zeta;a,\varphi(\theta),\sigma)
  \bigr|
  \leqslant
  C_{\alpha,\beta,j}(1+\|\zeta\|^{j+3})
  \leqslant
  C'_{\alpha,\beta,j}(1+\|\zeta\|^{r+1}).
\]
Therefore \(\mathscr R_\Delta\in\mathfrak P_{r-2}\).
\end{proof}

\begin{lemma}[Admissible scaled Taylor expansion of the local amplitude]
  \label{lem:adm_scaled_exp_amplitude}
  Assume \ref{ass:standing:measure}, \ref{ass:standing:chart}, and~\ref{ass:standing:density} of
  \Cref{ass:standing} with \(r\geqslant 2\). Let \(A,\sigma_0>0\) and \(R_0\in(0,4R)\). Then there
  exists \(\mathscr R_{\mathcal A}\in \mathfrak P_{r-2}\) defined on \(\mathcal D_{R_0,A,\Compct_{
  \Strat}, \sigma_0}\), such that
  \begin{equation}\label{eq:ampl_exp}
    \mathcal A(\theta,\sigma\zeta) = \rho(\varphi (\theta))J(\theta,0)
    + \sigma \,\nabla_\xi \mathcal A(\theta, 0)^\top \zeta
    + \sigma^2 \mathscr R_{\mathcal A}(\zeta;a,\varphi( \theta),\sigma).
  \end{equation}
\end{lemma}

\begin{proof}
Fix \(x=\varphi(\theta)\in\Compct_{\Strat}\) and consider the amplitude \(\mathcal A(\theta,\xi)
=\rho(M(\theta,\xi))\,J(\theta,\xi)\).
\AssDens{} yields \(\rho\circ M \in C^{r}\), and \(J\) is \(C^{r}\) by
\Cref{lem:unif_par_comp_strat}. Hence, the map \((\theta,\xi)\mapsto\mathcal A(\theta, \xi)\) is
jointly \(C^{r}\). Since \(r\geqslant 2\), Taylor's theorem at \(\xi=0\) gives
\begin{equation*}
  \mathcal A(\theta,\xi) = \mathcal A(\theta, 0) + \Diff_\xi\mathcal A(\theta, 0)[\xi]
  + \int_0^1(1-t)\,\Diff_\xi^2\mathcal A(\theta, t\xi) [\xi,\xi]\,dt.
\end{equation*}
With \(\xi=\sigma\zeta\), this becomes \cref{eq:ampl_exp}, where
\begin{equation}\label{eq:def_RA_precise}
  \mathscr R_{\mathcal A}(\zeta;a,x,\sigma)
  =\int_0^1 (1-t)\,\Diff_\xi^2\mathcal A(\theta, t\sigma\zeta) [\zeta,\zeta]\,dt.
\end{equation}
The same compactness argument applies. Derivatives in \(a\) vanish unless \(\alpha=0\). For
\(\alpha=0\), differentiating \cref{eq:def_RA_precise} under the integral sign gives terms
involving
\[
  \partial_\theta^\beta\Diff_\xi^{j+2}\mathcal A(\theta,t\sigma\zeta)[\zeta,\ldots,\zeta],
\]
evaluated on \(j+2\) copies of \(\zeta\). Since \(|\beta|+j+2\leqslant r\), the \(C^r\)-regularity
of \(\mathcal A\) suffices. Therefore
\begin{equation*}
  \bigl| \partial_a^\alpha\partial_\theta^\beta\partial_\sigma^j
  \mathscr R_{\mathcal A}(\zeta;a,\varphi(\theta),\sigma) \bigr|
  \leqslant C_{\alpha,\beta,j}(1+\|\zeta\|^{2+j}) \leqslant C'_{\alpha,\beta,j}(1+\|\zeta\|^{r}),
\end{equation*}
for all \(|\alpha|+|\beta|+j\leqslant r-2\). Hence \(\mathscr R_{\mathcal A}\in \mathfrak P_{r-2}\).
\end{proof}

\subsection{Expansion of the scaled exponent}
\label{app:proof_lem:adm_expon_exp}

We next insert the scaled chart expansions into the squared-distance exponent. The leading term is
the linearized exponent \(\Psi\), the first correction is the cubic term governed by \(\Lambda\),
see \cref{eq:Lambda} below. All remaining terms are collected into an admissible polynomial
remainder. The first variation of the exponent is the polynomial
\(\Lambda_x(\zeta;a)\) defined by
\begin{equation}\label{eq:Lambda}
  \Lambda_x(\zeta;a)= \langle \bfL(\theta)\zeta,\ \Diff^2_\xi \Delta_{1:m}(\theta,0)[\zeta,\zeta]
  \rangle - \langle a ,\ \Diff^2_\xi\Delta_{(m-c+1):d} (\theta,0)[\zeta,\zeta]\rangle,
\end{equation}
where \(a=(a_{\mathcal C},a_{\mathcal N})\in\mathbb R^c \times\mathbb R^k\) and
\(\Diff^2_\xi\Delta_{i:j}\) is the bilinear form obtained by differentiating twice with respect to
\(\xi\) the coordinates included between \(i\) and \(j\) of the vector field \(\Delta(\theta,\xi)\).
The normalization is chosen so that the coefficient of \(\sigma\) in \(\Psi_\sigma\) is
\(\Lambda_x/2\).

\begin{lemma}[Admissible expansion of the exponent]
  \label{lem:admi_expo_exp}
  Assume \AssChart{} with \(r\geqslant 2\). Let \(A,\sigma_0>0\) and \(R_0\in(0,4R)\). Then there
  exists \(\mathscr R_\Psi\in \mathfrak P_{r-2}\) defined on \(\mathcal
  D_{R_0,A,\Compct_{\Strat},\sigma_0}\), such that, for \(x = \varphi(\theta)\),
  \begin{equation}\label{eq:exponent_exp}
    \Psi_\sigma(\zeta;a,x) = \frac{\| y_\sigma(a,x)
    - M(\theta, \sigma\zeta)\|^2}{2\sigma^2} =\Psi(\zeta;a,x) +
    \frac{\sigma}{2}\,\Lambda_x (\zeta;a) + \sigma^2 \mathscr R_\Psi(\zeta;a,x,\sigma).
  \end{equation}
  Moreover, there exists a constant \(C_A>0\) such that, uniformly on \(\mathcal
  D_{R_0,A,\Compct_{\Strat},\sigma_0}\),
  \begin{equation}\label{eq:exponent_remainder_pointwise_degree6}
    |\mathscr R_\Psi(\zeta;a,x,\sigma)| \le C_A(1+\|\zeta\|^6).
  \end{equation}
\end{lemma}

\begin{proof}
Define
\begin{align*}
  Q_0(\zeta;a,x) &= \bigl[(\bfL(x)\zeta)_{1:(m-c)}; \,-(\bfL(x)\zeta)_{(m-c+1):m};\, 0_k
  \bigr] + [0_{m-c}; a], \\
  Q_1(\zeta;a,x) &= \frac12\bigl[\Diff^2_\xi \Delta_{1:(m-c)} (\theta,0)[\zeta,\zeta];
  - \Diff^2_\xi\Delta_{(m-c+1):d} (\theta,0)[\zeta,\zeta] \bigr], \\
  Q_2(\zeta;a,x,\sigma) &= \bigl[\mathscr R_\Delta (\zeta; a,x,\sigma)_{1:(m-c)};
  -\mathscr R_\Delta(\zeta; a,x, \sigma)_{(m-c+1):d} \bigr].
\end{align*}
Using the expansion of \(\Delta\) provided by \eqref{eq:exp_Delta}, and the definition of
\(\Psi_\sigma\), we get
\begin{align*}
  2\sigma^2\Psi_\sigma(\zeta;a,x) &= \big\|\bfQ(x)^\top (y_\sigma (a,x)
  - x)-\bfQ(x)^\top(M(\theta,\sigma\zeta)-x) \big\|^2 \\
  &= \sigma^2\big\|[0_{m-c}; a] - \sigma^{-1} \Delta(\theta,\sigma\zeta)\big\|^2 \\
  &=\sigma^2\|Q_0+\sigma Q_1+\sigma^2Q_2\|^2.
\end{align*}
Expanding the square gives
\begin{equation}\label{eq:exponent_exp_via_Q}
  \Psi_\sigma = \frac12\|Q_0\|^2 + \sigma Q_0^\top Q_1 + \sigma^2\left( Q_0^\top Q_2
  + \frac12\|Q_1\|^2 + \sigma Q_1^\top Q_2 + \frac{\sigma^2}{2}\|Q_2\|^2 \right).
\end{equation}
The zeroth-order term \((1/2)\|Q_0\|^2\) is exactly \(\Psi\). For the coefficient of the linear term
in \(\sigma\),
\begin{equation*}
  2Q_0^\top Q_1 = \Lambda_x(\zeta;a).
\end{equation*}
Thus \(Q_0^\top Q_1=\Lambda_x/2\). The desired expansion \eqref{eq:exponent_exp} now follows from
\cref{eq:exponent_exp_via_Q} by setting
\begin{equation}\label{eq:def_RPsi_precise}
  \mathscr R_\Psi = Q_0^\top Q_2 + \frac12\|Q_1\|^2 + \sigma Q_1^\top Q_2
  + \frac{\sigma^2}{2}\|Q_2\|^2.
\end{equation}
We now check admissibility. By construction and
\Cref{lem:ad_frame_comp_strat,lem:unif_par_comp_strat,lem:adm_scaled_exp_delta},
\begin{equation*}
  Q_0\in\mathfrak P_{r}, \qquad Q_1\in\mathfrak P_{r-1}, \qquad Q_2\in\mathfrak P_{r-2}.
\end{equation*}
The closure properties of \Cref{lem:prop_adm_classes} give
\begin{equation*}
  Q_0^\top Q_2\in\mathfrak P_{r-2}, \qquad \|Q_1\|^2\in\mathfrak P_{r-2},
  \qquad Q_1^\top Q_2\in\mathfrak P_{r-2}, \qquad \|Q_2\|^2\in\mathfrak P_{r-2}.
\end{equation*}
Multiplication by the bounded factors \(\sigma\) and \(\sigma^2\) does not change admissibility.
Therefore
\begin{equation*}
  \mathscr R_\Psi\in\mathfrak P_{r-2}.
\end{equation*}
Finally, we prove the sharper pointwise estimate. If \(\|a\|\leqslant A\), then uniformly in
\(x\in\Compct_{\Strat}\), \(\|Q_0(\zeta;a,x)\|\leqslant C_A(1+\|\zeta\|),
\|Q_1(\zeta;a,x)\|\leqslant C\|\zeta\|^2.\) Moreover, the explicit formula \cref{eq:def_Rg_precise},
together with \(\sigma\|\zeta\|<R_0\), imply \(\|Q_2(\zeta;a,x,\sigma)\|\leqslant C\|\zeta\|^3.\)
Substituting these three estimates into \cref{eq:def_RPsi_precise} and using
\(0<\sigma\leqslant\sigma_0\), we obtain
\begin{equation*}
  |\mathscr R_\Psi| \leqslant C_A(1+\|\zeta\|)\|\zeta\|^3 + C\|\zeta\|^4 + C\sigma\|\zeta\|^5
  + C\sigma^2\|\zeta\|^6 \leqslant C_A(1+\|\zeta\|^6),
\end{equation*}
after increasing \(C_A\) if necessary. This proves \cref{eq:exponent_remainder_pointwise_degree6}.
\end{proof}

\begin{lemma}[Polynomial bounds for the exact scaled exponent]
  \label{lem:exact_scaled_exponent_poly_bounds}
  Assume \AssChart. Let \(A,\sigma_0>0\), \(R_0\in(0,4R)\) and \(0\leqslant s\leqslant r\). Then the
  exact scaled
  exponent
  \[
    \Psi_\sigma(\zeta;a,x)
    =
    \frac{\|y_\sigma(a,x)-M(\theta,\sigma\zeta)\|^2}{2\sigma^2},
    \qquad x=\varphi(\theta),
  \]
  belongs to \(\mathfrak P_s\) on
  \(\mathcal D_{R_0,A,\Compct_{\Strat},\sigma_0}\).
\end{lemma}

\begin{proof}
By the rescaled tubular identity,
\[
  \bfQ(x)^\top(y_\sigma(a,x)-x)
  =
  \sigma[0_{m-c};a_{\mathcal C};a_{\mathcal N}].
\]
Using the definition of \(\Delta\), we therefore have
\[
  \Psi_\sigma(\zeta;a,x)
  =
  \frac12
  \left\|
    [0_{m-c};a_{\mathcal C};a_{\mathcal N}]
    -
    \sigma^{-1}\Delta(\theta,\sigma\zeta)
  \right\|^2 .
\]
Set \( V_\sigma(\zeta;\theta) = \sigma^{-1}\Delta(\theta,\sigma\zeta)\). Since
\(\Delta(\theta,0)=0\), the fundamental theorem of calculus gives the exact identity
\[
  V_\sigma(\zeta;\theta)
  =
  \int_0^1
  \Diff_\xi\Delta(\theta,t\sigma\zeta)[\zeta]\,\dd t .
\]
We first prove that \(V_\sigma\in\mathfrak P_s\). The function \(V_\sigma\) is independent of \(a\),
so all nonzero \(a\)-derivatives are trivial. Let \(|\beta|+j\leqslant s\). Differentiating under
the integral sign gives
\[
  \partial_\theta^\beta\partial_\sigma^j
  V_\sigma(\zeta;\theta)
  =
  \int_0^1
  t^j
  \bigl(\partial_\theta^\beta\Diff_\xi^{j+1}\Delta\bigr)
  (\theta,t\sigma\zeta)
  [\zeta,\ldots,\zeta]\,\dd t ,
\]
where the multilinear form is evaluated on \(j+1\) copies of \(\zeta\). Because
\[
  |\beta|+(j+1)\leqslant s+1\leqslant r+1
\]
and \(j+1\geqslant1\), the anisotropic regularity in \Cref{lem:unif_par_comp_strat} applies. Hence
\[
  \bigl|
  \partial_\theta^\beta\partial_\sigma^j
  V_\sigma(\zeta;\theta)
  \bigr|
  \leqslant
  C_{\beta,j}(1+\|\zeta\|^{j+1})
  \leqslant
  C'_{\beta,j}(1+\|\zeta\|^{s+1}),
\]
uniformly on \(\mathcal D_{R_0,A,\Compct_{\Strat},\sigma_0}\). Thus \(V_\sigma\in\mathfrak P_s\).

The affine vector \( a\mapsto [0_{m-c};a_{\mathcal C};a_{\mathcal N}] \) also belongs to \(\mathfrak
P_s\) on the same domain. Since \(\mathfrak P_s\) is stable under addition and multiplication by
\Cref{lem:prop_adm_classes}, the squared norm
\[
  \frac12\left\|
    [0_{m-c};a_{\mathcal C};a_{\mathcal N}]-V_\sigma(\zeta;\theta)
  \right\|^2
\]
belongs to \(\mathfrak P_s\). This proves the claim.
\end{proof}

\subsection{Support cutoffs and zero extension}
\label{subsec:app_cutoff_extension_tools}

Before proving Gaussian estimates on the full scaled domain, we present two elementary localization
facts: the growing cutoff \(\vartheta_\sigma\), which restricts to \(\|\zeta\|\leqslant
2\sigma^{-1/4}\), and a zero-extension principle for kernels supported away from the truncation
boundary. Let
\[
  \vartheta(w)=\chi(\|w\|^2), \qquad w\in\mathbb R^m,
\]
and define the growing cutoff
\[
  \vartheta_\sigma(\zeta)
  =
  \vartheta(\sigma^{1/4}\zeta)
  =
  \chi(\sigma^{1/2}\|\zeta\|^2).
\]
Then \(\vartheta_\sigma=1\) for \(\|\zeta\|\leqslant\sigma^{-1/4}\) and
\(\vartheta_\sigma=0\) for \(\|\zeta\|\geqslant2\sigma^{-1/4}\).

\begin{lemma}[Admissibility of the growing cutoff]
  \label{lem:admissibility_of_growing_cutoff}
  Let \(A,\sigma_0>0\). Then \(\vartheta_\sigma\) defined on
  \(\mathcal D_{\infty, A,\Compct_{\Strat},\sigma_0}\) belongs to \(\mathfrak P_s\) for every
  \(s\geqslant 0\).
\end{lemma}

\begin{proof}
Since the function \(\vartheta_\sigma\) does not depend on \((a,x)\), only \(\sigma\)-derivatives
have to be estimated. The case \(j=0\) is straightforward from \(|\vartheta_\sigma|\leqslant 1\).
For \(1\leqslant j\leqslant s\), repeated differentiation of \(\vartheta( \sigma^{1/4}\zeta)\)
produces a finite sum of terms of the form
\begin{equation*}
  \sigma^{-j}\, P_{j,\mu}(\sigma^{1/4}\zeta)\, \partial^\mu\vartheta(\sigma^{1/4}\zeta),
  \qquad 1\leqslant |\mu|\leqslant j,
\end{equation*}
where \(P_{j,\mu}\) is a polynomial. Since \(\partial^\mu\vartheta(\sigma^{1/4}\zeta)\) is supported
in \(1\leqslant \|\sigma^{1/4}\zeta\|\leqslant 2\), on the set where it is nonzero
\(P_{j,\mu}(\sigma^{1/4}\zeta)\) is uniformly bounded. Therefore
\begin{equation*}
  |\partial_\sigma^j\vartheta_\sigma(\zeta)|
  \leqslant C_j\sigma^{-j} \mathds 1(1\leqslant \| \sigma^{1/4}\zeta\|) \leqslant C_j\|\zeta\|^{4j}
  \leqslant C_j'(1+\|\zeta\|^{4s}),
\end{equation*}
which is the required admissible estimate.
\end{proof}

\begin{lemma}[Cutoff extension to the global domain]
  \label{lem:zero_extension_cutoff_local_kernels}
  Let \(s\geqslant0\) be an integer, \(R_0\in(0,4R)\), \(A>0\) and \(\sigma_0>0\). Suppose that
  \(F(\zeta;a,x,\sigma)\) defined on \(\mathcal D_{R_0,A,\Compct_{\Strat},\sigma_0} \) belongs to
  \(\mathfrak P_s\) or to \(\mathfrak G_s\), and is supported away from the truncation boundary, in
  the sense that, for some \(R_{\mathrm{supp}}\in(0,R_0)\),
  \begin{equation*}
    F(\zeta;a,x,\sigma)=0
    \qquad\text{whenever}\qquad R_{\mathrm{supp}}\leqslant \sigma\|\zeta\|<R_0 .
  \end{equation*}
  Let \(F_\infty\) be the zero extension of \(F\) to \(\mathcal
  D_{\infty,A,\Compct_{\Strat},\sigma_0}\). Then \(F_\infty\in\mathfrak P_s\) or
  \(F_\infty\in\mathfrak G_s\).
\end{lemma}

\begin{proof}
Since \(F\) already vanishes on a neighborhood of the boundary \(\{\sigma\|\zeta\|=R_0\}\), its zero
extension is \(C^s\) in \((a,x,\sigma)\). The admissible bounds are unchanged on the original domain
and are trivial on the complement, where the extension is identically zero.
\end{proof}

\subsection{Gaussian domination estimates}
\label{subsec:app_gaussian_domination_estimates}

We now prove the Gaussian estimates used to control the polynomial remainders after rescaling. Since
the exponent \(\Psi\) is uniformly comparable, for bounded \(a\), to the canonical conical Gaussian,
we get the uniform integrability needed in the scaled variable \(\zeta\). Recall that the linearized
exponent is
\begin{equation}\label{eq:Psi1}
  \Psi(\zeta;a,x) = \frac12 \Bigl[ \|\bfL(x)\zeta - [0_{m-c};\,a_{\mathcal C}]\|^2
  + \|a_{\mathcal N}\|^2 \Bigr].
\end{equation}

\begin{lemma}[Comparison with a canonical conical Gaussian]
  \label{lem:gaus_comp_can}
  Assume \AssChart. Fix \(A>0\). There are constants \(c_0,C>0\) such that
  \begin{equation*}
    e^{-\Psi(\zeta;a,x)} \leqslant C\Gamma_{c_0}(\zeta;a), \quad \text{for all}
    \quad x\in\Compct_{\Strat},\, \|a\|\leqslant A,\, \zeta\in\mathbb H_c^m.
  \end{equation*}
  Consequently, for every \(R\in(0,\infty]\), every \(\sigma_0>0\), every \(0\leqslant s\leqslant
  r\) and every \(P\in\mathfrak P_s\) on \(\mathcal D_{R,A,\Compct_{\Strat},\sigma_0} \), one has
  \(e^{-\Psi}P \in \mathfrak G_s\) on the same domain.
\end{lemma}

\begin{proof}
Using \(\|u+v\|^2\geqslant \frac12\|u\|^2-\|v\|^2\) in \eqref{eq:Psi1}, we get
\begin{align*}
  \|\bfL(x)\zeta - [0_{m-c};\,a_{\mathcal C}]\|^2 &\geqslant (1/2)\|\bfL(x)\zeta\|^2
  - \|a_{\mathcal C}\|^2 \\
  &\geqslant (1/4)\|\bfL(x)(\zeta - [0_{m-c};a_{\mathcal C}]) \|^2
  - (1/2)\|\bfL(x)[0_{m-c};a_{\mathcal C}]\|^2 - A^2 \\
  &\geqslant \|4\bfL(x)^{-1}\|^{-2}\|\zeta - [0_{m-c};\, a_{\mathcal C}] \|^2
  - (1/2)\|\bfL(x)\|^2 \| a_{\mathcal C} \|^2 - A^2 \\
  &\geqslant c_\bfL \big\{\|\zeta_{\mathcal S}\|^2 + \|\zeta_{\mathcal C}
  - a_{\mathcal C} \|^2\big\} - C_{\bfL,A},
\end{align*}
for some \(c_\bfL\in(0,1)\), \(C_{\bfL,A}\in (0,\infty)\), where the last line follows from the
uniform boundedness of the norms of \(\bfL (x)\) and \(\bfL(x)^{-1}\) established in
\Cref{lem:ad_frame_comp_strat}. Thus, we obtain
\begin{equation*}
  \Psi(\zeta;a,x) \geqslant c_\bfL \bigl(\|\zeta_{\mathcal S}\|^2
  + \|a_{\mathcal C}-\zeta_{\mathcal C}\|^2 + \|a_{\mathcal N}\|^2 \bigr) - C_{\bfL,A}.
\end{equation*}
Exponentiating, we get
\begin{equation*}
  e^{-\Psi(\zeta;a,x)} \leqslant C \exp\Bigl( -c_\bfL \bigl[\|\zeta_{\mathcal S}\|^2
  + \|a_{\mathcal C} - \zeta_{\mathcal C}\|^2 + \|a_{\mathcal N}\|^2 \bigr] \Bigr),
\end{equation*}
which is exactly the desired comparison with \(\Gamma_{c_0}\) for \(c_0 = c_\bfL\).

We now prove the admissible-class statement. After the pullback \(x=\varphi(\theta)\), the function
\(\Psi(\zeta;a,\varphi(\theta))\) is quadratic in \(\zeta\), with coefficients of class \(C^s\) in
\((a,\theta)\), uniformly for \(\|a\|\leqslant A\) and \(\theta\in\Theta_{\Compct_{\Strat}}\). Since
\(e^{-\Psi}\) is independent of \(\sigma\), every \((a,\theta,\sigma)\)-derivative of \(e^{-\Psi}\)
of order at most \(s\) is a finite sum of terms \(
Q(\zeta;a,\theta)e^{-\Psi(\zeta;a,\varphi(\theta))}, \) where \(Q\) is polynomial in \(\zeta\) of
degree at most \(2s\), with uniformly bounded coefficients. The already proved comparison gives
\begin{equation*}
  \bigl| \partial_a^\alpha\partial_\theta^\beta\partial_\sigma^j e^{-\Psi(\zeta;a,\varphi(\theta))}
  \bigr| \leqslant C_{\alpha,\beta,j} \bigl(1+\|\zeta\|^{2s}\bigr) \Gamma_{c_{\mathbf L}}(\zeta;a)
\end{equation*}
for all \(|\alpha|+|\beta|+j\leqslant s\). Thus \(e^{-\Psi}\in\mathfrak G_s\). If \(P\in\mathfrak
P_s\) on the same domain, the product rule and \Cref{lem:prop_adm_classes} give
\(e^{-\Psi}P\in\mathfrak G_s\).
\end{proof}

\subsection{Inner-kernel expansion}
\label{app:proof_inner_kernel_expansion}

We now expand the scaled kernel
\begin{equation*}
  \mathscr K_\sigma(\zeta;a,x) = e^{-\Psi_\sigma(\zeta;a,x)} \mathcal A_x(\sigma\zeta)
\end{equation*}
on the growing inner region selected by \(\vartheta_\sigma\). On this region, the exponent
perturbation \(\Psi_\sigma-\Psi\) is uniformly small, so the exponential can be expanded to first
order. Multiplication by the amplitude expansion then gives a Gaussian-class remainder.

\begin{lemma}[Inner expansion with Gaussian-class remainder]
  \label{lem:inner_exp_gaussian_rem}
  Assume \ref{ass:standing:measure}, \ref{ass:standing:chart}, and~\ref{ass:standing:density} of
  \Cref{ass:standing} with \(r\geqslant 2\). Fix \(A>0\). For \(x=\varphi(\theta)\), set
  \begin{equation}\label{eq:kappa1}
    \kappa_{1,x}(\zeta;a)
    =
    \nabla_\xi \mathcal A_x(0)^\top \zeta
    - \frac12\mathcal A_x(0)\Lambda_x(\zeta;a).
  \end{equation}
  Then there exist \(\sigma_0>0\) and \(\mathscr R_{\mathrm{ker}}\in\mathfrak G_{r-2}\), defined on
  \(\mathcal D_{\infty,A,\Compct_{\Strat},\sigma_0}\), such that
  \begin{equation}\label{eq:exp_inner}
    \vartheta_\sigma(\zeta)\mathscr K_\sigma(\zeta;a,x)
    =
    \vartheta_\sigma(\zeta)e^{-\Psi(\zeta;a,x)}
    \Bigl[
      \mathcal A_x(0)
      +\sigma\kappa_{1,x}(\zeta;a)
    \Bigr]
    +\sigma^2\mathscr R_{\mathrm{ker}}(\zeta;a,x,\sigma).
  \end{equation}
\end{lemma}

\begin{proof}
By \Cref{lem:adm_scaled_exp_delta,lem:adm_scaled_exp_amplitude,lem:admi_expo_exp}, on \(\mathcal
D_{R,A,\Compct_{\Strat},\sigma_0}\) we have
\[
  \Psi_\sigma
  =
  \Psi+\frac{\sigma}{2}\Lambda_x+\sigma^2\mathscr R_\Psi,
  \qquad
  \mathcal A_x(\sigma\zeta)
  =
  \mathcal A_x(0)
  +\sigma\nabla_\xi\mathcal A_x(0)^\top\zeta
  +\sigma^2\mathscr R_{\mathcal A},
\]
with \(\mathscr R_\Psi,\mathscr R_{\mathcal A}\in\mathfrak P_{r-2}\). Set
\[
  \delta_\sigma
  =
  \Psi_\sigma-\Psi
  =
  \frac{\sigma}{2}\Lambda_x+\sigma^2\mathscr R_\Psi .
\]
After decreasing \(\sigma_0\), the support property of \(\vartheta_\sigma\) and the pointwise
estimate in \Cref{lem:admi_expo_exp} imply
\[
  \operatorname{supp}\vartheta_\sigma
  \subset \{\sigma\|\zeta\|\le R/2\},
  \qquad
  |\delta_\sigma|\le \frac12
  \quad\text{on }\operatorname{supp}\vartheta_\sigma .
\]
Taylor's formula with integral remainder gives, on \(\operatorname{supp}\vartheta_\sigma\),
\[
  e^{-\delta_\sigma}
  =
  1-\delta_\sigma
  +\delta_\sigma^2
  \int_0^1(1-t)e^{-t\delta_\sigma}\,\dd t .
\]
Since
\[
  \delta_\sigma
  =
  \frac{\sigma}{2}\Lambda_x+\sigma^2\mathscr R_\Psi,
\]
we obtain
\[
  \vartheta_\sigma e^{-\delta_\sigma}
  =
  \vartheta_\sigma
  \left(
    1-\frac{\sigma}{2}\Lambda_x
  \right)
  +\sigma^2\mathscr R_{\exp},
\]
where, collecting the \(\sigma^2\) terms and using \(\delta_\sigma^2=\sigma^2\bigl(\tfrac12\Lambda_x
+\sigma\mathscr R_\Psi\bigr)^2\),
\begin{equation*}
  \mathscr R_{\exp}
  =
  \vartheta_\sigma
  \left[
    -\,\mathscr R_\Psi
    +\Bigl(\tfrac12\Lambda_x+\sigma\mathscr R_\Psi\Bigr)^2
    \int_0^1(1-t)e^{-t\delta_\sigma}\,\dd t
  \right].
\end{equation*}
We claim \(\mathscr R_{\exp}\in\mathfrak P_{r-2}\) on \(\mathcal D_{R,A,\Compct_{\Strat},\sigma_0}\),
with support contained in \(\operatorname{supp}\vartheta_\sigma\); its zero extension then
belongs to \(\mathfrak P_{r-2}\) on the full scaled domain.

To see the claim, we bound one admissible derivative; the rest are identical. Fix a multi-index
\(\gamma\) in \((a,\theta,\sigma)\) with \(|\gamma|\leqslant r-2\). By
\Cref{lem:adm_scaled_exp_delta,lem:admi_expo_exp}, \(\mathscr R_\Psi\in\mathfrak P_{r-2}\) and
\(\Lambda_x\in\mathfrak P_{r-2}\), so the factor \(-\mathscr R_\Psi\) and the polynomial factor
\(\bigl(\tfrac12\Lambda_x+\sigma\mathscr R_\Psi\bigr)^2\) have derivatives bounded by
\(C_\gamma(1+\|\zeta\|^{n})\) for some \(n\), by the product and scalar-multiplication rules of
\Cref{lem:prop_adm_classes}. For the integral factor, write \(\Theta(\delta_\sigma)=\int_0^1(1-t)
e^{-t\delta_\sigma}\,\dd t\). Since \(|\delta_\sigma|\leqslant\tfrac12\) on
\(\operatorname{supp}\vartheta_\sigma\), Leibniz and Fa\`a di Bruno give
\(\partial^{\gamma}\Theta(\delta_\sigma)\) as a finite sum of terms
\(\bigl(\int_0^1(1-t)(-t)^{j}e^{-t\delta_\sigma}\,\dd t\bigr)\prod_{i}\partial^{\gamma_i}
\delta_\sigma\) with \(\sum_i\gamma_i=\gamma\) and \(j\leqslant|\gamma|\); each \(t\)-integral is
bounded by \(\int_0^1(1-t)e^{t/2}\,\dd t<\infty\), and each \(\partial^{\gamma_i}\delta_\sigma\),
with \(\delta_\sigma=\tfrac{\sigma}{2}\Lambda_x+\sigma^2\mathscr R_\Psi\), grows at most
polynomially in \(\zeta\) because \(\Lambda_x,\mathscr R_\Psi\in\mathfrak P_{r-2}\). Finally, the
derivatives that fall on \(\vartheta_\sigma\) contribute the admissible factors of the growing
cutoff, which are \(O(1)\) times powers \(\|\zeta\|^{4j}\) supported on
\(\operatorname{supp}\vartheta_\sigma\) (see \Cref{lem:admissibility_of_growing_cutoff}); these are
absorbed into the same polynomial bound. Collecting the factors, \(|\partial^\gamma\mathscr
R_{\exp}|\leqslant C_\gamma(1+\|\zeta\|^{n'})\) uniformly on \(\mathcal
D_{R,A,\Compct_{\Strat},\sigma_0}\), which is the defining estimate of \(\mathfrak P_{r-2}\).

Multiplying by the amplitude expansion gives
\begin{equation*}
  \vartheta_\sigma e^{-\Psi_\sigma}
  \mathcal A_x(\sigma\zeta)
  =
  \vartheta_\sigma e^{-\Psi}
  \Bigl[
    \mathcal A_x(0)
    +\sigma\Bigl(
      \nabla_\xi\mathcal A_x(0)^\top\zeta
      -\frac12\mathcal A_x(0)\Lambda_x(\zeta;a)
    \Bigr)
  \Bigr]
  +\sigma^2 e^{-\Psi}\mathscr P_{\mathrm{in}},
\end{equation*}
where, on \(\mathcal D_{R,A,\Compct_{\Strat},\sigma_0}\), \(\mathscr P_{\mathrm{in}}\in\mathfrak
P_{r-2}\), and its support is contained in \(\operatorname{supp}\vartheta_\sigma\). Since
\(\operatorname{supp}\vartheta_\sigma\subset\{\sigma\|\zeta\|\le R/2\}\), this support is separated
from the truncation boundary \(\{\sigma\|\zeta\|=R\}\). Hence, by
\Cref{lem:zero_extension_cutoff_local_kernels}, the zero extension of \(\mathscr P_{\mathrm{in}}\)
belongs to \(\mathfrak P_{r-2}\) on \(\mathcal D_{\infty,A,\Compct_{\Strat},\sigma_0}\).

Finally, \Cref{lem:gaus_comp_can} gives \(e^{-\Psi}\mathscr P_{\mathrm{in}}\in\mathfrak G_{r-2}\).
Thus the desired expansion holds with \(\mathscr R_{\mathrm{ker}} = e^{-\Psi}\mathscr
P_{\mathrm{in}}\).
\end{proof}

\section{Integrated local expansion and coefficient estimates}
\label{app:integration_coefficients_tails}

This section completes the passage from the kernel-level expansion to the integrated local
expansion. We identify the first two coefficients, and show that all errors left outside the growing
inner region are exponentially small.

The argument has three parts. We first prove the regularity and positivity properties of the
coefficient functions \(\mathsf C_0\), \(\mathsf C_1\), and of the logarithmic coefficients
\(\mathsf L_0\), \(\mathsf L_1\). We then integrate the inner Gaussian-class remainder. Finally, we
compare the truncated inner integrals with the full-cone coefficients and absorb the resulting model
and exact tails into a \(\sigma^2\)-remainder.

\paragraph{Hypotheses in force.}
This section uses \Cref{ass:standing}~\ref{ass:standing:measure}, \ref{ass:standing:chart} and
\ref{ass:standing:density}. The only result in this section that additionally uses the positivity
item \ref{ass:standing:positivity} is \Cref{lem:uniform_positivity_of_C0}.

\subsection{Regularity of the coefficient integrals}
\label{app:proof_of_lem_coefficient_kernels_in_Br}

The coefficients \(\mathsf C_0\) and \(\mathsf C_1\) are obtained by integrating Gaussian-class
kernels over the tangent cone. Their regularity therefore follows from the Gaussian comparison
estimates and from the integration rule for admissible classes.

\begin{lemma}[Regularity of coefficient kernels]
  \label{lem:coeff_kernels_in_Br}
  Assume \ref{ass:standing:measure}, \ref{ass:standing:chart}, and~\ref{ass:standing:density}
  of \Cref{ass:standing}. Let \(A, \sigma_0>0\). If \(r\geqslant0\), then \(\mathsf C_0\in\mathfrak
  B_{r}\) on \(\mathcal
  D_{A,\Compct_{\Strat},\sigma_0}\). If \(r\geqslant1\), then \(\mathsf C_1\in\mathfrak B_{r-1}\) on
  the same domain.
\end{lemma}

\begin{proof}
By \Cref{lem:gaus_comp_can}, applied with \(P\equiv1\), we have \( e^{-\Psi}\in \mathfrak G_{r}\) on
\(\mathcal D_{\infty,A, \Compct_{\Strat}, \sigma_0}\). Since \(x\mapsto \mathcal A_x(0)\) is in
\(\mathfrak B_{r}\) on \(\mathcal D_{A,\Compct_{\Strat}, \sigma_0}\), the product rule in
\Cref{lem:prop_adm_classes} gives
\begin{equation*}
  \mathcal A_x(0) e^{-\Psi(\zeta;a,x)} \in \mathfrak G_{r}\quad\text{on}
  \quad \mathcal D_{\infty,A,\Compct_{\Strat},\sigma_0}.
\end{equation*}
Integrating in \(\zeta\) and using item~\ref{item:integr_Gauss_kern} of \Cref{lem:prop_adm_classes},
gives \(\mathsf C_0\in\mathfrak B_{r}\). On the other hand, in view of \cref{eq:kappa1}, using the
regularity of \(\nabla_\xi\mathcal A_x(0)\) and of \(\Lambda_x\), we get \(\kappa_{1,x}\in\mathfrak
P_{r-1}\). By \Cref{lem:gaus_comp_can}, this yields \(e^{-\Psi} \kappa_{1,x} \in \mathfrak G_{r-1}\)
on the domain \(\mathcal D_{\infty,A,\Compct_{\Strat},\sigma_0}\). Another application of
item~\ref{item:integr_Gauss_kern} of \Cref{lem:prop_adm_classes} yields \(\mathsf C_1\in\mathfrak
B_{r-1}\).
\end{proof}

We next record the lower bound for the leading coefficient. This is the only point at which
positivity of the density on the compact stratum set is used.

\begin{lemma}[Uniform positivity of \(\mathsf C_0\)]
  \label{lem:uniform_positivity_of_C0}
  Assume \ref{ass:standing:measure}, \ref{ass:standing:chart}, \ref{ass:standing:density}
  and~\ref{ass:standing:positivity} of \Cref{ass:standing}. For every \(A>0\),
  \begin{equation}\label{eq:lower_bound_C0}
    \inf_{x\in\Compct_{\Strat}, \|a\|\leqslant A}\mathsf C_0(a,x)>0.
  \end{equation}
  Consequently, \(\mathsf C_0^{-1}\in \mathfrak B_{r}\) on the domain \(\mathcal
  D_{A,\Compct_{\Strat}, \sigma_0}\). Furthermore, \(\mathsf L_0= \log \mathsf C_0\in \mathfrak
  B_{r}\) and \(\mathsf L_1 = \mathsf C_1/ \mathsf C_0 \in \mathfrak B_{r-1}\), provided that \(
  r\geqslant 1\).
\end{lemma}

\begin{proof}
By definition,
\begin{equation*}
  \mathsf C_0(a,x) = \mathcal A_x(0) \int_{\mathbb H_c^m} e^{-\Psi(\zeta;a,x)}\,\dd\zeta.
\end{equation*}
By item~\ref{ass:standing:positivity} of \Cref{ass:standing}, \(\rho(x)\geqslant\rho_*\) for all
\(x\in\Compct_{\Strat}\). Moreover, \(\bfL(x)\in\mathsf{GL}(m)\), so \(|\det\bfL(x)|>0\). Hence
\[
  \mathcal A_x(0)
  =
  \rho(x)|\det\bfL(x)|>0.
\]
The integrand \(e^{-\Psi(\zeta;a,x)}\) is strictly positive on \(\mathbb H_c^m\). Hence \(\mathsf
C_0(a,x)>0\) for all \(x\in\Compct_{\Strat}\) and \(\|a\|\leqslant A\). By
\Cref{lem:coeff_kernels_in_Br}, \(\mathsf C_0\in \mathfrak B_{r}\) on \(\mathcal D_{A,
\Compct_{\Strat},\sigma_0}\). In particular, since \(r \geqslant 0\), \(\mathsf C_0\) is continuous
in \((a,x)\). Since the domain \(\Compct_{\Strat}\times\overline{\mathbb B}_A^{c+k}\) is compact,
the positive continuous function \(\mathsf C_0\) attains a positive minimum on this domain. This
proves \cref{eq:lower_bound_C0}. The inversion statement follows from the inversion property stated
in \Cref{lem:prop_adm_classes}. Since \(\mathsf C_1\in\mathfrak B_{r-1}\) on \(\mathcal D_{A,
\Compct_{\Strat}, \sigma_0}\) by \Cref{lem:coeff_kernels_in_Br}, the product rule yields
\begin{equation*}
  \mathsf L_1=\mathsf C_1\mathsf C_0^{-1}
  \in \mathfrak B_{r-1}\quad\text{on}\quad \mathcal D_{A,\Compct_{\Strat},\sigma_0}.
\end{equation*}
Finally, since we already proved that the range of \(\mathsf C_0\) is contained in a compact
interval \([m_A,M_A]\Subset(0,\infty)\) and \(\log\in C^\infty (0,\infty)\), the smooth-composition
property in \Cref{lem:prop_adm_classes} gives \(\mathsf L_0 = \log \mathsf C_0 \in \mathfrak B_{r}
\).
\end{proof}

\begin{remark}[Invariance of the coefficients]
  \label{rem:invariance}
  Although \(\mathsf C_0\) and \(\mathsf C_1\) are written through the chart datum \(\bfL(x)\) and
  the adapted frames \(\bfC(x),\bfN(x)\), the functions they define are intrinsic: independent of
  the adapted corner chart and equivariant under rotations of the adapted frames. This is the
  property invoked by the globalization argument of \Cref{rem:globalization}.

  For \(\mathsf C_0\) this is seen directly from \eqref{eq:C0}. An adapted transition \(\tau\) with
  \(\Diff\tau(0)\,\mathbb H_c^m=\mathbb H_c^m\) replaces \(\bfL\) by \(\bfL\,\Diff\tau(0)\), and the
  substitution \(\zeta\mapsto\Diff\tau(0)\zeta\) in the integral cancels the change in
  \(|\det\bfL|\) through the Jacobian \(|\det\Diff\tau(0)|\); a rotation of the adapted frames,
  \(a\mapsto (O_{\mathcal C}a_{\mathcal C},O_{\mathcal N}a_{\mathcal N})\) with \(O_{\mathcal
  C}\in\mathsf O(c)\), \(O_{\mathcal N}\in\mathsf O(k)\), leaves the integral unchanged. Hence
  \(\mathsf C_0\bigl(a(y,\sigma),\pi(y)\bigr)\) --- and every logarithmic quantity derived from it
  --- is a well-defined function of \((y,\sigma)\).

  For \(\mathsf C_1\) the same substitution does not settle the matter, since its expression
  \eqref{eq:def_C1} mixes \(\nabla_\xi\mathcal A_x(0)\) and \(\Lambda_x\), each of which is
  individually chart-dependent. Instead we argue through the expansion itself. The quantity
  \((2\pi)^{d/2}\sigma^{k}p_\sigma\bigl(y_\sigma(a,x)\bigr)\) is intrinsic, being built from the
  density \(p_\sigma\) and the tubular point \(y_\sigma(a,x)\) of \eqref{eq:ysigma}, and
  \Cref{thm:2nd_coef_kernel} provides, in every admissible corner chart, the two-term expansion
  \[
    (2\pi)^{d/2}\sigma^{k}p_\sigma\bigl(y_\sigma(a,x)\bigr)
    = \mathsf C_0(a,x) + \sigma\,\mathsf C_1(a,x) + O(\sigma^2), \qquad \sigma\downarrow 0.
  \]
  By uniqueness of asymptotic expansions the coefficients are determined by the left-hand side
  alone; since \(\mathsf C_0>0\) (established above) is already chart-free, so are
  \[
    \mathsf C_0(a,x) = \lim_{\sigma\downarrow 0}(2\pi)^{d/2}\sigma^{k}p_\sigma\bigl(y_\sigma(a,x)\bigr),
    \qquad
    \mathsf C_1(a,x) = \lim_{\sigma\downarrow 0}\sigma^{-1}
      \Bigl[(2\pi)^{d/2}\sigma^{k}p_\sigma\bigl(y_\sigma(a,x)\bigr) - \mathsf C_0(a,x)\Bigr],
  \]
  whose right-hand sides involve only intrinsic objects.
\end{remark}

\subsection{Gaussian tails and domination of the exact cutoff
kernel}
\label{app:gauss_tails_kernel_domin}

We now control the part of the integral outside the growing window selected by \(\vartheta_\sigma\).
We first prove an abstract tail estimate for Gaussian-class kernels. We then verify that the exact
cutoff kernel \(\bar\chi_R(\sigma\zeta) \mathscr K_\sigma\), where \(\bar\chi_R(\xi) =
\chi(\|\xi\|^2 /R^2)\), belongs to such a Gaussian class after zero extension. This requires one
additional geometric estimate: points represented by the corner chart separate uniformly from the
base point at least linearly in the chart variable.

\begin{lemma}[Gaussian tail estimate]
  \label{lem:gaus_tail_beyond_window}
  Let \(F\in\mathfrak G_s\) on \(\mathcal D_{\infty,A, \Compct_{\Strat},\sigma_0}\) for an integer
  \(s\geqslant0\), and define
  \begin{equation*}
    \mathbb T[F](a,x,\sigma)
    = \int_{\mathbb H_c^m} (1-\vartheta_\sigma(\zeta)) F(\zeta;a,x,\sigma) \,\dd\zeta.
  \end{equation*}
  Then, for every \(M\in\mathbb N\), \((a,x,\sigma)\mapsto \sigma^{-M}\mathbb T[F]\) is in
  \(\mathfrak E_s^{1/2}\) on \(\mathcal D_{A,\Compct_{\Strat},\sigma_0}\).
\end{lemma}

\begin{proof}
Let \(n\in\mathbb N\) be such that \(F\in \mathfrak G_s\). Choose \(c_0>0\) such that, for every
\(|\alpha| + |\beta| + j\leqslant s\),
\begin{equation*}
  \bigl| \partial_a^\alpha\partial_\theta^\beta
  \partial_\sigma^j F(\zeta;a,\varphi(\theta),\sigma) \bigr|
  \leqslant C_{\alpha,\beta,j} (1+\|\zeta\|^n) \Gamma_{c_0}(\zeta;a).
\end{equation*}
By \Cref{lem:admissibility_of_growing_cutoff}, derivatives of \(1-\vartheta_\sigma\) satisfy
polynomial bounds. More precisely, every term produced by Leibniz' rule in an admissible derivative
of \((1-\vartheta_\sigma(\zeta))F(\zeta;a,\varphi(\theta),\sigma)\) is bounded by
\begin{equation*}
  C_{\alpha,\beta,j} (1+\|\zeta\|^{n+4s}) \Gamma_{c_0}
  (\zeta;a) \mathbf 1_{\{\sigma\|\zeta\|^4\geqslant 1\}}(\zeta).
\end{equation*}
Here we used that \(1-\vartheta_\sigma=0\) when \(\sigma\|\zeta\|^4\leqslant 1\), and that all
\(\sigma\)-derivatives of \(\vartheta_\sigma\) are supported where \(\sigma\|\zeta\|^4\geqslant 1\).

Since \(\|a\|\leqslant A\), the canonical conical Gaussian is dominated by an ordinary centered
Gaussian: there exist constants \(c_A,C_A>0\) such that \(\Gamma_{c_0} (\zeta;a) \leqslant C_A
e^{-c_A\|\zeta\|^2}\) uniformly for \(\|a\| \leqslant A\). Therefore differentiation under the
integral sign is justified, and each admissible derivative of \(\mathbb T[F]\) is bounded by
\begin{equation*}
  C \int_{\sigma\|\zeta\|^4\geqslant 1} (1+\|\zeta\|^{n+4s}) e^{-c_A\|\zeta\|^2} \,\dd\zeta
  \leqslant C' e^{-c'\sigma^{-1/2}}.
\end{equation*}
This yields
\begin{equation*}
  \bigl| \partial_a^\alpha\partial_\theta^\beta\partial_\sigma^j
  \,\mathbb T[F](a,\varphi(\theta),\sigma) \bigr| \leqslant C_{\alpha,\beta,j}e^{-c\sigma^{-1/2}},
\end{equation*}
for all \(|\alpha|+|\beta|+j\leqslant s\). Hence \(\mathbb T[F]\in\mathfrak E_s^{1/2}\). The final
assertion follows from item~\ref{item:mult_div_power_sigma} of \Cref{lem:prop_adm_classes}.
\end{proof}

Recall that we use the notation \(\Phi_x(\xi)\) for \(M(\theta,\xi) = \Phi(\theta + \xi_{\mathcal
S}, \xi_{\mathcal C}) \) for \(\theta\in \Theta_{\rm out}\) such that \(x = \Phi(\theta,0)\).

\begin{lemma}[Uniform chart separation]
  \label{lem:chart_sep_exact_tail}
  Assume \AssChart. There exists \(c_{\mathrm{ch}}>0\) such that, for every \(x\in\Compct_{\Strat}\)
  and every
  \(\xi\in\mathbb H_c^m\) with \(\|\xi\|\leqslant 2R\),
  \begin{equation}\label{eq:chart_sep_tail}
    \|\Phi_x(\xi)-x\| \geqslant c_{\mathrm{ch}}\|\xi\|.
  \end{equation}
\end{lemma}

\begin{proof}
The definition of \(\Phi_x\) implies that the inequality we aim to prove is equivalent to
\begin{equation*}
  \|M(\theta, \xi) - M(\theta,0)\| \geqslant c_{\rm ch}\|\xi\|, \qquad \forall \theta
  \in \Theta_{\Compct_{\Strat}}.
\end{equation*}
Since \(\Diff_\xi M(\theta,0)=\bfA_{\mathrm{ch}} (\varphi(\theta))\) has rank \(m\), compactness of
\(\Theta_{\Compct_{\Strat}}\) gives
\begin{equation*}
  \lambda_* = \inf_{\theta\in\Theta_{\Compct_{\Strat}},\,\|v\|=1} \|\Diff_\xi M(\theta,0)v\| >0.
\end{equation*}
By uniform continuity on compact sets of \(\Diff_\xi M\), there is a radius \(r_*>0\) such that \(\|
\Diff_\xi M(\theta,\xi)-\Diff_\xi M(\theta,0)\| \leqslant {\lambda_*}/{2}\) whenever \(\theta \in
\Theta_{\Compct_{\Strat}}\) and \(\|\xi\|\leqslant r_*\). Hence the fundamental theorem of calculus
gives
\begin{equation*}
  \|M(\theta,\xi) - M(\theta,0)\| \geqslant ({\lambda_*}/{2})\|\xi\|, \qquad \|\xi\|\leqslant r_*.
\end{equation*}
On the compact set
\[
  \bigl\{(\theta,\xi):
  \theta\in\Theta_{\Compct_{\Strat}},\
  \xi\in\mathbb H_c^m,\
  r_*\leqslant\|\xi\|\leqslant2R
  \bigr\},
\]
the function
\[
  (\theta,\xi)\mapsto
  \frac{\|M(\theta,\xi)-M(\theta,0)\|}{\|\xi\|}
\]
is continuous and strictly positive, because each mapping \(M(\theta,\cdot)\) is injective. Its
minimum on this compact set is therefore positive. Combining the small-\(\xi\) and annular
estimates gives \cref{eq:chart_sep_tail}.
\end{proof}

\begin{lemma}[Exact cutoff kernel belongs to a Gaussian class]
  \label{lem:scaled_kern_gaus_class}
  Assume \ref{ass:standing:measure}, \ref{ass:standing:chart}, and~\ref{ass:standing:density}
  of \Cref{ass:standing}. Let \(A>0\), and let \(0\leqslant s\leqslant r\). There exists
  \(\sigma_0>0\) such that the zero
  extension of \(\bar\chi_R(\sigma\zeta)\mathscr K_\sigma(\zeta;a,x)\) to
  \(\mathcal D_{\infty,A,\Compct_{\Strat},\sigma_0}\) belongs to \(\mathfrak G_s\).
\end{lemma}

\begin{proof}
\Cref{lem:rec_map_prop} implies that there exists \(\sigma_0 > 0\) such that \(y_\sigma\) given by
\eqref{eq:ysigma} is in \(\mathfrak B_{r}\) on \(\mathcal D_{A,\Compct_{\Strat},\sigma_0}\), and, in
view of \cref{eq:scaled_exp}, the identity
\begin{equation*}
  \Psi_\sigma(\zeta;a,x) = \frac{\|y_\sigma(a,x) - \Phi_x(\sigma\zeta)\|^2}{2\sigma^2}
\end{equation*}
holds for \(\sigma\|\zeta\|<4R\). We shall prove the required estimates first on the local scaled
domain \( \mathcal D_{3R,A,\Compct_{\Strat}, \sigma_0}\) given by \eqref{eq:mathcal-D}. On this
domain the chart quantities \(\Delta(\theta,\sigma\zeta),\Psi_\sigma(\zeta;a,x)\), and \(\mathcal
A_x(\sigma\zeta)\) are well defined, because \(3R<4R\). We first prove Gaussian domination on the
support of the cutoff.
If \(\bar\chi_R(\sigma\zeta)\neq0\), then \(\sigma\|\zeta\| \leqslant 2R\). Hence
\Cref{lem:chart_sep_exact_tail} gives \(\|\Phi_x(\sigma\zeta)-x\| \geqslant c_{\mathrm{ch}} \sigma
\|\zeta\|\). On the other hand, by the definition of \(y_\sigma\) and the orthonormality of the
columns of \(\bfC(x)\) and \(\bfN(x)\), \(\|y_\sigma(a,x)-x\| = \sigma\|a\|\leqslant A\sigma\).
Using the elementary inequality \(\|u-v\|^2\geqslant \frac12\|u\|^2-\|v\|^2,\) with
\(u=\Phi_x(\sigma\zeta)-x, v=y_\sigma(a,x)-x,\) we obtain
\begin{equation*}
  \|y_\sigma(a,x)-\Phi_x(\sigma\zeta)\|^2 \geqslant \frac12 c_{\mathrm{ch}}^2\sigma^2\|\zeta\|^2
  - A^2\sigma^2.
\end{equation*}
Therefore, by \cref{eq:scaled_exp}, on the support of \(\bar\chi_R(\sigma\zeta)\),
\begin{equation*}
  \Psi_\sigma(\zeta;a,x) \geqslant \frac{c_{\mathrm{ch}}^2}{4}\|\zeta\|^2 - \frac{A^2}{2}
  \quad\Longrightarrow\quad e^{-\Psi_\sigma(\zeta;a,x)} \leqslant C e^{-c\|\zeta\|^2}
\end{equation*}
with \(c = c_{\rm ch}^2/4>0\) and \(C = e^{A^2/2} >0\). We next compare this centered Gaussian with
the canonical conical Gaussian. Since \(\|a\|\leqslant A\), there exist constants \(c_A,C_A>0\) such
that
\begin{equation*}
  e^{-c\|\zeta\|^2} \leqslant C_A \exp\left(
  - c_A\bigl[\|\zeta_{\mathcal S}\|^2+\|a_{\mathcal C}-\zeta_{\mathcal C}\|^2
  + \|a_{\mathcal N}\|^2 \bigr] \right) = C_A\Gamma_{c_A}(\zeta;a),
\end{equation*}
where \(\zeta=(\zeta_{\mathcal S},\zeta_{\mathcal C})\). Hence, after changing the constants,
\begin{equation}\label{eq:exp_domin_can_in_proof}
  e^{-\Psi_\sigma(\zeta;a,x)} \leqslant C_A\Gamma_{c_A}(\zeta;a)
\end{equation}
uniformly over \(\|a\|\leqslant A\), \(x\in \Compct_{\Strat}\), \(0<\sigma\leqslant\sigma_0\), and
\(\bar\chi_R(\sigma\zeta)\neq0\).

By \Cref{lem:exact_scaled_exponent_poly_bounds}, applied with \(R_0=3R\), the exact exponent
\(\Psi_\sigma\) belongs to \(\mathfrak P_s\) on
\(\mathcal D_{3R,A,\Compct_{\Strat},\sigma_0}\). Therefore, by repeated use of the chain rule, every
admissible derivative of \(e^{-\Psi_\sigma}\) of order at most \(s\) is a finite sum of terms of the
form
\[
  P_{\gamma}(\zeta;a,x,\sigma)e^{-\Psi_\sigma(\zeta;a,x)},
  \qquad
  P_{\gamma}\in\mathfrak P_s.
\]
Consequently, there exists \(N_1\in\mathbb N\) such that
\begin{equation*}
  \left|
  \partial_a^\alpha\partial_\theta^\beta\partial_\sigma^j
  e^{-\Psi_\sigma(\zeta;a,\varphi(\theta))}
  \right|
  \leqslant
  C_{\alpha,\beta,j}(1+\|\zeta\|^{N_1})
  e^{-\Psi_\sigma(\zeta;a,\varphi(\theta))}
\end{equation*}
for all \(|\alpha|+|\beta|+j\leqslant s\) on
\(\mathcal D_{3R,A,\Compct_{\Strat},\sigma_0}\). In
every term of Leibniz' rule for
\begin{equation}\label{eq:partial_derivs}
  \partial_a^\alpha\partial_\theta^\beta\partial_\sigma^j \left[
  \bar\chi_R(\sigma\zeta)e^{-\Psi_\sigma(\zeta;a,\varphi(\theta))} \right],
\end{equation}
the cutoff factor or one of its \(\sigma\)-derivatives is present. Since \(\chi\) is identically
zero for \(\sigma\|\zeta\|\ge2R\), all the terms of form \eqref{eq:partial_derivs} vanish unless
\(\sigma \|\zeta\|\le2R.\) Thus the domination \eqref{eq:exp_domin_can_in_proof} applies to every
nonzero term. Moreover,
\begin{equation*}
  \left|\partial_\sigma^j\bar\chi_R(\sigma\zeta)\right| \leqslant C_j(1+\|\zeta\|^j).
\end{equation*}
Therefore, with a suitably adjusted \(N_1\),
\begin{equation*}
  \left| \partial_a^\alpha\partial_\theta^\beta\partial_\sigma^j \left[
  \bar\chi_R(\sigma\zeta)e^{-\Psi_\sigma(\zeta;a,\varphi(\theta))} \right] \right|
  \leqslant C_{\alpha,\beta,j} (1+\|\zeta\|^{N_1}) \Gamma_{c_A}(\zeta;a)
\end{equation*}
on \(\mathcal D_{3R,A,\Compct_{\Strat},\sigma_0}\). It remains to include the amplitude. Since
\(\mathcal A(\theta,\xi)\) is \(C^{r}\) on the compact chart region, the ordinary chain rule applied
to \((\theta,\zeta,\sigma)\mapsto\mathcal A(\theta,\sigma\zeta)\) gives \(\mathcal
A_x(\sigma\zeta)\in\mathfrak P_s\) on \(\mathcal D_{3R,A,\Compct_{\Strat},\sigma_0}\). Using that \(
\mathscr K_\sigma(\zeta;a,x) = e^{-\Psi_\sigma(\zeta;a,x)} \mathcal A_x(\sigma\zeta)\), and applying
Leibniz' rule once more, we infer the existence of constants \(N\in \mathbb N\), \(c_A>0\), and
\(C_{\alpha,\beta,j}>0\) such that
\begin{equation*}
  \left| \partial_a^\alpha\partial_\theta^\beta
  \partial_\sigma^j \left[ \bar\chi_R(\sigma\zeta)\mathscr
  K_\sigma(\zeta;a,\varphi(\theta)) \right] \right|
  \leqslant C_{\alpha,\beta,j} (1+\|\zeta\|^N) \Gamma_{c_A}(\zeta;a)
\end{equation*}
for all \(|\alpha|+|\beta|+j\leqslant s\), uniformly on \(\mathcal
D_{3R,A,\Compct_{\Strat},\sigma_0}\).

Finally, since \(\bar\chi_R(\sigma\zeta)=0 \text{ whenever } \sigma\|\zeta\|\geqslant 2R,\) the
product \(\bar\chi_R(\sigma\zeta)\mathscr K_\sigma(\zeta;a,x)\) and all of its admissible
derivatives vanish on the annular set \(2R\leqslant \sigma\|\zeta\|<3R\). Thus
\Cref{lem:zero_extension_cutoff_local_kernels}, applied with \(R_{\mathrm{supp}}=2R, R_0=3R,\) shows
that the zero extension of \(\bar\chi_R(\sigma\zeta) \mathscr K_\sigma(\zeta;a,x)\) to the global
scaled domain \(\mathcal D_{\infty,A,\Compct_{\Strat},\sigma_0}\) belongs to \(\mathfrak G_s\).
\end{proof}

\subsection{Exact decomposition and absorption of the tails}
\label{app:decomposition_and_tail_absorption}

We can now compare the integrated inner expansion with the full scaled local integral, defined by
\begin{equation}\label{eq:I_sigma}
  I_\sigma(a,x)
  = \int_{\mathbb H_c^m} \bar\chi_R (\sigma\zeta)\, \mathscr K_\sigma(\zeta;a,x) \,\dd\zeta.
\end{equation}
We introduce the following three tails:
\begin{align}
  \operatorname{Tail}_{\mathrm{ex}}(a,x,\sigma)
  &= \int_{\mathbb H_c^m} (1-\vartheta_\sigma(\zeta)) \bar\chi_R(\sigma\zeta)
  \mathscr K_\sigma(\zeta;a,x) \,\dd\zeta, \label{eq:def_Tail_ex} \\
  \operatorname{Tail}_{0}(a,x,\sigma)
  &= \mathcal A_x(0) \int_{\mathbb H_c^m} (1-\vartheta_\sigma(\zeta)) e^{-\Psi(\zeta;a,x)}
  \,\dd\zeta, \label{eq:def_Tail_0} \\
  \operatorname{Tail}_{1}(a,x,\sigma)
  &= \int_{\mathbb H_c^m} (1-\vartheta_\sigma(\zeta)) e^{-\Psi(\zeta;a,x)} \kappa_{1,x}(\zeta;a)
  \,\dd\zeta. \label{eq:def_Tail_1}
\end{align}
The first is the exact contribution outside the growing inner region. The other two are model tails,
arising because the coefficients \(\mathsf C_0\) and \(\mathsf C_1\) are defined by full-cone
integrals, whereas the inner expansion initially contains the truncated factor \(\vartheta_\sigma\).
The three integrals are absolutely convergent by the Gaussian-class bounds proved above.

\begin{lemma}[One-term local expansion]
  \label{lem:one_term_local_expansion}
  Assume \ref{ass:standing:measure}, \ref{ass:standing:chart}, and~\ref{ass:standing:density} of
  \Cref{ass:standing} with \(r\geqslant1\). Let \(A>0\). There exists \(\sigma_0>0\) and a function
  \(\mathscr E_{0,\mathrm{loc}}\in \mathfrak B_{r-1}\) on \(\mathcal D_{A,\Compct_{\Strat},
  \sigma_0}\) such that
  \begin{equation*}
    I_\sigma(a,x) = \mathsf C_0(a,x) + \sigma\mathscr E_{0,\mathrm{loc}}(a,x,\sigma).
  \end{equation*}
\end{lemma}

\begin{proof}
Taylor's theorem applied directly to \(\Delta\) and \(\mathcal A\), with one fewer term than in the
two-term expansion, gives on \(\mathcal D_{R_0,A,\Compct_{\Strat},\sigma_0}\)
\begin{equation*}
  \sigma^{-1}\Delta(\theta,\sigma\zeta)
  = \begin{bmatrix} \bfL(\theta)\\ \mathbf 0_{k\times m} \end{bmatrix}\zeta
  + \sigma\widetilde{\mathscr R}_\Delta(\zeta;a,x,\sigma), \qquad \mathcal A(\theta,\sigma\zeta)
  = \mathcal A(\theta,0) + \sigma\widetilde{\mathscr R}_{\mathcal A}(\zeta;a,x,\sigma),
\end{equation*}
with \(\widetilde{\mathscr R}_\Delta\in\mathfrak P_{r-1}\) and \(\widetilde{\mathscr R}_{\mathcal A}\in\mathfrak
P_{r-1}\). As in \Cref{lem:adm_scaled_exp_delta}, the remainders \(\widetilde{\mathscr R}_\Delta\) and
\(\widetilde{\mathscr R}_{\mathcal A}\) do not depend on \(a\); the argument \(a\) is retained only
to record the domain on which the admissible estimates are taken. Consequently,
\(\Psi_\sigma=\Psi+\sigma\widetilde{\mathscr R}_\Psi\) with \(\widetilde{\mathscr
R}_\Psi\in\mathfrak P_{r-1}\). On the support of \(\vartheta_\sigma\), after decreasing \(\sigma_0\)
if necessary, \(|\Psi_\sigma-\Psi|\leqslant1/2\). The integral remainder formula for \(e^{-u}\), the
closure properties of \Cref{lem:prop_adm_classes}, and the Gaussian comparison in
\Cref{lem:gaus_comp_can} therefore give, after zero extension to the full scaled domain,
\begin{equation*}
  \vartheta_\sigma(\zeta)\mathscr K_\sigma(\zeta;a,x)
  = \vartheta_\sigma(\zeta)e^{-\Psi(\zeta;a,x)} \mathcal A_x(0)
  + \sigma\mathscr R_{0,\mathrm{ker}}(\zeta;a,x,\sigma), \qquad \mathscr R_{0,\mathrm{ker}}\in\mathfrak G_{r-1}.
\end{equation*}
Integrating this identity yields
\begin{equation*}
  I_\sigma(a,x) = \mathsf C_0(a,x)
  + \sigma\int_{\mathbb H_c^m} \mathscr R_{0,\mathrm{ker}}(\zeta;a,x,\sigma)\,\dd\zeta
  + \operatorname{Tail}_{\mathrm{ex}}(a,x,\sigma) - \operatorname{Tail}_{0}(a,x,\sigma).
\end{equation*}
The integral of \(\mathscr R_{0,\mathrm{ker}}\) belongs to \(\mathfrak B_{r-1}\). For the exact tail,
\Cref{lem:scaled_kern_gaus_class} applied with \(s=r-1\), followed by
\Cref{lem:gaus_tail_beyond_window}, gives \(\sigma^{-1}\operatorname{Tail}_{\mathrm{ex}}\in
\mathfrak E_{r-1}^{1/2}\). For the model tail, the same tail estimate applied to \(\mathcal
A_x(0)e^{-\Psi}\in\mathfrak G_{r-1}\) gives \(\sigma^{-1}\operatorname{Tail}_{0}\in \mathfrak
E_{r-1}^{1/2}\). Since \(\mathfrak E_{r-1}^{1/2}\subset\mathfrak B_{r-1}\), both tails have the
required regularity after division by \(\sigma\). Absorbing these two tails into the coefficient of
\(\sigma\) gives the claim.
\end{proof}

\begin{lemma}[Exact decomposition of the scaled integral]
  \label{lem:decomp_loc_integral}
  Assume \ref{ass:standing:measure}, \ref{ass:standing:chart}, and~\ref{ass:standing:density} of
  \Cref{ass:standing} with \(r\geqslant 2\). Let \(A>0\). There exists \(\sigma_0>0\) and a function
  \(\mathscr E_{\mathrm{in}}\) from the class \(\mathfrak B_{r-2}\) on \(\mathcal D_{A,\Compct_{
  \Strat}, \sigma_0}\) such that
  \begin{align}
    I_\sigma(a,x) &= \mathsf C_0(a,x) + \sigma\mathsf C_1(a,x)
    + \sigma^2\mathscr E_{\mathrm{in}}(a,x,\sigma) \notag \\
    &\qquad + \operatorname{Tail}_{\mathrm{ex}}(a,x,\sigma) - \operatorname{Tail}_{0}(a,x,\sigma)
    - \sigma\operatorname{Tail}_{1}(a,x,\sigma). \label{eq:decomposition_of_scaled_local_integral}
  \end{align}
\end{lemma}

\begin{proof}
For \(\sigma_0\) sufficiently small, \(\operatorname{supp}\vartheta_\sigma \subset
\{\sigma\|\zeta\|\leqslant R\}\). Since \(\bar\chi_R(\xi) = 1\) when \(\|\xi\|\leqslant R\), we have
\(\bar\chi_R(\sigma\zeta) = 1\) on \(\operatorname{supp}\vartheta_\sigma\). This implies that
\(\vartheta_\sigma(\zeta) = \vartheta_\sigma(\zeta)\bar\chi_R(\sigma\zeta)\). Therefore
\begin{align*}
  I_\sigma(a,x) &= \int_{\mathbb H_c^m} \bar\chi_R (\sigma\zeta)\, \mathscr K_\sigma(\zeta;a,x)
  \,\dd\zeta \\
  &= \int_{\mathbb H_c^m} \vartheta_\sigma(\zeta) \mathscr K_\sigma(\zeta;a,x) \,\dd\zeta
  +\int_{\mathbb H_c^m} (\bar\chi_R (\sigma\zeta)
  - \vartheta_\sigma(\zeta)) \mathscr K_\sigma(\zeta;a,x) \,\dd\zeta \\
  &= \int_{\mathbb H_c^m} \vartheta_\sigma(\zeta)\mathscr K_\sigma(\zeta;a,x) \,\dd\zeta
  + \operatorname{Tail}_{\mathrm{ex}}(a,x,\sigma).
\end{align*}
Integrating \cref{eq:exp_inner} gives
\begin{align*}
  \int_{\mathbb H_c^m} \vartheta_\sigma\mathscr K_\sigma \,\dd\zeta
  &= \mathcal A_x(0)\int_{\mathbb H_c^m} \vartheta_\sigma
  (\zeta) e^{-\Psi(\zeta;a,x)} \,\dd\zeta
  + \sigma \int_{\mathbb H_c^m}
  \vartheta_\sigma(\zeta) e^{-\Psi(\zeta;a,x)} \kappa_{1,x}(\zeta;a) \,\dd\zeta \\
  &\qquad + \sigma^2\int_{\mathbb H_c^m} \mathscr R_{\mathrm{ker}}(\zeta;a,x,\sigma)\,\dd\zeta.
\end{align*}
Let us set \(\mathscr E_{\mathrm{in}}(a,x,\sigma) = \int_{ \mathbb H_c^m} \mathscr
R_{\mathrm{ker}}(\zeta;a,x,\sigma) \, \dd\zeta\). On the one hand, since \(\mathscr
R_{\mathrm{ker}}\in \mathfrak G_{r-2}\) on \(\mathcal D_{\infty,A, \Compct_{\Strat},\sigma_0}\), the
integration rule in \Cref{lem:prop_adm_classes} yields \(\mathscr E_{\mathrm{in}} \in\mathfrak
B_{r-2}\). On the other hand, by \cref{eq:C0} and \cref{eq:def_Tail_0},
\begin{equation*}
  \mathcal A_x(0) \int_{\mathbb H_c^m} \vartheta_\sigma e^{-\Psi} \,\dd\zeta
  = \mathsf C_0(a,x)-\operatorname{Tail}_0(a,x,\sigma).
\end{equation*}
Likewise, by \cref{eq:def_C1} and \cref{eq:def_Tail_1}, \(\int_{\mathbb H_c^m} \vartheta_\sigma
e^{-\Psi}\kappa_{1,x} \,\dd\zeta = \mathsf C_1(a,x)-\operatorname{Tail}_1(a,x,\sigma)\).
Substituting these two identities into the previous expression gives
\cref{eq:decomposition_of_scaled_local_integral}.
\end{proof}

We now show that the tails in \Cref{lem:decomp_loc_integral} are exponentially small at the order
needed for the integrated expansion.

\begin{lemma}[Negligibility of the tails]
  \label{lem:model_tails_exp_small}
  Assume \ref{ass:standing:measure}, \ref{ass:standing:chart}, and~\ref{ass:standing:density} of
  \Cref{ass:standing} with \(r\geqslant 2\). Let \(A>0\). There exists \(\sigma_0>0\) such that
  \(\sigma^{-2}\operatorname{Tail}_{\mathrm{ex}},\, \sigma^{-2}\operatorname{Tail}_{0} \in \mathfrak
  E_{r}^{1/2}\) and \(\sigma^{-1}\operatorname{Tail}_{1}\in \mathfrak E_{r-1}^{1/2}\) on \(\mathcal
  D_{A,\Compct_{\Strat}, \sigma_0}\).
\end{lemma}

\begin{proof}
First consider the model tails \(\operatorname{Tail}_0\) and \(\operatorname{Tail}_{1}\). By
\Cref{lem:gaus_comp_can},
\begin{equation*}
  \mathcal A_x(0)e^{-\Psi(\zeta;a,x)} \in \mathfrak G_{r} \qquad \text{on}
  \qquad \mathcal D_{\infty,A,\Compct_{\Strat},\sigma_0},
\end{equation*}
and, using \(\kappa_{1,x}\in\mathfrak P_{r-1}\), \(e^{-\Psi}\kappa_{1,x} \in \mathfrak G_{r-1}\) on
\( \mathcal D_{\infty,A,\Compct_{\Strat},\sigma_0}\). Thus \(\operatorname{Tail}_0\) and
\(\operatorname{Tail}_1\) are precisely the abstract tails \(\mathbb T_\sigma[F]\) from
\Cref{lem:gaus_tail_beyond_window}, with admissible orders \(r\) and \(r-1\), respectively. Hence
\begin{equation*}
  \operatorname{Tail}_0\in\mathfrak E_{r}^{1/2} \quad\text{and} \quad \operatorname{Tail}_1
  \in\mathfrak E_{r-1}^{1/2}.
\end{equation*}
For the exact tail, let us set
\begin{equation*}
  F_{\mathrm{ex}}(\zeta;a,x,\sigma) = \bar\chi_R(\sigma\zeta)\mathscr K_\sigma(\zeta;a,x),
\end{equation*}
and apply the zero extension from \Cref{lem:scaled_kern_gaus_class} with \(s=r\). We arrive at \(
F_{\mathrm{ex}} \in \mathfrak G_{r} \) on the domain \( \mathcal D_{\infty,A,
\Compct_{\Strat},\sigma_0}\). Since \(\operatorname{Tail}_{ \mathrm{ex}} = \mathbb
T_\sigma[F_{\mathrm{ex}}],\) another application of \Cref{lem:gaus_tail_beyond_window} yields
\begin{equation*}
  \operatorname{Tail}_{\mathrm{ex}} \in \mathfrak E_{r}^{1/2}.
\end{equation*}
The final statement follows from item~\ref{item:mult_div_power_sigma} of
\Cref{lem:prop_adm_classes} applied with \(M=1\) and \(M=2\).
\end{proof}

We finally absorb the exponentially small tails into the \(\sigma^2\)-remainder. This provides the
two-term integrated expansion in the parameter class available under the present assumptions.

\begin{corollary}[Absorption of the tails]
  \label{cor:tails_absorbed_into_sigma2_remainder}
  Assume \ref{ass:standing:measure}, \ref{ass:standing:chart}, and~\ref{ass:standing:density} of
  \Cref{ass:standing} with \(r\geqslant 2\). Let \(A>0\). There exists \(\sigma_0>0\) such that the
  function
  \begin{equation*}
    \mathscr E_{\mathrm{tail}}(a,x,\sigma)
    = \sigma^{-2}\operatorname{Tail}_{\mathrm{ex}}(a,x,\sigma)
    - \sigma^{-2}\operatorname{Tail}_{0}(a,x,\sigma)
    - \sigma^{-1}\operatorname{Tail}_{1}(a,x,\sigma)
  \end{equation*}
  belongs to \(\mathfrak E_{r-1}^{1/2}\), and hence to \(\mathfrak B_{r-1}\), on \(\mathcal
  D_{A,\Compct_{\Strat}, \sigma_0}\). Consequently,
  \begin{equation*}
    I_\sigma(a,x) = \mathsf C_0(a,x) + \sigma\mathsf C_1(a,x)
    + \sigma^2\mathscr E_{\mathrm{loc}}(a,x,\sigma),
  \end{equation*}
  where \(\mathscr E_{\mathrm{loc}} = \mathscr E_{ \mathrm{in}} + \mathscr E_{\mathrm{tail}} \in
  \mathfrak B_{r-2}\) on \(\mathcal D_{A, \Compct_{\Strat}, \sigma_0}\).
\end{corollary}

\begin{proof}
The first item follows from \Cref{lem:model_tails_exp_small} and the vector-space property of
\(\mathfrak E_{r-1}^{1/2}\). Since \(\mathfrak E_{r-1}^{1/2}\subset \mathfrak B_{r-1}\) by
\Cref{lem:prop_adm_classes}, and \(\mathscr E_{\mathrm{in}}\in\mathfrak B_{r-2}\) by
\Cref{lem:decomp_loc_integral}, the total remainder \(\mathscr E_{\mathrm{loc}} = \mathscr
E_{\mathrm{in}} + \mathscr E_{\mathrm{tail}}\) belongs to \(\mathfrak B_{r-2}\).
\end{proof}

\section{Far-field control of the global remainder}
\label{app:far_field}

To show that the contribution of the far-field integral to the regularized heat kernel is
negligible, we first establish the desired properties of smoothness and separation of the
transported cutoff function.

\paragraph{Hypotheses in force.}
This section uses the measure item~\ref{ass:standing:measure} and the chart item
\ref{ass:standing:chart}. It does not use the density-regularity item~\ref{ass:standing:density} or
the positivity item~\ref{ass:standing:positivity}. Here, the only requirement is that the far-field
integral is taken with respect to the probability measure \(q\).

Recall that
\begin{equation}\label{eq:pfar2}
  p_{\textrm{far},\sigma}(y) = \int_{\Man} \big( 1
  - \chi_{\pi(y)} (x')\big)\, \phi_\sigma(y-x')\, q(\dd x').
\end{equation}

\begin{lemma}[Regularity and support separation]
  \label{lem:transp_cutoff_reg_supp}
  Assume \AssChart. For \(x=\varphi(\theta)\in\Compct_{\Strat}\), let \(\chi_x\) be defined
  \cref{eq:chi_x}. Then the
  following statements hold.
  \vspace*{-10pt}

  \begin{enumerate}\itemsep=0pt
    \item The function
    \(\chi_x\) is \(C^{r+1}\) on \(\Man\).

    \item The mapping \((\theta,x')\mapsto\chi_{\varphi
    (\theta)}(x')\) has uniformly bounded \(\theta\)-derivatives up to order \(r+1\).

    \item For every \(x'\in\Man\),
    \(1-\chi_x(x')\neq 0\) implies that \( \|x'-x\|\geqslant 2\delta_0\).

    \item For every multi-index \(\beta\) satisfying
    \(1\leqslant |\beta| \leqslant r+1\),
    \begin{equation*}
      \partial_\theta^\beta \chi_{\varphi(\theta)}(x')\neq0\quad\Longrightarrow
      \quad \|x'-\varphi(\theta)\|\geqslant 2\delta_0 .
    \end{equation*}
  \end{enumerate}
\end{lemma}

\begin{proof}
Let us write \(M_\theta(\cdot) = M(\theta,\cdot)\). By definition, for every \(x'\in
M_\theta(\mathbb H_c^m \cap \mathbb B_{4R}^m)=\mathcal U_\theta\), the transported cutoff is
\begin{equation*}
  \chi_{\varphi(\theta)}\bigl(x'\bigr) = \chi(\|M_\theta^{-1}(x')\|^2/R^2).
\end{equation*}
According to item~\ref{item:Crplusone-corner-chart} of \Cref{lem:unif_par_comp_strat}, the mapping
\(M_\theta^{-1}\) is \(C^{r+1}\) on \(\mathcal U_\theta\). Since the functions \(v\mapsto \|v \|^2\)
and \(\chi\) are \(C^\infty\) everywhere, we conclude that \(\chi_{\varphi(\theta)}\) is \(C^{r+1}\)
on the open set \(\mathcal U_\theta\). On the other hand, this function vanishes outside the set
\(M_\theta(\mathbb H_c^m \cap {\mathbb B}_{2R}^m)\Subset \mathcal U_\theta\). Thus extending it by
zero outside \(\mathcal U_\theta\) does not cause any loss of regularity, and
\(\chi_{\varphi(\theta)}\) is \(C^{r+1}\) on \(\Man\).

Note also that if \(x'=\Phi(\xi_{\mathcal S},\xi_{\mathcal C})\) is written in the original local
corner chart, then \(x'=\Phi(\xi_{\mathcal S},\xi_{\mathcal C}) = \Phi(\theta+\xi_{\mathcal
S}-\theta,\xi_{\mathcal C}) = M(\theta,(\xi_{\mathcal S}-\theta,\xi_{\mathcal C}))\) and, therefore,
\begin{equation*}
  \chi_{\varphi(\theta)}(x') = \chi\Big(\big\{\|\xi_{\mathcal S} - \theta\|^2
  + \|\xi_{\mathcal C}\|^2\big\}/R^2\Big).
\end{equation*}
This formula yields the uniform boundedness of derivatives of any order with respect to \(\theta\).

It remains to prove the support separation. By \Cref{lem:unif_par_comp_strat},
\begin{equation*}
  \Man\cap\mathbb B_{2\delta_0}^d(x) \subset M_\theta\bigl(\mathbb H_c^m\cap\mathbb B_{R}^m\bigr).
\end{equation*}
Thus, if \(x'\in\Man\) satisfies \(\|x'-x\|<2\delta_0\), then \(x'=M_\theta(\xi)\) for some \(\xi
\in \mathbb H_c^m\) with \(\|\xi\|<R\). Since \(\chi=1\) on \([0,1]\), we get
\(\chi_x(x')=\chi(\|\xi\|^2/R^2)=1\). Hence \(1-\chi_x(x') =0\) whenever \(\|x'-x\|<2\delta_0\),
which proves the third item.

The derivative statement follows from the same observation. If \(\|x'-\varphi(\theta)\|<2\delta_0\),
then \(x'=M_\theta(\xi)\) for some \(\|\xi\|<R\). For \(\theta'\) sufficiently close to \(\theta\),
the coordinate of the fixed point \(x'\) in the \(M_{\theta'}\)-chart still lies in \(\mathbb
B_{R}^m\). Therefore \(\chi_{\varphi(\theta')}(x')=1\) for all \(\theta'\) in a neighborhood of
\(\theta\). Thus the map \(\theta'\mapsto\chi_{\varphi(\theta')}(x')\) is locally constant near
\(\theta\), and every positive-order \(\theta\)-derivative vanishes at \(\theta\). This proves the
last item.
\end{proof}

We now apply the support separation to the far integral itself. Let us introduce the set
\begin{equation}\label{eq:Compct_A}
  \Compct_A
  =
  \left\{
    y_\tau(a,x):
    \|a\|\leqslant A,\ x\in\Compct_{\Strat},\ 0\leqslant\tau\leqslant\sigma_0
  \right\}
  \subset\mathbb R^d .
\end{equation}
This set is compact after decreasing \(\sigma_0\) so that the reconstructed points remain in the
fixed tubular neighborhood.

\begin{lemma}[Exponential smallness of the far field integral]
  \label{lem:uniform_exponentially_small_far_term}
  Assume \ref{ass:standing:measure} and~\ref{ass:standing:chart} of \Cref{ass:standing}.
  For every \(A>0\) there exist \(\sigma_0>0\) and \(c>0\) such that, for every multi-index
  \(\beta\) and every \(j\in\mathbb N\) with \(|\beta|+j\leqslant r\), there is a constant
  \(C_{\beta,j}>0\) satisfying
  \begin{equation*}
    \sup_{y\in \Compct_A}\left| \partial_y^\beta\partial_\sigma^j p_{{\mathrm{far}},\sigma}(y)
    \right| \leqslant C_{\beta,j}e^{-c/\sigma^2}.
  \end{equation*}
\end{lemma}

\begin{proof}
Since \(y\in\Compct_A\), there exist
\[
  a_0\in\overline{\mathbb B}_A^{c+k},\qquad
  x_0\in\Compct_{\Strat},\qquad
  \tau\in[0,\sigma_0]
\]
such that \(y=y_\tau(a_0,x_0)\). By \Cref{lem:rec_map_prop}, \(\pi(y)=x_0\) and
\[
  \|y-x_0\|=\tau\|a_0\|\leqslant A\sigma_0.
\]
Choose \(\sigma_0\leqslant\delta_0/A\), so that \(y\in\mathbb B_{\delta_0}(x_0)\). Since
\(x_0\in\Compct_{\Strat}\), the transported cutoff \(\chi_{x_0}\) has the support-separation
property proved in \Cref{lem:transp_cutoff_reg_supp}. This implies that whenever the integrand in
\eqref{eq:pfar2} is nonzero, \(\|x'-x_0\|\geqslant2\delta_0\), and therefore
\begin{equation*}
  \|y-x'\| \geqslant \|x'-x_0\|-\|y-x_0\| \geqslant \delta_0 .
\end{equation*}
We now differentiate \cref{eq:pfar2}. The map \(y\mapsto \pi(y)\) is \(C^{r}\), and all its
derivatives up to order \(r\) are bounded on the compact subset \(\Compct_A\) of \(\mathbb R^d\).
The family \(x \mapsto\chi_x\) has uniformly bounded base-point derivatives for
\(x\in\Compct_{\Strat}\). Hence, after differentiating, every term is bounded by a finite sum of
integrals of the form
\begin{equation*}
  C_{\beta,j}\sigma^{-N_{\beta,j}} \int_{\Man \setminus \mathbb B_{\delta_0}(y)}
  \exp\Bigl(-\frac{\|y-x'\|^2}{4\sigma^2}\Bigr) q(\dd x').
\end{equation*}
Here the integer \(N_{\beta,j}\) accounts for derivatives of the Gaussian kernel with respect to
\(y\) and \(\sigma\). Since \(q\) is a probability measure (\AssMeas), this gives
\begin{equation*}
  \left| \partial_y^\beta\partial_\sigma^j p_{{\mathrm{far}},\sigma}(y) \right|
  \leqslant C_{\beta,j}\sigma^{-N_{\beta,j}} \exp\left(-{\delta_0^2}/({4\sigma^2})\right).
\end{equation*}
Decreasing the exponential rate absorbs the polynomial factor \(\sigma^{-N_{\beta,j}}\), and the
asserted estimate follows.
\end{proof}

It remains to express the ambient far-field estimate in the variables used in the local expansion.
For bounded \(a\), the reconstructed points \(y_\sigma(a,x)\) form a compact subset of the tubular
neighborhood, and their projections are the corresponding base points \(x\). Therefore the ambient
estimate applies uniformly to this family. The resulting contribution is exponentially small, and
remains so after multiplication by any power of \(\sigma\).

\begin{lemma}[Far term in rescaled variables]
  \label{lem:far_term_bound_layer}
  Assume \ref{ass:standing:measure} and~\ref{ass:standing:chart} of \Cref{ass:standing}.
  Fix \(A>0\). There exists \(\sigma_0>0\) such that
  \begin{equation*}
    \widetilde p_{{\mathrm{far}},\sigma}(a,x) = p_{{\mathrm{far}},\sigma}\bigl(y_\sigma(a,x)\bigr)
  \end{equation*}
  is well defined and satisfies \(\widetilde p_{{\mathrm{far}}, \sigma}\in \mathfrak E_r^2\) on
  \(\mathcal D_{A, \Compct_{\Strat},\sigma_0}\). Consequently, the normalized remainder
  \begin{equation*}
    \mathscr E_{\mathrm{far}}(a,x,\sigma)
    = (2\pi)^{d/2} \sigma^{k-2} \widetilde p_{{\mathrm{far}},\sigma}(a,x)
  \end{equation*}
  belongs to \(\mathfrak E_r^2\), and hence to \(\mathfrak B_r\), on the same parameter domain.
\end{lemma}

\begin{proof}
By \Cref{lem:rec_map_prop}, suitable choice of \(\sigma_0\) guarantees that the map
\((a,x,\sigma)\mapsto y_\sigma(a,x) \) is \(C^{r}\) on \(\mathcal D_{A, \Compct_{\Strat},\sigma_0}\)
and satisfies
\begin{equation*}
  \pi\bigl(y_\sigma(a,x)\bigr)=x, \qquad \left\|y_\sigma(a,x)-x\right\| = \sigma\|a\|.
\end{equation*}
\Cref{lem:uniform_exponentially_small_far_term} implies that for all \(|\mu|+\lambda\leqslant r\),
\begin{equation*}
  \sup_{y\in \Compct_A}\left| \partial_y^\mu\partial_\sigma^\lambda p_{{\mathrm{far}},\sigma}(y)
  \right| \leqslant C_{\mu,\lambda}e^{-c/\sigma^2}.
\end{equation*}
Since the parameter derivatives of \(y_\sigma(a,x)\) are uniformly bounded, repeated use of the
chain rule gives
\begin{equation*}
  \left| \partial_a^\alpha\partial_\theta^\beta\partial_\sigma^j p_{{\mathrm{far}},\sigma}
  \bigl(y_\sigma(a,\varphi(\theta))\bigr) \right| \leqslant C_{\alpha,\beta,j}e^{-c'/\sigma^2}
\end{equation*}
for all \(|\alpha|+|\beta|+j\leqslant r\). Hence \(\widetilde p_{{\mathrm{far}},\sigma}\in\mathfrak
E_r^2\) on \(\mathcal D_{A,\Compct_{\Strat},\sigma_0}\). Finally, by the multiplication and division
rule for powers of \(\sigma\), the factor \(\sigma^{k-2}\) preserves the class \(\mathfrak E_r^2\).
Hence \(\mathscr E_{\mathrm{far}}\in\mathfrak E_r^2\). Since \(\mathfrak E_r^2\subset\mathfrak
B_r\), the proof is complete.
\end{proof}

\section{Logarithmic stability and conical-layer
  differentiation}
\label{app:logarithmic_differentiation}

This section collects the analytic tools that allow us to pass from the density expansion to
expansions of its logarithm and of its first- and second-order derivatives.

\paragraph{Hypotheses in force.}
The chain-rule and reconstructed-coefficient estimates use only the
geometric item~\ref{ass:standing:chart}. When this section is applied to \(\log p_\sigma\), the
assumptions needed for the density expansion have already entered through
\Cref{thm:first_coef_kernel,thm:2nd_coef_kernel}.

\begin{lemma}[Logarithm of a unit perturbation]
  \label{lem:logarithm_of_unit_perturbation}
  Let \(s\geqslant0\) be an integer and let \(B\in\mathfrak B_s\) on \(\mathcal
  D_{A,\Compct_{\Strat}, \sigma_0}\). There is \(\sigma'_0>0\) such that
  \begin{equation}\label{eq:log_of_exp}
    \log\Bigl( 1+\sigma B(a,x,\sigma)\Bigr) = \sigma B(a,x,\sigma)
    + \sigma^2\mathscr R_{\log}(a,x,\sigma)
  \end{equation}
  for some \(\mathscr R_{\log}\in\mathfrak B_s\) on \(\mathcal D_{A,\Compct_{\Strat}, \sigma'_0}\).
\end{lemma}

\begin{proof}
Since \(B\in\mathfrak B_s\), it is uniformly bounded. Thus, there is \(\sigma'_0\in(0,\sigma_0]\),
such that
\begin{equation*}
  \sup_{\mathcal D_{A,\Compct_{\Strat},\sigma_0}}\left| \sigma B(a,x,\sigma) \right|
  \leqslant \frac12.
\end{equation*}
Hence, the logarithm is well defined. Define
\begin{equation*}
  \Xi_{\log}(u) = \begin{cases} \dfrac{\log(1+u)-u}{u^2}, & u\neq 0, \\[1.2ex]
    -\frac12, & u=0. \end{cases}
\end{equation*}
Then \(\Xi_{\log}\in C^\infty((-1,\infty))\), and \(\log(1+u)=u+u^2\Xi_{\log}(u)\). Since the range
of \(\sigma B\) is contained in \([-1/2,1/2]\Subset(-1, \infty)\), the smooth-composition property
gives \(\Xi_{\log}\circ (\sigma B)\in\mathfrak B_s\). Therefore,
\begin{equation*}
  \log(1 + \sigma B) = \sigma B + \sigma^2\underbrace{ B^2\,\Xi_{\log}
  \circ(\sigma B)}_{\mathscr R_{\log}}.
\end{equation*}
Thus \cref{eq:log_of_exp} holds with \(\mathscr R_{\log}\in \mathfrak B_s\), by the product and
composition properties of \(\mathfrak B_s\).
\end{proof}

We recall that \(\theta(y)=\varphi^{-1}(\pi(y))\) and \(\nu(y) = [\,\bfC(x)\
\bfN(x)\,]^\top(y-\pi(y))\), see \eqref{eq:theta-nu-def}. Furthermore, we have used the first-order
differentials of these mappings:
\begin{equation*}
  J_\theta(y)=\Diff_y\theta(y), \qquad J_\nu(y)=\Diff_y\nu(y).
\end{equation*}
In the second-order expansions, we also need the second-order differentials of the mappings
\(\theta\) and \(\nu\). For \(v\in\mathbb R^{c+k}\) and \(w\in\mathbb R^{m-c}\), define the
symmetric matrices
\begin{equation*}
  \mathsf Q^\nu_y[v] = \sum_{\mu=1}^{c+k} v_\mu\nabla_y^2\nu_\mu(y), \qquad \mathsf Q^\theta_y[w]
  = \sum_{i=1}^{m-c} w_i\nabla_y^2\theta_i(y),
\end{equation*}
where \(\nu_\mu(y)\) and \(\theta_i(y)\) refer to the corresponding coordinate of the vectors
\(\nu(y)\) and \(\theta(y)\), respectively. Let \(F=F(a,x,\sigma)\) be scalar-valued and \(C^1\) in
\((a,x)\), with \(x\)-derivatives understood after the pullback \(x=\varphi(\theta)\). If we define
\begin{equation*}
  F^\sharp(y,\sigma) = F\bigl(a(y,\sigma),\pi(y),\sigma\bigr)
\end{equation*}
then the chain rule readily yields
\begin{equation}\label{eq:app_chain_rule_gradient}
  \nabla_yF^\sharp = J_\theta(y)^\top\nabla_\theta F + \frac1\sigma J_\nu(y)^\top\nabla_aF.
\end{equation}
If, in addition, \(F\) is \(C^2\) in \((a,x)\), then
\begin{align}
  \label{eq:app_chain_rule_hessian} \nabla_y^2F^\sharp
  &= J_\theta(y)^\top \Diff_\theta^2F\,J_\theta(y) + \mathsf Q^\theta_y[\nabla_\theta F] \\
  &\quad + \frac1\sigma \Bigl[ J_\theta(y)^\top \Diff_{\theta a}^2F\,J_\nu(y)
  + J_\nu(y)^\top \Diff_{a\theta}^2F\,J_\theta(y) + \mathsf Q^\nu_y[\nabla_aF] \Bigr] \notag \\
  &\quad + \frac1{\sigma^2} J_\nu(y)^\top \Diff_a^2F\,J_\nu(y). \notag
\end{align}
All derivatives of \(F\) on the right-hand sides of the above formulae are evaluated at \((a, x,
\sigma) = \bigl( a(y,\sigma),\pi(y),\sigma\bigr)\). These chain rules will be applied after pulling
the observation point back to the reconstructed form \(y=y_\sigma(a,x)\). We therefore need the
corresponding differential coefficients, evaluated along this reconstructed family, to be admissible
functions of \((a,x,\sigma)\). This follows from the smoothness of the tubular coordinates on
\(\mathcal U\), the compactness of the reconstructed observation set, and the uniform smoothness of
\(y_\sigma\).

\begin{lemma}[Admissibility of reconstructed differential coefficients]
  \label{lem:reconstructed_differential_coefficients_in_Br}
  Assume \AssChart{} with \(r\geqslant1\). There exists \( \sigma_0 > 0\) such that the coefficient
  fields \(\mathbf J_\theta(a,x,\sigma) = J_\theta\bigl(y_\sigma (a,x)\bigr)\) and \(\mathbf
  J_\nu(a,x,\sigma) = J_\nu\bigl( y_\sigma(a,x)\bigr)\) belong to \(\mathfrak B_{r-1}\) on
  \(\mathcal D_{A,\Compct_{\Strat},\sigma_0}\). If, furthermore, \(r\geqslant2\), the coefficients
  of the linear maps \(\mathbf Q_\theta(a,x,\sigma)[w] = \mathsf Q^\theta_{y_\sigma(a,x)}[w]\) and
  \(\mathbf Q_\nu (a,x,\sigma)[v] = \mathsf Q^\nu_{y_\sigma(a,x)}[v]\) belong to \(\mathfrak
  B_{r-2}\) on the same domain.
\end{lemma}

\begin{proof}
Choose \(\sigma_0>0\) so that \(y_\sigma(a,x)\in\mathcal U\) for all \(\|a\|\leqslant A\), all
\(x\in\Compct_{\Strat}\) and all \(0\leqslant\sigma\leqslant\sigma_0\). Then the reconstructed
observation set \(\Compct_A\) of \eqref{eq:Compct_A}
is a compact subset of \(\mathcal U\). According to \Cref{lem:orth_tub_coord}, the maps \(\theta\)
and \(\nu\) are \(C^{r}\) on \(\mathcal U\). Hence \(J_\theta\) and \(J_\nu\) are \(C^{r-1}\),
provided that \(r\geqslant1\), while \(\mathsf Q^\theta\) and \(\mathsf Q^\nu\) have \(C^{r-2}\)
coefficients, provided that \(r\geqslant2\). Their derivatives up to the required order are
therefore uniformly bounded on \(\Compct_A\).

The reconstructed map \((a,x,\sigma)\mapsto y_\sigma(a,x)\) is \(C^{r}\) after the pullback \( x =
\varphi (\theta)\) in view of \Cref{lem:orth_tub_coord}. Composing the coefficient fields above with
this map and applying the ordinary chain rule gives uniform bounds for all
\((a,\theta,\sigma)\)-derivatives up to the required order. This is exactly the defining estimate
for membership in \(\mathfrak B_s\) for \(s = r-1\) or \(s = r-2\), depending on the considered
quantity.
\end{proof}

\section{Proofs of the main theorems}

In this section we present the proofs of the main theorems, which are based on all the technical
results established in the preceding sections of this appendix. We give the proof of
\Cref{thm:first_coef_kernel} separately because its endpoint case \(r=1\) uses the one-term local
expansion \Cref{lem:one_term_local_expansion}. We then prove \Cref{thm:2nd_coef_kernel} and
\Cref{thm:2nd-order-score}; the proof of \Cref{thm:first-order-score} is obtained by the same
chain-rule argument with the first-order logarithmic expansion.

\begin{proof}[Proof of \Cref{thm:first_coef_kernel}]
Fix \(A>0\). We choose \(\sigma_0\) small enough for the preceding appendix results to hold, and
shrink it below without changing notation. Let \((y,\sigma)\in\mathcal
Y_{A,\Compct_{\Strat},\sigma_0}\), set \(x=\pi(y)\) and \(a=a(y,\sigma)\). By
\Cref{lem:rec_map_prop}, \(y=y_\sigma(a,x)\).

Decompose \(p_\sigma=p_{{\rm loc},\sigma}+p_{{\rm far},\sigma}\) as in
\cref{eq:p_sigma_decomposition}. The local change of variables gives
\begin{equation*}
  p_{{\rm loc},\sigma}\bigl(y_\sigma(a,x)\bigr) = \sigma^{-k}(2\pi)^{-d/2}I_\sigma(a,x).
\end{equation*}
By \Cref{lem:one_term_local_expansion},
\begin{equation*}
  I_\sigma(a,x) = \mathsf C_0(a,x) + \sigma\mathscr E_{0,\mathrm{loc}}(a,x,\sigma),
  \qquad \mathscr E_{0,\mathrm{loc}}\in\mathfrak B_{r-1}.
\end{equation*}
The far term is exponentially small in rescaled variables: \Cref{lem:far_term_bound_layer} implies
that
\begin{equation*}
  \mathscr E_{0,\mathrm{far}}(a,x,\sigma)
  = (2\pi)^{d/2}\sigma^{k-1} p_{{\rm far},\sigma}\bigl(y_\sigma(a,x)\bigr)
\end{equation*}
belongs to \(\mathfrak E_r^2\subset\mathfrak B_r\), hence to \(\mathfrak B_{r-1}\). Thus, with
\(\mathscr E_0=\mathscr E_{0,\mathrm{loc}}+ \mathscr E_{0,\mathrm{far}}\), the density expansion has
the stated form and \(\mathscr E_0\in\mathfrak B_{r-1}\).

For the logarithm, \Cref{lem:uniform_positivity_of_C0} gives \(\mathsf C_0^{-1}\in\mathfrak B_r\)
and \(\mathsf L_0=\log\mathsf C_0\in\mathfrak B_r\). Hence \(B=\mathsf C_0^{-1}\mathscr
E_0\in\mathfrak B_{r-1}\). After decreasing \(\sigma_0\), \Cref{lem:logarithm_of_unit_perturbation}
shows that the logarithm of \(1+\sigma B\) can be written as \(\sigma\mathscr E_{0,\log}\) with
\(\mathscr E_{0,\log}\in\mathfrak B_{r-1}\). This proves the logarithmic expansion.
\end{proof}

\begin{proof}[Proof of \Cref{thm:2nd_coef_kernel}]
Fix \(A>0\). We choose \(\sigma_0\) small enough for all the preceding appendix results to hold, and
shrink it below without changing notation. Let \((y,\sigma)\in\mathcal Y_{A,\Compct_{\Strat},
\sigma_0}\), set \(x=\pi(y)\) and \(a=a(y,\sigma)\). By \Cref{lem:rec_map_prop},
\(y=y_\sigma(a,x)\).

Decompose the heat kernel by means of the transported cutoff,
\begin{equation}
  \label{eq:p_sigma_decomposition}
  p_\sigma(y)=p_{{\rm loc},\sigma}(y)+p_{{\rm far},\sigma}(y),
\end{equation}
where
\begin{align*}
  p_{{\rm loc},\sigma}(y)
  &= \int_\Man \chi_{\pi(y)}(x')\phi_\sigma(y-x')\rho(x')\,\dd\vol_\Man(x'), \\
  p_{{\rm far},\sigma}(y)
  &= \int_\Man \bigl(1-\chi_{\pi(y)}(x')\bigr) \phi_\sigma(y-x')\rho(x')\,\dd\vol_\Man(x').
\end{align*}
Write \(x=\varphi(\theta)\). In the local term, the change of variables
\(x'=M(\theta,\sigma\zeta)\), together with the definition of the amplitude \(\mathcal A\) and of
the scaled exponent \(\Psi_\sigma\), gives
\begin{equation*}
  p_{{\rm loc},\sigma}\bigl(y_\sigma(a,x)\bigr) = \sigma^{-k}(2\pi)^{-d/2} I_\sigma(a,x),
\end{equation*}
where \(I_\sigma\) is the scaled local integral defined in \eqref{eq:I_sigma}. By
\Cref{cor:tails_absorbed_into_sigma2_remainder},
\begin{equation*}
  I_\sigma(a,x) = \mathsf C_0(a,x)+\sigma\mathsf C_1(a,x) +\sigma^2\mathscr E_{\rm loc}(a,x,\sigma),
  \qquad \mathscr E_{\rm loc}\in\mathfrak B_{r-2}.
\end{equation*}
Thus the local part has the required two-term expansion.

The far part is absorbed into the same order. Indeed, \Cref{lem:far_term_bound_layer} gives
\begin{equation*}
  \mathscr E_{\rm far}(a,x,\sigma)
  = (2\pi)^{d/2}\sigma^{k-2} p_{{\rm far},\sigma}\bigl(y_\sigma(a,x)\bigr)
  \in \mathfrak E_r^2\subset \mathfrak B_r .
\end{equation*}
Hence, with \(\mathscr E_1=\mathscr E_{\rm loc}+ \mathscr E_{\rm far}\in\mathfrak B_{r-2}\),
\begin{equation*}
  p_\sigma(y) = \sigma^{-k}(2\pi)^{-d/2} \left[ \mathsf C_0(a,x)+\sigma\mathsf C_1(a,x)
  +\sigma^2\mathscr E_1(a,x,\sigma) \right].
\end{equation*}

It remains to pass to the logarithm. By \Cref{lem:uniform_positivity_of_C0}, \(\mathsf C_0\) is
uniformly bounded away from zero on \(\Compct_{\Strat}\times\overline{\mathbb B}_A\), \(\mathsf
C_0^{-1}\in\mathfrak B_r\), \(\mathsf L_0=\log\mathsf C_0\in\mathfrak B_r\), and \(\mathsf
L_1=\mathsf C_1/\mathsf C_0\in\mathfrak B_{r-1}\). Therefore
\begin{equation*}
  p_\sigma(y) = \sigma^{-k}(2\pi)^{-d/2}\mathsf C_0(a,x) \left[1+\sigma B(a,x,\sigma)\right],
\end{equation*}
with \(B(a,x,\sigma) = \mathsf L_1(a,x) +\sigma\mathsf C_0(a,x)^{-1}\mathscr E_1(a,x,\sigma)
\in\mathfrak B_{r-2}\), where we used the closure properties in \Cref{lem:prop_adm_classes}. After
possibly decreasing \(\sigma_0\), \Cref{lem:logarithm_of_unit_perturbation} applies to \(B\) with
\(s=r-2\), and yields
\begin{equation*}
  \log\bigl(1+\sigma B(a,x,\sigma)\bigr) = \sigma\mathsf L_1(a,x)
  +\sigma^2\mathscr E_{1,\log}(a,x,\sigma), \qquad \mathscr E_{1,\log}\in\mathfrak B_{r-2}.
\end{equation*}
Combining this identity with the prefactor and with \(\mathsf L_0=\log\mathsf C_0\) proves the
logarithmic expansion and the stated uniform regularity of both remainders.
\end{proof}

\begin{proof}[Proof of \Cref{thm:2nd-order-score}]
We use the logarithmic expansion from \Cref{thm:2nd_coef_kernel}. For \((y,\sigma)\in \mathcal
Y_{A,\Compct_{\Strat},\sigma_0}\), write \(x=\pi(y)\), \(a=a(y,\sigma)\), and, by
\Cref{lem:rec_map_prop}, \(y=y_\sigma(a,x)\). Set
\begin{equation*}
  F(a,x,\sigma) = \mathsf L_0(a,x)+\sigma\mathsf L_1(a,x) +\sigma^2E(a,x,\sigma),
  \qquad E=\mathscr E_{1,\log}.
\end{equation*}
The two explicit terms \(-k\log\sigma\) and \(-(d/2)\log(2\pi)\) in the logarithmic expansion do not
depend on \(y\), so they disappear after spatial differentiation. All the differential coefficient
fields below are evaluated at \((a,x,\sigma)\), that is along \(y_\sigma(a,x)\).

First define, for a scalar function \(U(a,x,\sigma)\),
\begin{equation*}
  \mathcal G_1[U] = \mathbf J_\nu^\top\nabla_a U, \qquad \mathcal G_0[U]
  = \mathbf J_\theta^\top\nabla_\theta U.
\end{equation*}
The gradient chain rule \eqref{eq:app_chain_rule_gradient} gives
\begin{equation*}
  \score(y) = \sigma^{-1}\mathcal G_1[F]+ \mathcal G_0[F] = \sigma^{-1}\mathsf S_0(a,x,\sigma)
  +\mathsf S_1(a,x,\sigma) +\sigma\mathscr R_{1,\score}(a,x,\sigma),
\end{equation*}
where
\begin{equation*}
  \mathsf S_0=\mathcal G_1[\mathsf L_0], \quad \mathsf S_1=\mathcal G_0[\mathsf L_0]
  +\mathcal G_1[\mathsf L_1], \quad \mathscr R_{1,\score}
  = \mathcal G_1[E]+ \mathcal G_0[\mathsf L_1] +\sigma\mathcal G_0[E].
\end{equation*}
This is the asserted score expansion.

For the Hessian, introduce the three operators appearing in the second chain rule:
\begin{align*}
  \mathcal H_2[U] &= \mathbf J_\nu^\top \Diff_a^2U\,\mathbf J_\nu, \\
  \mathcal H_1[U] &= \mathbf J_\theta^\top \Diff_{\theta a}^2U\,\mathbf J_\nu
  +\mathbf J_\nu^\top \Diff_{a\theta}^2U\,\mathbf J_\theta +\mathbf Q_\nu[\nabla_aU], \\
  \mathcal H_0[U] &= \mathbf J_\theta^\top \Diff_\theta^2U\,\mathbf J_\theta
  +\mathbf Q_\theta[\nabla_\theta U].
\end{align*}
By \cref{eq:app_chain_rule_hessian},
\begin{equation*}
  \Hess(y) = \sigma^{-2}\mathcal H_2[F] +\sigma^{-1}\mathcal H_1[F] +\mathcal H_0[F].
\end{equation*}
Expanding \(F=\mathsf L_0+\sigma\mathsf L_1+ \sigma^2E\) in this identity gives
\begin{equation*}
  \Hess(y) = \sigma^{-2} \left( \mathsf H_0(a,x,\sigma) +\sigma\mathsf H_1(a,x,\sigma)
  +\sigma^2\mathscr R_{1,\Hess}(a,x,\sigma) \right),
\end{equation*}
with \(\mathsf H_0=\mathcal H_2[\mathsf L_0]\), \(\mathsf H_1=\mathcal H_1[\mathsf L_0] +\mathcal
H_2[\mathsf L_1]\) and
\begin{equation*}
  \mathscr R_{1,\Hess} = \mathcal H_2[E] +\mathcal H_1[\mathsf L_1] +\sigma\mathcal H_1[E]
  +\mathcal H_0[\mathsf L_0] +\sigma\mathcal H_0[\mathsf L_1] +\sigma^2\mathcal H_0[E].
\end{equation*}
The displayed formula for \(\mathsf H_1\) is exactly the one stated in the theorem.

It remains to differentiate the score with respect to the scale at fixed \(y\). Since \(x=\pi(y)\)
and \(\nu(y)\) do not depend on \(\sigma\), \(\partial_\sigma a(y,\sigma)=-\sigma^{-1}a(y,\sigma)\).
Differentiating \(\score(y)=\sigma^{-1}\mathsf S_0+ \mathsf S_1+\sigma\mathscr R_{1,\score}\), with
this rule for the composed \(a\)-variable, gives
\begin{align*}
  \Vel(y) &= \sigma^{-2} \bigl[-\mathsf S_0-(\Diff_a\mathsf S_0)[a]\bigr]
  +\sigma^{-1} \bigl[\partial_\sigma\mathsf S_0 -(\Diff_a\mathsf S_1)[a]\bigr] \\
  &\qquad + \partial_\sigma\mathsf S_1 +\mathscr R_{1,\score} -(\Diff_a\mathscr R_{1,\score})[a]
  +\sigma\partial_\sigma\mathscr R_{1,\score}.
\end{align*}
Thus
\begin{equation*}
  \Vel(y) = \sigma^{-2} \left( \dot{\mathsf S}_0(a,x,\sigma) +\sigma\dot{\mathsf S}_1(a,x,\sigma)
  +\sigma^2\mathscr R_{1,\Vel}(a,x,\sigma) \right),
\end{equation*}
where \(\dot{\mathsf S}_0 = -\mathsf S_0 - (\Diff_a\mathsf S_0)[a]\), \(\dot{\mathsf S}_1 =
\partial_\sigma\mathsf S_0-(\Diff_a\mathsf S_1)[a]\) and
\begin{equation*}
  \mathscr R_{1,\Vel} = \partial_\sigma\mathsf S_1 +\mathscr R_{1,\score}
  -(\Diff_a\mathscr R_{1,\score})[a] +\sigma\partial_\sigma\mathscr R_{1,\score}.
\end{equation*}
This proves the asserted formula for the scale derivative.

We finally check the regularity claims. By \Cref{lem:uniform_positivity_of_C0}, \(\mathsf
L_0\in\mathfrak B_r\) and \(\mathsf L_1\in\mathfrak B_{r-1}\), while \Cref{thm:2nd_coef_kernel}
gives \(E\in\mathfrak B_{r-2}\). By \Cref{lem:reconstructed_differential_coefficients_in_Br},
\(\mathbf J_\theta\) and \(\mathbf J_\nu\) belong to \(\mathfrak B_{r-1}\), and the coefficients of
\(\mathbf Q_\theta\) and \(\mathbf Q_\nu\) belong to \(\mathfrak B_{r-2}\). Each differentiation in
\(a\) or \(\theta\) lowers the admissible order by one, and products, sums, and harmless powers of
\(\sigma\) are controlled by \Cref{lem:prop_adm_classes}. Consequently, for the score expansion
\((r_0=3)\),
\begin{equation*}
  \mathsf S_0,\mathsf S_1\in\mathfrak B_{r-2}, \qquad \mathscr R_{1,\score}\in\mathfrak B_{r-3},
\end{equation*}
and, for the Hessian and scale derivative expansions \((r_0=4)\),
\begin{equation*}
  \mathsf H_0,\mathsf H_1,\dot{\mathsf S}_0, \dot{\mathsf S}_1\in\mathfrak B_{r-3},
  \qquad \mathscr R_{1,\Hess},\mathscr R_{1,\Vel} \in\mathfrak B_{r-4}.
\end{equation*}
These are precisely the uniform mixed-derivative bounds stated in the theorem, with derivatives in
\(x\) understood in the stratum coordinate \(\theta\).
\end{proof}